\documentclass[10pt]{article}
\usepackage[nohead,margin=1.0in]{geometry}
\usepackage{amssymb, amsmath, amsthm,amsfonts}
\usepackage{graphicx,epsfig}
\usepackage[tight]{subfigure}
\usepackage{cite}
\usepackage{times}
\usepackage{algpseudocode}
\usepackage{float}
\usepackage{appendix}
\usepackage{color,xcolor}
\usepackage{epstopdf}
\usepackage{mathrsfs}
\usepackage{algorithm}
\usepackage{booktabs}
\usepackage{verbatim}
\usepackage[colorlinks,linktocpage,linkcolor=blue]{hyperref}

\graphicspath{{./pics/}}
\usepackage{booktabs}
\usepackage{threeparttable}

\numberwithin{equation}{section}
\newtheorem{theorem}{Theorem}[section]
\newtheorem{remark}{Remark}[section]
\newtheorem{proposition}{Proposition}[section]
\newtheorem{lemma}{Lemma}[section]

\newtheorem{assumption}[theorem]{Assumption}
\newtheorem{example}{Example}[section]
\allowdisplaybreaks

\definecolor{gray}{RGB}{168,168,168}

\def\d{\mathrm{d} }
\def\II{(\Omega)}

\def\P01{P_\mathcal{A}}

\title{Numerical Analysis of Space-Time Dependent Source Identification in Subdiffusion Equations\thanks{The work of S. C. and Y. K. is supported by the French National Research Agency ANR and Hong Kong RGC Joint Research Scheme for the project IdiAnoDiff (grant ANR-24-CE40-7039). The work of B. J. is supported by Hong Kong RGC General Research Fund (Project
14306423 and Project 14306824), ANR / Hong Kong RGC Joint Research Scheme (A-CUHK402/24) and a start-up fund from The Chinese University of Hong Kong. The work of Z. Z. is supported by by National Natural
Science Foundation of China (Project 12422117 and Project 12426312) and Hong Kong Research
Grants Council (15302323), and an internal grant of Hong Kong Polytechnic
University (Project ID: P0053938, Work Programme: 4-ZZVA).}}
\author{Siyu Cen\thanks{Univ Rouen Normandie, CNRS, Normandie Univ, LMRS UMR 6085, F-76000 Rouen, France  (\texttt{siyu.cen@univ-rouen.fr}, \texttt{yavar.kian@univ-rouen.fr})} \and Bangti Jin\thanks{Department of Mathematics, The Chinese University of Hong Kong, Shatin, New Territories, Hong Kong, P.R. China (\texttt{b.jin@cuhk.edu.hk, bangti.jin@gmail.com})}\and
Yavar Kian\footnotemark[2] \and Zhi Zhou\thanks{Department of Applied Mathematics, The Hong Kong Polytechnic University, Kowloon, Hong Kong, China. (\texttt{zhizhou@polyu.edu.hk})}
}
\date{\today}

\begin{document}

\maketitle
\begin{abstract}
In this work, we investigate the numerical reconstruction of a space-time dependent source in a subdiffusion model from lateral boundary measurements. The proposed scheme combines the Galerkin finite element method for spatial discretization, backward Euler convolution quadrature for temporal discretization, and an easy-to-implement fixed-point iterative algorithm. We prove the linear convergence of the fixed-point iteration and derive an error estimate with an explicit dependence on the discretization parameters and the noise level. The space-time dependence of the source and the use of partial boundary data pose substantial challenges in the numerical analysis of the proposed method. The key tools to address these challenges include suitable stability estimates for the continuous inverse problem, and new refined error estimates, optimal with respect to the problem data, for the associated direct problem. Numerical experiments are provided to corroborate and complement the theoretical results.
\end{abstract}
\renewcommand{\d}{\mathrm{d}}
\def\II{(\Omega)}
\def\tu{\tilde{u}}

\allowdisplaybreaks[1]

\section{Introduction}\label{sec:intro}
This work is concerned with the numerical recovery of a space-time varying source in a cylindrical domain from the lateral boundary measurement. Let $\Omega=\omega\times(-\ell,\ell)\subset \mathbb{R}^d$ ($d=2,3$) and $\omega\subset \mathbb{R}^{d-1}$ be a bounded convex domain.  Consider the following initial-boundary problem for the subdiffusion model:
\begin{equation}\label{eqn:forward_eq}
 \left\{\begin{aligned}
     \partial_t^\alpha u(t,x',x_d) -\Delta u(t,x',x_d)  &=\mathcal{F}(t,x',x_d), && (t,x',x_d)=(t,x)\in   (0,T]\times\omega\times(-\ell,\ell),\\
    \partial_d u(t,x',\pm\ell)&=0,&&  (t,x')\in   (0,T]\times\omega,\\
    u(t,x',x_d)&=0,&& (t,x',x_d)\in   (0,T]\times\partial\omega\times(-\ell,\ell),\\
    u(0,x)&=0,&& x\in \Omega,
  \end{aligned}\right.
 \end{equation}
where   $\partial^{\alpha}_t u $ denotes the Caputo fractional derivative of order $\alpha\in(0,1)$ in time, defined by \cite{KilbasSrivastavaTrujillo:2006, Jin:2021book}
\begin{equation}
    \partial^{\alpha}_t u(t) = \frac{1}{\Gamma (1 - \alpha)} \int_0^t (t-s)^{-\alpha} u'(s){\rm d}s,
\end{equation}
where $\Gamma(z) = \int_0^{\infty} s^{z-1}e^{-s}\d s$ for $\Re (z)>0$ denotes Euler's Gamma function. The fractional derivative $\partial_t^\alpha u$ reduces to the classical first-order derivative $u'(t)$ as $\alpha\to 1^{-}$, and then the model \eqref{eqn:forward_eq} recovers the standard diffusion model.

The model \eqref{eqn:forward_eq} has received much attention in recent years, due to its excellent ability to describe anomalous diffusion. At the microscopic level, these transport processes are described by continuous-time random walk, with the waiting time between consecutive jumps following a heavy-tailed distribution. The probability density function of a particle's position $x$ at time $t$ satisfies the model \eqref{eqn:forward_eq}. The model has been successfully employed in various fields, including thermal diffusion in media with fractal geometry \cite{Nigmatulin:1986}, dispersion in heterogeneous aquifer \cite{AdamsGelhar:1992} and ion dispersion in column experiments \cite{HatanoHatano:1998}; see the   reviews \cite{MetzlerKlafter:2000,MetzlerJeon:2014} for the motivation and diverse applications in physics and biology. 

The inverse source problem aims to  determine the source term $\mathcal{F}$ from  boundary measurement. Due to the natural obstruction for the full recovery of a general source $\mathcal{F}$ (see, e.g. \cite[Section 1.2]{KianYamamoto:2019} and \cite[Section 1.3.1]{KianEric:2022}), we impose additional structural assumptions on the source $\mathcal{F}$, and focus on separable sources $\mathcal{F}$ of the form:
\begin{equation}\label{eqn:F=fR}
    \mathcal{F}(t,x',x_d)=f(t,x')R(t,x',x_d),\quad (t,x',x_d)\in   (0,T]\times\omega\times(-\ell,\ell),
\end{equation}
where $R$ is   known and  $f$ is unknown. Then the inverse source problem 
\textbf{(ISP)} is to recover the space-time dependent source component $f(t,x')$ from the measurement of the solution $u$ to problem \eqref{eqn:forward_eq} on the sub-boundary   $ (0,T)\times\omega\times\{\ell\} $.
Physically, \textbf{(ISP)} is motivated by source identification in anomalous diffusion processes arising in heterogeneous porous media, layered thermal materials, and biological tissues, where only partial boundary measurements are available and the source profile in one spatial direction is known \textsl{a priori}.

Inverse source problems for subdiffusion models have attracted considerable attention in the community (see, e.g., \cite{JinRundell:2015,LiuLiYamamoto:2019isp} for surveys). The majority of existing studies is devoted to the unique recovery of sources depending only on either the temporal or spatial variable \cite{ChenZhangZou:2022, JiangLi:2017,LiZhang:2020,RundellZhang:2018,Wei:2016}, and the space-time dependent sources \cite{RundellZhang:2020,JK,JannoKian:2023,KianLiuYamamoto:2023}. Also there are a few works on the stability of recovering subdiffusion sources. Early contributions were restricted to sources depending only on the time variable \cite{FujishiroKian:2016,SakamotoYamamoto:2011}, and space-dependent sources \cite{KST}. In the context of cylindrical domains, the stability analysis of \textbf{(ISP)} was first developed in \cite{KianYamamoto:2019}.  
The analysis relies on reformulating the inverse problem as an integral equation and proving its well-posedness via a contraction mapping argument.   
Then the work \cite{JinKianZhou:2021} extended the result of \cite{KianYamamoto:2019} to the model with time-dependent diffusion coefficients via a perturbation argument \cite{JinLiZhou:2019}.

Despite the extensive literature on the uniqueness and stability of \textbf{(ISP)}, relatively few works have investigated the reconstruction algorithm and error analysis. Existing studies mostly employ variational regularization, which minimizes the discrepancy between the state and the measurement, with proper regularization \cite{EnglHankeNeubauer:1996,ItoJin:2015}.  In practice,  one must discretize the continuous formulation, which inevitably incurs discretization errors. It is essential to derive error estimates for discrete reconstructions. This motivates the design of numerical algorithms with rigorous error estimates. To the best of our knowledge, the numerical reconstruction of a space-time varying source component $f(t,x')$ from boundary data with rigorous error analysis has not been addressed. The space-time dependence and boundary measurement require far more technical estimates for both the stability result and the error analysis.

In this work, we  develop a discrete reconstruction scheme for \textbf{(ISP)} and establish an error estimate. In view of the stability analysis for \textbf{(ISP)}   in \cite[Theorem 1]{KianYamamoto:2019}, we define an operator $K:L^2((0,T)\times\omega)\to L^2((0,T)\times\omega)$, 
\begin{equation*}
     Kf(t,x^\prime)= \frac{ \partial_t^\alpha u -\Delta' u  }{R}(t,x^\prime)-\frac{ \partial_{dd}u(f)}{R}(t,x^\prime),
\end{equation*}
where $\Delta'$ is the Laplacian defined on $\omega\subset\mathbb{R}^{d-1}$. The first term involves differentiating the data and the second term involves $x_d$-directional derivatives of the solution to \eqref{eqn:forward_eq}. Given the exact data $u(f^\dagger)|_{(0,T)\times\omega\times\{\ell\}}$, the exact source $f^\dagger$ is the unique fixed point of the map $K$. We prove that the operator $K$ is  contractive in a suitably defined weighted Bochner space, which directly leads to the stability result of \textbf{(ISP)} in Theorem \ref{thm:stab}. Moreover, the contraction property yields an efficient, easy-to-implement algorithm. We employ the Galerkin finite element method (FEM) and backward Euler convolution quadrature  to discretize the operator $K$. We prove the linear convergence of the iterative sequence in the finite dimensional space and establish an \textsl{a priori} error estimate for the reconstruction $f^*$. Specifically, the following error estimate holds (see Theorem \ref{thm:error})
\begin{equation}\label{eqn:err-est-main}
    \| ((f^*)^n-(f^\dagger)^n)_{n=1}^{N} \|_{ \ell^2(L^2(\omega))}\le  c(\tau^{-\alpha}\delta+\tau+ h^{-2}\delta+h^{\min(q,\frac{1}{2})}  ),
\end{equation}
where $h$, $\tau$ and $\delta$ denote the mesh size, time step size, and noise level, respectively, and $q\in(0,1]$ is the data regularity index given in Assumption \ref{assum:noise}. The error estimate \eqref{eqn:err-est-main} provides guidelines to choose the parameters $h$ and $\tau$. Our analysis relies heavily on nonstandard error estimates for the direct problem, which employ  the enhanced regularity in the $x_d$-direction, as well as sharp error bounds for the discrete fractional derivative and spatial Laplacian of low regularity noisy data. By integrating these technical tools with the stability estimate, we establish a rigorous error analysis for \textbf{(ISP)}.
Notably, the argument works for both normal diffusion  ($\alpha=1$) and subdiffusion ($\alpha\in(0,1)$).
The  analysis is supported by extensive numerical experiments in Section \ref{sec:num}.  Throughout, for the ease of the exposition, we have adopted the simplifying assumptions that the initial condition $u_0$ is zero, the elliptic operator is the negative Laplacian, and the source term $f$ satisfies a suitable compatibility condition in  Assumption \ref{assum:noise}\rm{(i)}. However, the arguments extend directly to more general settings, in which the solution may exhibit a weak singularity near $t=0^+$; see Remarks \ref{rmk:dis_noise} and \ref{rmk:error} for further discussions.

The rest of the paper is organized as follows.  In Section \ref{sec:stab}, we review the stability result for \textbf{(ISP)}. In Section \ref{sec:recon}, we propose a discrete reconstruction scheme and analyze the reconstruction error. 
Numerical experiments are presented in Section \ref{sec:num} to validate the theoretical results. 
Throughout, for any real$ s \ge 0$ and $p\ge 1$, we denote the standard Sobolev spaces of order $s$ by $W^{s,p}(\Omega)$, equipped with the norm $\|\cdot\|_{W^{s,p}(\Omega)}$, and write  $H^s(\Omega)=W^{s,2}(\Omega)$. For any set $D$, the notation $(\cdot,\cdot)_D$ denotes the $L^2(D)$ inner product, with $(\cdot,\cdot)_\Omega$ abbreviated to $ (\cdot,\cdot)$.
 For any Banach space $X$, $W^{s,p}(0,T;X)$ denotes  the Bochner space.  We denote by $c$ a generic  constant that may change at each occurrence, but is always independent of  $h$, $\tau$ and $\delta$.
 
\section{Stability estimate}\label{sec:stab} 
In this section, we establish a  stability estimate for \textbf{(ISP)}. This estimate is not only of theoretical interest; it motivates the numerical algorithm and its discretization, and forms the basis for the error analysis. We refer interested readers to the survey \cite{CJQZ:2025} and the references therein for further discussion of using (conditional) stability in the numerical analysis of PDE parameter identifications.

\subsection{Preliminaries}\label{subsec:prelim}
First, we discuss the well-posedness of  problem \eqref{eqn:forward_eq}. 
We associate the operator $-\Delta$ with a zero Dirichlet boundary condition with its domain
$D(-\Delta):=\{v\in H_0^1(\Omega) \,:\,  -\Delta v\in L^2(\Omega) \}$.
Let $-\widetilde{\Delta}$ be the operator with mixed boundary condition
$D(-\widetilde{\Delta}):=\{v\in H^1(\Omega) \,:\,  -\Delta v\in L^2(\Omega),\, \partial_d v|_{\omega\times\{ \pm\ell\} }=0,\, v|_{\partial\omega\times(-\ell,\ell)  }=0 \}$.
Let $\{\lambda_\ell\}_{\ell=1}^\infty$ and $\{\varphi_\ell\}_{\ell=1}^\infty$ (respectively  $\{\widetilde{\lambda}_\ell\}_{\ell=1}^\infty$ and $\{\widetilde{\varphi}_\ell\}_{\ell=1}^\infty$) be eigenvalues (ordered nondecreasingly with multiplicity
counted) and the $L^2(\Omega)$-orthonormal eigenfunctions of $-\Delta$ (respectively $-\widetilde{\Delta}$). For $\beta\ge 0$, we define the fractional powers:
\begin{equation*}
    (-\Delta)^\beta v=\sum_{\ell=1}^{\infty} \lambda_\ell^\beta (v,\varphi_\ell) \varphi_\ell \quad\text{and}\quad (-\widetilde{\Delta})^\beta v=\sum_{\ell=1}^{\infty} \widetilde{\lambda}_\ell^\beta (v,\widetilde{\varphi}_\ell) \widetilde{\varphi}_\ell,
\end{equation*}
with the domains $D((-\Delta)^\beta)=\{v\in L^2(\Omega): (-\Delta)^\beta v\in L^2(\Omega)\}$ and $D((-\widetilde{\Delta})^\beta )=\{v\in L^2(\Omega): (-\widetilde{\Delta})^\beta v\in L^2(\Omega)\}$. Since $\omega\subset\mathbb{R}^{d-1}$ is a convex domain,   $D(-\Delta)$ and $D(-\widetilde{\Delta})$ embed continuously into $H^{2}(\Omega)$. Thus, the domains $D((-\Delta)^\beta )$ and $D((-\widetilde{\Delta})^\beta )$  embed continuously into $H^{2\beta}(\Omega)$, for $\beta\in [0,1]$. The following resolvent estimate will be used extensively 
\begin{equation} \label{eqn:resol}
  \| (z -\Delta)^{-1} \|_{L^2\II \rightarrow L^2\II}+\| (z -\widetilde{\Delta})^{-1} \|_{L^2\II \rightarrow L^2\II}\le c_\theta |z|^{-1},  \quad \forall z \in \Sigma_{\theta},
  \,\,\,\theta\in(0,\pi),
\end{equation}
with $\Sigma_{\theta}:=\{0\neq z\in\mathbb{C}: {\rm arg}(z)\leq\theta\}$.

The solution operators $E(t)$ and $\widetilde{E}(t)$ are defined respectively by \cite[Section 6.2.1]{Jin:2021book} 
\begin{align*}
    E(t):=\frac{1}{2\pi {\rm i}}\int_{\Gamma_{\theta,\sigma}}e^{zt}  (z^\alpha -\Delta)^{-1}\, \d z \quad\mbox{and}\quad
    \widetilde{E}(t):=\frac{1}{2\pi {\rm i}}\int_{\Gamma_{\theta,\sigma}}e^{zt}  (z^\alpha -\widetilde{\Delta})^{-1}\, \d z , 
\end{align*}
with the contour $\Gamma_{\theta,\sigma}\subset \mathbb{C}$
(oriented counterclockwise) defined by
\begin{equation*}
  \Gamma_{\theta,\sigma}=\left\{z\in \mathbb{C}: |z|=\sigma, |\arg z|\le \theta\right\}\cup
  \{z\in \mathbb{C}: z=\rho e^{\pm\mathrm{i}\theta}, \rho\ge \sigma\} .
\end{equation*}
Throughout, fix $\theta \in(\frac{\pi}{2},\pi)$ so that $z^{\alpha} \in \Sigma_{\alpha\theta}
\subset \Sigma_{\theta}:=\{0\neq z\in\mathbb{C}: {\rm arg}(z)\leq\theta\},$ for all $z\in\Sigma_{\theta}$. Then the solution $u$ to problem \eqref{eqn:forward_eq} admits the following representation \cite[Section 6.2.1]{Jin:2021book}
\begin{equation}\label{eqn:solrep-u}
    u(t) =  \int_0^t \widetilde{E}(t-s) (f(s)R(s))\,\d s.
\end{equation}  

The next lemma gives smoothing properties of the solution operators $E(t)$ and $\widetilde{E}(t)$.
\begin{lemma}\label{lem:sol-op}
For all $t>0$ and  $\beta\in[0,1]$, there exists  $c>0$  independent of $t$ and $\beta$ such that 
\begin{equation*}
    t^{1-(1-\beta)\alpha}\|(-\Delta)^\beta E(t)v\|_{L^2(\Omega)}+ t^{1-(1-\beta)\alpha}\|(-\widetilde{\Delta})^\beta \widetilde{E}(t)v\|_{L^2(\Omega)} \le c \|v\|_{L^2(\Omega)}.
\end{equation*}   
\end{lemma}

\begin{proof}
The estimates in the case $(- \Delta )$, with $\beta=0,1$, can be found in \cite[Theorem 6.4]{Jin:2021book}. For $\beta\in (0,1)$, the result follows by interpolation. The  case $(- \widetilde{\Delta} )$ follows similarly from the resolvent estimate \eqref{eqn:resol}. 
\end{proof}

Then we have the following well-posedness result for \eqref{eqn:forward_eq} \cite[Theorem 1.2, Lemma 2.1]{KianYamamoto:2019}. 
\begin{theorem}\label{thm:sol-reg}
Let  $R,\partial_d R \in L^\infty(0,T;L^\infty(\Omega))$ and $f\in L^2(0,T;L^2(\omega))$. Then there exists a unique solution $u$ to problem \eqref{eqn:forward_eq} satisfying $u\in H^3(-\ell,\ell; L^2( (0,T)\times \omega ))$ and $\partial_t^\alpha u,\mathcal{A}u\in H^1(-\ell,\ell; L^2( (0,T)\times \omega ))$.   
\end{theorem}

\subsection{Stability estimate}\label{subsec:stab}
Now we present a stability result for \textbf{(ISP)}, which forms the basis for the numerical analysis in Section~\ref{sec:recon}. A related stability estimate was proved in \cite[Theorem~1.2]{KianYamamoto:2019}. We provide a substantially different proof based on a contraction argument in the weighted Bochner space
$L^2_\lambda(0,T;L^2(\Omega)):=L^2\!\bigl((0,T),e^{-\lambda t}\,\d t;L^2(\Omega)\bigr)$.
The analysis not only motivates the proposed iterative solver and numerical scheme, but also informs the convergence analysis. Note that the  map is not necessarily a contraction in the space $L^2(0,T;L^2(\Omega))$, and thus the proof strategy of \cite[Theorem~1.2]{KianYamamoto:2019} cannot be directly adapted to the convergence and error analysis in Section~\ref{sec:recon}.

\begin{assumption}\label{assum:reg} 
  $f  \in  L^{2}((0,T)\times\omega)$ and $R,\,\partial_d R\in L^\infty((0,T)\times\Omega)$.   Further, there exists $c_R>0$ such that $ |R(t,x',\ell)| \ge c_R$, for all $(t,x')\in (0,T)\times\omega$.
\end{assumption}

The non-vanishing condition in Assumption \ref{assum:reg} is crucial for deriving the stability estimate. By Theorem \ref{thm:sol-reg},  the solution $ u\in H^3(-\ell,\ell; L^2((0,T)\times \omega))$. Thus the solution $u$ admits higher regularity in the $x_d$-direction. This motivates the $x_d$-directional derivative $w:=\partial_d u$, which satisfies
\begin{equation}\label{eqn:dudx}
        \left\{\begin{aligned}
            \partial_t^\alpha w -\Delta w  &=f\partial_d R , &&  \mbox{in }   (0,T]\times\Omega,\\
            w(t,x)&=0,&&   \mbox{on }   (0,T]\times\partial \Omega,\\ 
            w(0,x)&=0,&&  \mbox{on } \Omega.
        \end{aligned}\right.
\end{equation}

The next lemma provides pointwise-in-time regularity estimates of $u$ and $w$.
\begin{lemma}\label{lem:dudx}
    Let  Assumption \ref{assum:reg} hold, $u$ be the solution to \eqref{eqn:forward_eq} and $w:=\partial_d u$. Then there exists $c>0$ depending only on $R$ and $T$ such that  
    \begin{equation*}
        \|u(t)\|_{H^1(\Omega)}+\|w(t)\|_{H^1(\Omega)}\le c\int_0^t (t-s)^{\frac{\alpha}{2}-1} \|f(s)\|_{L^2(\Omega)}   \d s.
    \end{equation*}  
\end{lemma}
\begin{proof}
    The estimate is direct from Lemma \ref{lem:sol-op}; see also \cite[Lemma 2.1]{KianYamamoto:2019}.
\end{proof}

Next, we establish the stability for \textbf{(ISP)}. For $\lambda>0$ and $p\in [1,\infty]$, we define a weighted norm in $L^2(0,T)$ by 
\begin{equation}\label{eqn:weighted_norm}
\|v\|_{L^2_\lambda(0,T)} =
\Big(\int_0^T |e^{-\lambda t } v(t)|^2 \d t\Big)^{\frac{1}{2}}.
\end{equation}
The  $L_\lambda^2(0,T)$ norm is equivalent to the standard $L^2(0,T)$ norm. For a Banach space $X$, we denote by $\|v\|_{L^2_\lambda(0,T;X)} $ the weighted norm for the vector-valued space $L^2_\lambda(0,T;X)$.
\begin{theorem}\label{thm:stab} 
Let  $R$ satisfy Assumption \ref{assum:reg} and fix $f_i\in L^2(0,T;L^2(\omega)),\,i=1,2$. Let $u_i$ be the solution to problem \eqref{eqn:forward_eq} with $f=f_i$. Then there exists  $c=c(\alpha,\gamma,c_R,R,T)>0$ such that  
    \begin{equation*}
        \|f_1-f_2\|_{L^2(0,T;L^2(\omega))}\le c\|u_1-u_2\|_{L^2(0,T;H^2(\omega))}+ c\|\partial_t^\alpha (u_1-u_2)\|_{L^2(0,T;L^2(\omega))}.
    \end{equation*} 
\end{theorem}
\begin{proof}
   First, we reformulate problem \eqref{eqn:forward_eq} as 
\begin{equation}\label{eqn:f-repre}
f_i=\frac{\partial_t^\alpha u_i-\Delta u_i}{R}=\frac{\partial_t^\alpha u_i-\Delta' u_i-\partial_dw_i}{R},\quad i=1,2,
\end{equation}
where $\Delta'$ denotes the Laplacian defined on  $\omega\subset\mathbb{R}^{d-1}$ and $w_i=\partial_d u_i$ solves \eqref{eqn:dudx}.
Let $f=f_1-f_2$, $u=u_1-u_2$ and $w=w_1-w_2$. Restricting the identity \eqref{eqn:f-repre} to the boundary $(0,T)\times\omega\times\{\ell\}$ and taking difference of the equation for $i=1,2$ lead to  
    \begin{equation*}
        \|f(t)\|_{L^2(\omega)}\le c\|\partial_t^\alpha u(t)\|_{L^2(\omega)}+ c\| u(t)\|_{H^2(\omega)}+ c\|\partial_d w(t)\|_{L^2(\omega)}
    \end{equation*}
Meanwhile, for $\gamma\in (\frac{3}{4},1)$, the trace inequality \cite[Theorem 1.5.1.7]{Grisvard:1985}, and Lemmas \ref{lem:sol-op} and \ref{lem:dudx} imply  
    \begin{align*}
        \|\partial_d w(t)\|_{L^2(\omega)}\le& \int_0^t \| \partial_d E(t-s)( fR ) (s)\|_{L^2(\omega)} \d s \\
        \le& c\int_0^t \| E(t-s) \|_{L^2(\Omega)\to H^{2\gamma}(\Omega)} \|f(s)\|_{L^2(\omega)}\d s\\
        \le & c\int_0^t (t-s)^{(1-\gamma)\alpha-1} \|f(s)\|_{L^2(\omega)} \d s.
    \end{align*}
    Consequently,
    \begin{align*}
        \|f(t)\|_{L^2(\omega)}\le  c\|  u(t) \|_{H^2(\omega)}+c\|  \partial_t^\alpha u(t) \|_{L^2(\omega)}   + c \int_0^t (t-s)^{(1-\gamma)\alpha-1} \|f(s)\|_{L^2(\omega)}  \d s.
    \end{align*}
    Taking the weighted norm $\|\cdot\|_{L^2_\lambda(0,T;L^2(\omega))}$ on both sides gives
    \begin{align*}
        \|f\|_{L^2_\lambda(0,T;L^2(\omega))}^2\le& c\|u\|_{L^2_\lambda(0,T;H^2(\omega))}^2+ c\|\partial_t^\alpha u\|_{L^2_\lambda(0,T;L^2(\omega))}^2\\
        &+c\int_0^T e^{-2\lambda t}\Big(\int_0^t (t-s)^{(1-\gamma)\alpha-1} \|f(s)\|_{L^2(\omega)}  \d s\Big)^2 \d t.
    \end{align*}
    Now Young's inequality for convolution leads to 
    \begin{align*}
        &\int_0^T e^{-2\lambda t}\left(\int_0^t (t-s)^{(1-\gamma)\alpha-1} \|f(s)\|_{L^2(\omega)}  \d s\right)^2 \d t\\
        =&\int_0^T \left(  \int_0^t e^{- \lambda (t-s)}(t-s)^{(1-\gamma)\alpha-1}  e^{-\lambda s}\|f(s)\|_{L^2(\omega)}  \d s\right)^2 \d t\\
        \le & c \left(\int_0^T e^{-\lambda t} t^{(1-\gamma)\alpha-1} \d t\right)^2\int_0^T |e^{-\lambda t}\|f(t)\|_{L^2(\omega)} |^2\d t\\
        \le & c \lambda^{-2(1-\gamma)\alpha}\|f\|_{L^2_\lambda(0,T;L^2(\omega))}^2,
    \end{align*}
   using the inequality
    \begin{equation}\label{eqn:int_ett}
        \int_0^T e^{-\lambda t} t^{(1-\gamma)\alpha-1} \d t= \lambda^{-(1-\gamma)\alpha}\int_0^{\lambda T} e^{-y} y^{(1-\gamma)\alpha-1} \d y\le  \lambda^{-(1-\gamma)\alpha} \Gamma((1-\gamma)\alpha).
    \end{equation}
    Hence, there exists $c>0$ independent of $\lambda$ such that 
\begin{equation*}
\|f\|_{L^2_\lambda(0,T;L^2(\omega))}
 \le c\|u\|_{L^2_\lambda(0,T;H^2(\omega))}+ c\|\partial_t^\alpha u\|_{L^2_\lambda(0,T;L^2(\omega))}  +c\lambda^{-(1-\gamma)\alpha}\|f\|_{L^2_\lambda(0,T;L^2(\omega))}.
\end{equation*} 
    By taking a sufficiently  large $\lambda$ such that $c\lambda^{-(1-\gamma)\alpha}\le \frac{1}{2}$, we arrive at 
    \begin{equation*}
        \|f\|_{L^2_\lambda(0,T;L^2(\omega))}\le c\|u\|_{L^2_\lambda(0,T;H^2(\omega))}+ c\|\partial_t^\alpha u\|_{L^2_\lambda(0,T;L^2(\omega))}.
\end{equation*}
This and the equivalence of the norms $\|\cdot\|_{L^2(0,T;L^2(\Omega))}$ and $\|\cdot\|_{L_\lambda^2(0,T;L^2(\Omega))}$ give the desired estimate.
\end{proof}

Theorem \ref{thm:stab} motivates an iterative scheme for recovering the space-time dependent source:
 \begin{equation}\label{eqn:recon}
        f^{k+1}(t,x^\prime)= \frac{ \partial_t^\alpha u -\Delta' u  }{R}(t,x^\prime) - \frac{ \partial_{d}w(f^k)}{R}(t,x^\prime),
    \end{equation}
where  $w(f^k)=\partial_d u(f^k)$ satisfies \eqref{eqn:dudx}. The first term of \eqref{eqn:recon} can be computed explicitly from the measured data $z^\delta$, and the second term requires solving  equation \eqref{eqn:dudx} iteratively. 
Next, we prove a convergence rate for the iteration scheme \eqref{eqn:recon}  in the $L^2_\lambda(0,T;L^2(\omega))$ norm. Throughout, we denote by $f^\dagger$ the exact source and by $u^\dagger=u(f^\dagger)$ the exact state.
\begin{proposition}\label{prop:recon}
Let  $f^\dagger$ and $R$ satisfy Assumption \ref{assum:reg}. Fix $f^0\in L^2(0,T;L^2(\omega))$. Then, for any sufficiently large $\lambda$, the sequence $\{f^k\}_{k\ge0}$ generated by the scheme \eqref{eqn:recon} converges to $f^\dagger$ in $L_\lambda^2(0,T;L^2(\omega))$ and there exists $c>0$ depending only on $\alpha$, $\gamma$, $c_R$, $R$ and $T$ such that
    \begin{equation*}
        \|f^k-f^\dagger\|_{L_\lambda^2(0,T;L^2(\omega))}\le \big(c\lambda^{-(1-\gamma)\alpha}\big)^k  \|f^0-f^\dagger\|_{L_\lambda^2(0,T;L^2(\omega))}.
    \end{equation*}
\end{proposition}
\begin{proof}
    We define an operator $K\colon  L^2(0,T;L^2(\omega))\to  L^2(0,T;L^2(\omega))$ by 
    \begin{equation*}
        Kf(t,x^\prime)= \frac{ \partial_t^\alpha u -\Delta' u  }{R}(t,x^\prime)-\frac{ \partial_{d}w(f)}{R}(t,x^\prime).
    \end{equation*}
By Assumption \ref{assum:reg} and Theorem \ref{thm:sol-reg}, the operator $K$ is well-defined. We claim that $K$ is a contraction map in the norm $\|\cdot\|_{L^2_\lambda(0,T;L^2(\omega)) }$ for any large $\lambda$. Indeed, for $f_1,f_2\in  L^2(0,T;L^2(\omega))$,  the argument in Theorem \ref{thm:stab} implies 
    \begin{align*}
        \|(Kf_1-Kf_2)(t)\|_{L^2(\omega)}\le c \int_0^t (t-s)^{(1-\gamma)\alpha-1}\|(f_1-f_2)(s)\|_{L^2(\omega)}\d s.
    \end{align*}
    By multiplying the weight function $e^{-\lambda t}$ and integrating over the interval $(0,T)$, we arrive at
    \begin{align*}
        \|Kf_1-Kf_2\|_{L^2_\lambda(0,T;L^2(\omega))}^2
        \le & c\int_0^T e^{-2\lambda t}\Big(\int_0^t (t-s)^{(1-\gamma)\alpha-1} \|(f_1-f_2)(s)\|_{L^2(\omega)}  \d s\Big)^2 \d t\\
        \le &  c \lambda^{-2(1-\gamma)\alpha}\|f_1-f_2\|_{L^2_\lambda(0,T;L^2(\omega))}^2,
    \end{align*}
where $c>0$ is independent of $\lambda$. Clearly, $f^\dagger$ is a fixed point of $K$ in $\mathcal{Q}$. By taking $f_1=f^\dagger$ and $f_2=f^k$, we obtain
    \begin{equation*}
        \|f^\dagger-f^{k+1}\|_{L^2_\lambda(0,T;L^2(\omega))}\le c \lambda^{-(1-\gamma)\alpha}\|f^\dagger-f^{k}\|_{L^2_\lambda(0,T;L^2(\omega))}.
    \end{equation*}
    Now taking $\lambda$ sufficiently large such that $c\lambda^{-(1-\gamma)\alpha}<1$ gives the desired result.    
\end{proof}
\begin{remark}\label{rmk:varied_conductivity}
Theorem \ref{thm:stab} and  Proposition \ref{prop:recon} rely on the cylindrical  geometry of the domain $\Omega$ and the relation $\Delta=\Delta'+\partial_{dd}$. The argument can be extended to elliptic operators of the form $\mathcal{A}u(x)=-\sum_{i,j=1}^d \partial_i(a_{ij}(x)\partial_j u(x))+q(x)u(x)$, with mild additional structural assumptions on  $a_{ij}$ \cite{KianYamamoto:2019,JinKianZhou:2021}. Since the proofs of Theorem \ref{thm:stab} and Proposition \ref{prop:recon} mainly exploit the smoothing property of the solution operator $E(t)$, the same argument also extends readily to the case $\alpha=1$. 
\end{remark}

\section{Numerical scheme and error estimate}\label{sec:recon}
In this section, we propose a practical iterative algorithm and provide an error analysis of the scheme. 

\subsection{Numerical scheme for the forward problem}\label{subsec:forward_err}
First, we describe a fully discrete scheme for problem \eqref{eqn:forward_eq}. For the time discretization,  we divide the interval $[0, T]$ into $N$ uniform subintervals with a time step size $\tau = T/N$ and let $\{t_n=n\tau\}_{n=0}^{N}$ be the time grids. We employ convolution quadrature generated by the backward Euler scheme (BECQ) \cite[Chapter 3]{JinZhou:2023book} to approximate the fractional derivative $\partial_t^\alpha u(t_n)$ (with $u^n:=u(t_n)$):
\begin{equation}\label{eqn:BECQ}
    \bar{\partial}_\tau^\alpha u^n:=\tau^{-\alpha}\sum_{j=0}^{n} \omega_j^{(\alpha)}(u^{n-j}-u^0), \quad \text{with }(1-\xi)^\alpha=\sum_{j=0}^{\infty}\omega_j^{(\alpha)}.
\end{equation}
For the space discretization, we employ the Galerkin finite element method (FEM) \cite[Chapter 2]{JinZhou:2023book}. Let $\mathcal{T}_h$ be a quasi-uniform simplicial triangulation of the domain $\Omega$ with a mesh size $h$. Over the triangulation $\mathcal{T}_h$, we define a continuous piecewise linear finite element space $V_h$ by 
$V_h=\{\varphi_h\in H^1(\Omega): \varphi_h|_{T} \text{ is a linear function }\forall T\in \mathcal{T}_h\}$.
Let $\widetilde{V}_h^0=\{\varphi_h\in V_h: \varphi_h=0 \text{ on } \partial\omega\times(-\ell,\ell) \}$ and $V_{h}^{0}=\{\varphi_h\in V_h: \varphi_h=0 \text{ on } \partial\Omega  \}$.
The following inverse inequality holds on the space $V_h$ \cite[Lemma 4.5.3]{Brenner:2008}: for $0\le  s_2\le s_1 \le 1$ and $1\le p_1,p_2\le \infty$, we have
\begin{align}\label{eqn:inverse_ineq}
	\|\varphi_h\|_{W^{s_1,p_1}\II}\le ch^{s_2-s_1+\frac{d}{p_1}-\frac{d}{p_2}}\| \varphi_h\|_{W^{s_2,p_2}\II},\quad \forall\varphi_h\in V_h.
\end{align}
For all $T\in \mathcal{T}_h$, the following trace inequalities hold \cite[Lemmas 1.46 and 1.49]{Pietro:2012book}
\begin{align}\label{eqn:dis_trace_ineq}
\|\varphi_h\|_{L^2(\partial T)}&\le c h^{-\frac{1}{2}}\|\varphi_h\|_{L^2(T)},\quad \forall \varphi_h\in V_h,\\
\label{eqn:cts_trace_ineq}
\|v\|_{L^2(\partial T)} & \le   c(h^{\frac12}\|\nabla v\|_{L^2(T)}+ h^{-\frac12}\|v\|_{L^2(  T)})  ,\quad \forall v\in H^1(T),
\end{align}
with $c$ independent of $h$ and $T$. Let $\mathcal{I}_h\colon C(\overline{\Omega})\rightarrow V_h$ be the Lagrange nodal interpolation operator. Then for $s=1,2$ and $1\le p\le \infty $ (with $sp>d$ if $p>1$ and $sp\ge d $ if $p=1$)  \cite[Corollary 4.4.20]{Brenner:2008}: 
\begin{align}\label{eqn:error_I_h}
	\| v-\mathcal{I}_h v \|_{L^p\II}+h\|\nabla(v-\mathcal{I}_h v)\|_{L^p\II}\le ch^s\| v \|_{W^{s,p}\II},\quad\forall v\in W^{s,p}\II.
\end{align}
We define the $L^2(\Omega)$-projection $P_h\colon L^2\II\rightarrow  V_h$ by
\begin{align*}
	(P_h v,\varphi_h)=( v,\varphi_h),\quad\forall\varphi_h\in V_h.
\end{align*}
Then for every $s\in [1,2]$,  we have  \cite[p. 32]{Thomee:2006}
\begin{align}\label{eqn:error_P_h}
	\| v-P_h v\|_{L^2\II}+h\|\nabla(v-P_h v)\|_{L^2\II}\le ch^s\| v\|_{H^s\II},\quad\forall v \in H^s\II.
\end{align}

Next, we define the discrete Laplacians $-\Delta_h\colon V_h^{0}\to V_h^{0}$ and $-\widetilde{\Delta}_h\colon \widetilde{V}_h^{0}\to \widetilde{V}_h^{0}$ by 
\begin{align*}
    (-\Delta_h v_h,\varphi_h)&= (\nabla v_h, \nabla \varphi_h),\quad \forall v_h,\varphi_h\in V_h^{0},\\
    (-\widetilde{\Delta}_h v_h,\varphi_h)&=-(\partial_n v_h,\varphi_h)_{ \omega\times\{\pm\ell\} }+(\nabla v_h, \nabla \varphi_h),\quad \forall v_h,\varphi_h\in \widetilde{V}_h^{0},
\end{align*} 
where $n$ denotes the unit outward normal vector to the boundary $\partial\Omega$ and $(\cdot,\cdot)_{ \omega\times\{\pm\ell\} }$  denotes the $L^2( \omega\times\{\pm\ell\} )$ inner product. 

Let $R^n=R(t_n,\cdot)$. Then a fully discrete scheme for  problem \eqref{eqn:forward_eq} reads: given $u_h^0=0$, find $u_h^n\in \widetilde{V}_h^0 $ for $n=1,\dots,N$ such that 
\begin{equation}\label{eqn:fully-dis-u}
    (\bar{\partial}_\tau^\alpha u_h^n,\varphi_h)+ (  \nabla  u_h^n,\nabla \varphi_h) - (\partial_n u_h,\varphi_h)_{ \omega\times\{\pm\ell\} }=(f^n R^n ,\varphi_h),\quad\forall \varphi_h\in \widetilde{V}_h^0.
\end{equation} 
Using the $L^2(\Omega)$-projection $P_h$ and the discrete  operator $-\widetilde{\Delta}_h$,  the scheme \eqref{eqn:fully-dis-u} can be written as 
\begin{equation*}
    \bar{\partial}_\tau^\alpha u_h^n -\widetilde{\Delta}_hu_h^n =P_h(f^n R^n).
\end{equation*}
The fully discrete scheme for \eqref{eqn:dudx} reads: given $w_h^0=0$, find $w_h^n\in V_h^{0}$ for $n=1,\dots,N$ such that 
\begin{equation}\label{eqn:fully-dis-w}
    \bar{\partial}_\tau^\alpha w_h^n -\Delta_hw_h^n = P_h(f^n\partial_d R^n).
\end{equation} 
Using the discrete Laplace transform, $u_h^n$ and $w_h^n$ can be represented by  \cite[Chapter 3]{JinZhou:2023book}:
\begin{equation}\label{eqn:dis-solrep}
    u_h^n=\tau \sum _{j=1}^{n} \widetilde{E}_{h,\tau}^{n-j}P_h (f^j R^j) \quad\text{and}\quad w_h^n=\tau \sum _{j=1}^{n} E_{h,\tau}^{n-j}P_h (f^j \partial_d R^j),
\end{equation}
where the discrete solution operators $E_{h,\tau}^n$ and $\widetilde{E}_{h,\tau}^n$ are respectively defined by 
\begin{equation*} 
    E_{h,\tau}^n=\frac{1}{2\pi \mathrm{i}}\int_{\Gamma_{\theta,\sigma}^\tau}e^{z t_n}\big( \delta_\tau( e^{-z \tau})^\alpha -\Delta_h \big)^{-1}\d z\quad \text{and}\quad \widetilde{E}_{h,\tau}^n=\frac{1}{2\pi \mathrm{i}}\int_{\Gamma_{\theta,\sigma}^\tau}e^{z t_n}\big( \delta_\tau( e^{-z \tau})^\alpha -\widetilde{\Delta}_h \big)^{-1}\d z,
\end{equation*}
with the kernel function $\delta_\tau(\xi)=(1-\xi)/\tau$ and the contour $\Gamma_{\theta,\sigma}^\tau:=\{z\in \Gamma_{\theta,\sigma} : |\Im (z)|\le \pi /\tau \} $ with $\theta\in (\frac{\pi}{2},\pi)$ close to $\frac{\pi}{2}$ (oriented counterclockwise). Note that for any $\theta\in (\frac{\pi}{2},\pi)$, there exists $\theta^\prime\in (\frac{\pi}{2},\pi)$ such that for all  $z\in \Gamma_{\theta,\sigma}^\tau$  \cite[Lemma 3.1]{JinZhou:2023book}
\begin{equation}\label{eqn:delta_tau}
    \delta_\tau(e^{-z\tau})\in \Sigma_{\theta^\prime },\quad  c_1|z|\le |\delta_\tau(e^{-z\tau})|\le c_2|z|,\quad |\delta_\tau(e^{-z\tau})^\alpha-z^\alpha|\le c_3 \tau|z|^{1+\alpha},
\end{equation}
with $c_1,c_2,c_3>0$ independent of $\tau$. 
By the coercivity of $(-\Delta_h)$,  \eqref{eqn:delta_tau} and interpolation inequality,  we have for $\beta\in[0,1]$,
\begin{equation} \label{eqn:dis-resol}
  \| (-\Delta_h)^{\beta}(\delta_\tau(e^{-z\tau}) -\Delta_h)^{-1} \|_{L^2\II \rightarrow L^2(\Omega)} \le c |z|^{\beta-1},  \quad \forall z \in  \Gamma_{\theta,\sigma}^\tau,
\end{equation}
where $(-\Delta_h)^\beta$ is defined by spectral decomposition. Now, we describe smoothing properties for the discrete  operators $ E_{h,\tau}^n$ and $\widetilde{E}_{h,\tau}^n$ \cite[Lemma 4.6]{ZhangZhangZhou:2022}.
\begin{lemma}\label{lem:dis-sol-op}
For the discrete solution operators $E_{h,\tau}^n$ and $\widetilde{E}_{h,\tau}^n$, there exists  $c>0$ independent of $h$, $\tau$ and $t_n$ such that
\begin{equation*}
    t_{n+1}^{1-(1-\beta)\alpha}\|(-\Delta_h)^\beta E_{h,\tau}^n \varphi_h\|_{L^2(\Omega)}+t_{n+1}^{1-(1-\beta)\alpha}\|(-\widetilde{\Delta}_h)^\beta \widetilde{E}_{h,\tau}^n \varphi_h\|_{L^2(\Omega)} \le c \|\varphi_h\|_{L^2(\Omega)}.
\end{equation*} 
\end{lemma}

The next lemma gives a  discrete Sobolev embedding.
\begin{lemma}\label{lem:dis_embedding}
For $\gamma\in(\frac{3}{4},1)$, there exists $c>0$ independent of $h$ such that  the element-wise derivative $\partial_d v_h$ of $v_h$ satisfies
\begin{equation*}
    \|\partial_d v_h\|_{L^2(\omega)}\le c\|(-\Delta_h)^{\gamma}v_h\|_{L^2(\Omega)}, \quad \forall v_h\in V_h^0.
\end{equation*}
\end{lemma}
\begin{proof}
Let $g_h:=(-\Delta_h)v_h\in V_h^{0}$ and let $v\in H_0^1(\Omega)$ solve $-\Delta v=g_h$ in $\Omega$. 
By the triangular inequality, we have 
    \begin{equation*}
         \|\partial_d v_h-\partial_d v\|_{L^2(\omega)}\le  \|\nabla v_h-\nabla v\|_{L^2(\omega)}
            \le \|\nabla v_h-\nabla \mathcal{I}_hv\|_{L^2(\omega)}+ \|\nabla \mathcal{I}_h v-\nabla v\|_{L^2(\omega)}.
    \end{equation*}
    The  trace inequality \eqref{eqn:cts_trace_ineq} and the  estimate \eqref{eqn:error_I_h}    imply 
\begin{equation*}
    \|\nabla \mathcal{I}_h v-\nabla v\|_{L^2(\omega)}
    \le  
    c\sum_{T \in \mathcal{T}_h} \big(h^\frac12 |v|_{H^2(T)} + h^{-\frac12} \|\nabla \mathcal{I}_h v-\nabla v\|_{L^2(T)}\big) \le ch^\frac12 \|v\|_{H^2(\Omega)}.
\end{equation*}
    Meanwhile, the trace inequality \eqref{eqn:dis_trace_ineq}, \eqref{eqn:error_I_h} and the energy estimate lead to 
    \begin{align*}
        &\|\nabla v_h-\nabla \mathcal{I}_hv\|_{L^2(\omega)}\le ch^{-\frac{1}{2}}\|\nabla v_h-\nabla \mathcal{I}_hv\|_{L^2(\Omega)}\\
        \le& ch^{-\frac{1}{2}}\|\nabla v-\nabla \mathcal{I}_hv\|_{L^2(\Omega)}+ch^{-\frac{1}{2}}\|\nabla v-\nabla v_h\|_{L^2(\Omega)}\le ch^{\frac{1}{2}}\|v\|_{H^2(\Omega)}.
    \end{align*}
    Thus, by combining the preceding estimates, we arrive at 
\begin{equation}\label{eqn:elliptic_err}
    \|\partial_d v_h-\partial_d v\|_{L^2(\omega)} \le ch^{\frac{1}{2}}\|v\|_{H^2(\Omega)}.
\end{equation}
Since $v=0$ on the boundary $\partial\Omega$, we have   $ \|v\|_{H^2(\Omega)}\le c\|(-\Delta) v\|_{L^2(\Omega)} $ \cite[Lemma 3.1]{Thomee:2006}. By the relation $ (-\Delta)v=g_h=(-\Delta_h)v_h$ and the  inverse inequality \eqref{eqn:inverse_ineq}, we derive 
    \begin{align*}
        \|\partial_d v_h-\partial_d v\|_{L^2(\omega)}\le& ch^{\frac{1}{2}}\|(-\Delta) v\|_{L^2(\Omega)}= ch^{\frac{1}{2}}\|(-\Delta_h)v_h\|_{L^2(\Omega)}\\
        \le& ch^{\frac{1}{2}-2(1-\gamma)}\|(-\Delta_h)^{\gamma}v_h \|_{L^2(\Omega)}\le c \|(-\Delta_h)^{\gamma}v_h \|_{L^2(\Omega)} .
    \end{align*}
    Moreover, the Sobolev embedding, elliptic regularity theory and the norm equivalence imply
    \begin{align*}
        \|\partial_d   v\|_{L^2(\omega)}\le c\| v\|_{H^{2\gamma}(\Omega)}\le c\| g_h\|_{H^{2\gamma-2}(\Omega)}= c\| (-\Delta_h)v_h\|_{H^{2\gamma-2}(\Omega)}\le  c\|(-\Delta_h)^{\gamma}v_h \|_{L^2(\Omega)}.
    \end{align*}
\end{proof}

Now we present an error analysis of the scheme   \eqref{eqn:fully-dis-w}, under the following assumption. 
\begin{assumption}\label{assum:reg-num}
$f^\dagger\in  W^{1,p}(0,T;L^\infty(\omega))$ with $p>1/(1-\gamma)\alpha$, and satisfies   $ \|f^\dagger\|_{W^{1,p}(0,T;L^\infty(\omega))}\le c_f$.  Moreover, we assume that $R,\partial_d R\in    W^{1,\infty}(0,T;L^2(\Omega)) $.
\end{assumption}

Under Assumption \ref{assum:reg-num},  $\mathcal{F}=fR$ satisfies $\mathcal{F},\partial_d \mathcal{F}\in C([0,T];L^2(\Omega))\subset W^{1,p}(0,T;L^2(\Omega))$ and by Young's inequality for convolution, with $c=c(\alpha, \gamma, c_f,T,R)$, we have
\begin{equation*}
    \int_0^t (t-s)^{(1-\gamma)\alpha-1}\|\mathcal{F}'(s)\|_{L^2(\Omega)}\d s+\int_0^t (t-s)^{(1-\gamma)\alpha-1}\|\partial_d\mathcal{F}'(s)\|_{L^2(\Omega)}\d s\le c.
\end{equation*} 
The next lemma provides an error estimate of the   scheme \eqref{eqn:fully-dis-w}. 
\begin{lemma}\label{lem:wnh-wdag}
Fix $f^\dagger$. Let $w^\dagger=\partial_d u(f^\dagger)$ and $w_h^n=w_h^n(f^\dagger)$ solve  \eqref{eqn:dudx} and \eqref{eqn:fully-dis-w}, respectively. Then under Assumption \ref{assum:reg-num}, for $\gamma\in(\frac{3}{4},1)$, there exists $c>0$ independent of $h$, $\tau$ and $t_n$ such that
    \begin{equation*}
        \|\partial_d w_h^n-\partial_d w^\dagger(t_n) \|_{L^2(\omega)}\le c(h^{\frac{1}{2}}+  \tau t_n^{(1-\gamma)\alpha-1}).
    \end{equation*} 

\end{lemma}
\begin{proof}
Let $w_h(t)$ be the spatially semidiscrete solution  satisfying $w_h(0)=0$ and 
    \begin{equation*}
        \partial_t^\alpha w_h(t)-\Delta_h w_h(t) =P_h(\partial_d\mathcal{F}(t)).
    \end{equation*}
Then by the triangle inequality, we split the error $\|\partial_d w_h^n-\partial_d w^\dagger(t_n) \|_{L^2(\omega)}$ into two parts:
    \begin{equation*}
        \|\partial_d w_h^n-\partial_d w^\dagger(t_n) \|_{L^2(\omega)}\le 
        \| \partial_d w^\dagger(t_n)-\partial_d w_h(t_n) \|_{L^2(\omega)}+ \|  \partial_d w_h(t_n)-\partial_d w_h^n \|_{L^2(\omega)}:=\mathrm{I}^n+\mathrm{II}^n.
    \end{equation*}
The rest of the proof is split into two steps.

\smallskip
\noindent\underline{Step 1. Bound on the term ${\rm I}_1^n$.} With   the semidiscrete solution operator $$E_h(t):=\frac{1}{2\pi {\rm i}}\int_{\Gamma_{\theta,\sigma}}e^{zt}  (z^\alpha -\Delta_h)^{-1}\, \d z,$$ $w_h(t)$ can be represented by \cite[Section 6.2]{Jin:2021book}
    \begin{equation*}
        w_h(t)=\int_0^t E_h(t-s) P_h( \partial_d \mathcal{F}(s)) \d s.
    \end{equation*}
Recall that the following two identities hold \cite[Lemma 1.7]{JinZhou:2023book}
\begin{equation*}
     -E(t)=(-\Delta)^{-1} \frac{\d}{\d t}(I-F(t))   \quad \text{and}\quad -E_h(t)=(-\Delta)^{-1} \frac{\d}{\d t}(I-F_h(t)),  
\end{equation*} 
with the operators $F(t)$ and $F_h(t)$ defined respectively by
\begin{equation*}
     F(t):=\frac{1}{2\pi {\rm i}}\int_{\Gamma_{\theta,\sigma }}e^{zt} z^{\alpha-1} (z^\alpha-\Delta)^{-1}\, \d z  \quad \text{and}\quad   F_h(t):=\frac{1}{2\pi {\rm i}}\int_{\Gamma_{\theta,\sigma }}e^{zt} z^{\alpha-1} (z^\alpha-\Delta_h)^{-1}\, \d z.
\end{equation*}
 Then by integration by parts, we have
    \begin{align*}
        &w^\dagger(t_n)-  w_h(t_n)=\int_0^{t_n} E(s)(\partial_d \mathcal{F}(t_n-s))-E_h(s)P_h( \partial_d \mathcal{F}(t_n-s)) \d s\\
        =&(-\Delta_h)^{-1} (I-F_h(t_n))P_h(\partial_d \mathcal{F}(0))- (-\Delta)^{-1} (I-F(t_n))\partial_d \mathcal{F}(0)\\
        &+\int_0^{t_n} (-\Delta_h)^{-1}(I-F_h(s))  P_h(\partial_d \mathcal{F}'(t_n-s) ) - (-\Delta)^{-1}(I-F(s))  \partial_d \mathcal{F}'(t_n-s)\d s\\
        =:&\mathrm{I}_1^n+\mathrm{I}_2^n .
    \end{align*}
Note that the term $\mathrm{I}_1^n$ can be further split into
\begin{align*}
    \mathrm{I}_1^n=& (-\Delta_h)^{-1}  P_h(\partial_d \mathcal{F}(0))- (-\Delta)^{-1}  \partial_d \mathcal{F}(0)\\
      &-\left( (-\Delta_h)^{-1}  F_h(t_n) P_h(\partial_d \mathcal{F}(0))- (-\Delta)^{-1}  F(t_n) \partial_d \mathcal{F}(0)\right)=:\mathrm{I}_{1,1}^n-\mathrm{I}_{1,2}^n.
\end{align*}
Next we bound the terms ${\rm I}_{1,1}^n$ and ${\rm I}_{1,2}^n$ separately. To bound the term ${\rm I}_{1,1}^n$, let $v\in H_0^1(\Omega)\cap H^2(\Omega)$ solve
    $ -\Delta v= \partial_d \mathcal{F}(0)$ in $ \Omega$  and $v_h\in V_h^{0}$ solve
    \begin{equation*}
        ( \nabla v_h,\nabla \varphi_h)= (\partial_d \mathcal{F}(0),\varphi_h),\quad \forall \varphi_h\in V_h^{0}.
    \end{equation*}
Then the error estimate \eqref{eqn:elliptic_err} directly implies  
    \begin{align*}
        \|\partial_d  \mathrm{I}_{1,1}^n\|_{L^2(\omega)}=\|\partial_d v_h-\partial_d v\|_{L^2(\omega)}\le ch^{\frac{1}{2}}\|\partial_d \mathcal{F}(0)\|_{L^2(\Omega)}.
    \end{align*} 
Next by the definitions of   $F(t_n)$ and $F_h(t_n)$, we have 
    \begin{align*} 
    & {\rm I}_{1,2}^n
    =\frac{1}{2\pi \mathrm{i}}\int_{\Gamma_{\theta,\sigma}} e^{zt_n}z^{\alpha-1}\left((z^\alpha-\Delta_h)^{-1}(-\Delta_h)^{-1}P_h(\partial_d\mathcal{F}(0))-(z^\alpha-\Delta)^{-1}(-\Delta)^{-1}\partial_d\mathcal{F}(0) \right)\d z.
\end{align*}
Now using the identity $z^\alpha(z^\alpha-\Delta)^{-1}=I-(-\Delta)(z^\alpha-\Delta)^{-1}$ (and similarly for $\Delta_h$), we deduce
\begin{align*}
  {\rm I}_{1,2}^n  =&\frac{1}{2\pi \mathrm{i}}\int_{\Gamma_{\theta,\sigma}} e^{zt_n}z^{-1}\left((-\Delta_h)^{-1}P_h(\partial_d\mathcal{F}(0))-(-\Delta)^{-1}\partial_d\mathcal{F}(0) \right) \d z\\
    &+ \frac{1}{2\pi \mathrm{i}}\int_{\Gamma_{\theta,\sigma}} e^{zt_n}z^{-1}\left((z^\alpha-\Delta)^{-1}\partial_d\mathcal{F}(0)-(z^\alpha-\Delta)^{-1}P_h(\partial_d\mathcal{F}(0) )\right) \d z.
    \end{align*}
    Similarly, by repeating the preceding argument with the 
 estimate \eqref{eqn:elliptic_err}, we derive
    \begin{align*} 
        \|\partial_d (z^\alpha-\Delta_h)^{-1} P_h(\partial_d\mathcal{F}(0))-\partial_d (z^\alpha-\Delta)^{-1} \partial_d\mathcal{F}(0)\|_{L^2(\omega)}\le ch^{\frac{1}{2}}\|\partial_d \mathcal{F}(0)\|_{L^2(\Omega)},\quad \forall z\in\Gamma_{\theta,\sigma}.
    \end{align*}
Then  we get
    \begin{align*}       
        \|\partial_d { \rm I}_{1,2}^n\|_{L^2(\omega)}\le  ch^{\frac{1}{2}}\|\partial_d\mathcal{F}(0)\|_{L^2(\Omega)} \left| \int_{\Gamma_{\theta,\sigma} } e^{zt_n}z^{-1}\d z \right| \le ch^{\frac{1}{2}} \|\partial_d\mathcal{F}(0)\|_{L^2(\Omega)}.
    \end{align*}
    This proves the estimate of $\mathrm{I}_1^n$.
Next, by Assumption \ref{assum:reg-num}, $\partial_d \mathcal{F}'\in L^p(0,T;L^2(\Omega))$, and repeating the argument for the term $\mathrm{I}_1^n$ gives
\begin{align*}
   \|\partial_d  \mathrm{I}_2^n\|_{L^2(\omega)}
        \le &ch^{\frac{1}{2}}\int_0^{t_n}\| \partial_d \mathcal{F}'(s)\|_{L^2(\Omega)}\d s.
    \end{align*}
    Thus, we can bound the term  $\mathrm{I}^n$ by
   $\| \partial_d \mathrm{I}^n \|_{L^2(\omega)} \le ch^{\frac{1}{2}}$.

\medskip
\noindent\underline{Step 2. Bound on the term $\mathrm{II}^n$.} Since $\partial_d \mathcal{F}\in W^{1,p}(0,T;L^2(\Omega))$, we have 
$\partial_d\mathcal{F}(t)= \partial_d\mathcal{F}(0)+(1\ast \partial_d\mathcal{F}')(t)$.
    Then $w_h(t)$ can be represented by 
    \begin{align*}
        w_h(t)&=\frac{1}{2\pi\mathrm{i}}\int_{\Gamma_{\theta,\sigma}} e^{zt}z^{-1}(z^\alpha-\Delta_h)^{-1} P_h\partial_d\mathcal{F}(0)\d z +\left((E_h\ast 1)\ast \partial_d\mathcal{F}'\right)(t).
    \end{align*}
Meanwhile, let $ \delta_{t_n}$ be the Dirac-delta function at $t_n$ (from the left side) and let $ E_{h,\tau}(t)=\tau \sum_{n=0}^{\infty}E_{h,\tau}^n\delta_{t_n}(t) $. Then $w_h^n$ can be expressed as 
    \begin{align*}
        w_h^n=&\frac{1}{2\pi \mathrm{i}}\int_{\Gamma_{\theta,\sigma}^\tau}e^{z t_n} e^{-z\tau} z^{-1} \big( \delta_\tau( e^{-z \tau})^\alpha -\Delta_h \big)^{-1} P_h(\partial_d\mathcal{F}(0))\d z 
        +  \left((E_{h,\tau}\ast1)  \ast P_h(\partial_d\mathcal{F}') \right)(t_n ).
    \end{align*}
    Then $\mathrm{II}^n=\mathrm{II}^n_1+\mathrm{II}^n_2 $, with 
    \begin{alignat*}{2}
        \mathrm{II}^n_1&:= \frac{1}{2\pi\mathrm{i}}\int_{\Gamma_{\theta,\sigma}} e^{zt_n}z^{-1}(z^\alpha-\Delta_h)^{-1} P_h(\partial_d\mathcal{F}(0))\d z &&  \\
        & \quad-\frac{1}{2\pi \mathrm{i}}\int_{\Gamma_{\theta,\sigma}^\tau}e^{z t_n} e^{-z\tau} z^{-1} \big( \delta_\tau( e^{-z \tau})^\alpha -\Delta_h \big)^{-1} P_h(\partial_d\mathcal{F}(0))\d z,  \\
        \mathrm{II}^n_2 &:=\left((E_h\ast 1)\ast P_h(\partial_d\mathcal{F}')\right)(t_n) -\left((E_{h,\tau}\ast1)  \ast P_h(\partial_d\mathcal{F}') \right)(t_n ) .  
    \end{alignat*} 
    For the term $\mathrm{II}_1^n$, we have 
    \begin{align*}
        \mathrm{II}_1^n=&\frac{1}{2\pi\mathrm{i}}\int_{\Gamma_{\theta,\sigma} \setminus \Gamma_{\theta,\sigma}^\tau } e^{zt_n}z^{-1}(z^\alpha-\Delta_h)^{-1} P_h(\partial_d\mathcal{F}(0))\d z\\
        &+\frac{1}{2\pi \mathrm{i}}\int_{\Gamma_{\theta,\sigma}^\tau}e^{z t_n} \left(z^{-1}(z^\alpha-\Delta_h)^{-1}-e^{-z\tau} z^{-1} \big( \delta_\tau( e^{-z \tau})^\alpha -\Delta_h \big)^{-1}\right) P_h(\partial_d\mathcal{F}(0))\d z.
    \end{align*}
    By the error estimate \eqref{eqn:delta_tau} and the resolvent estimate \eqref{eqn:dis-resol}, we have for $\gamma\in(\frac{3}{4},1)$,
    \begin{equation*}
        \|(-\Delta_h)^{\gamma}(z^{-1}(z^\alpha-\Delta_h)^{-1}-e^{-z\tau} z^{-1} \big( \delta_\tau( e^{-z \tau})^\alpha -\Delta_h \big)^{-1}) \|_{L^2(\Omega)\to L^2(\Omega)}\le c\tau |z|^{(\gamma-1)\alpha}.
    \end{equation*}
    Together with  Lemma \ref{lem:dis_embedding} and the   resolvent estimate \eqref{eqn:dis-resol}, by setting $\sigma=t_n^{-1}$, we  derive
    \begin{align*}
         \| \partial_d \mathrm{II}^n_1\|_{L^2(\omega)}
         \le& c\|P_h(\partial_d\mathcal{F}(0)) \|_{L^2(\Omega)}\int_{c\tau^{-1}}^{\infty} e^{-c\rho t_n} \rho ^{(\gamma-1)\alpha-1} \d \rho \\
         &+ c\tau\|P_h(\partial_d\mathcal{F}(0)) \|_{L^2(\Omega)}\left(\int_{t_n^{-1}}^{\infty} e^{-c\rho t_n} \rho ^{(\gamma-1)\alpha} \d \rho+ t_n ^{(1-\gamma)\alpha-1}   \right)\\
         \le & c\tau t_n ^{(1-\gamma)\alpha-1} \| \partial_d\mathcal{F}(0) \|_{L^2(\Omega)}.
    \end{align*}
    For the term $\mathrm{II}^n_2$, similarly, we have   for $\gamma\in (\frac{3}{4},1)$,
    \begin{equation*}
        \| (-\Delta_h)^{\gamma}((E_h-E_{h,\tau})\ast 1)(t_n)\|_{L^2(\Omega)\to L^2(\Omega)}\le c\tau t_n^{(1-\gamma)\alpha-1}.
    \end{equation*}
    For $t\in (t_{n-1},t_n]$, using Taylor expansion, we have
    \begin{equation*}
        (-\Delta_h)^{\gamma}((E_h-E_{h,\tau})\ast 1)(t)=(-\Delta_h)^{\gamma}((E_h-E_{h,\tau})\ast 1)(t_n)-\int_{t}^{t_n}(-\Delta_h)^{\gamma}(E_h-E_{h,\tau})(s)\d s.
    \end{equation*}
    By the smoothing property in Lemma \ref{lem:dis-sol-op}, we arrive at 
    \begin{align*}
   &\quad \left\|\int_{t}^{t_n}(-\Delta_h)^{\gamma}E_h (s)\d s\right\|_{L^2(\Omega)\to L^2(\Omega)}+\left\|\int_{t}^{t_n}(-\Delta_h)^{\gamma}E_{h,\tau} (s)\d s\right\|_{L^2(\Omega)\to L^2(\Omega)}\\
    &\le  c\int_t^{t_n} s^{(1-\gamma)\alpha-1} \d s\le c\tau t_n^{(1-\gamma)\alpha-1}.
    \end{align*}
    Thus, 
    $$ \| (-\Delta_h)^{\gamma}((E_h-E_{h,\tau})\ast 1)(t)\|_{L^2(\Omega)\to L^2(\Omega)}\le c\tau t^{(1-\gamma)\alpha-1}. $$
   This and Lemma \ref{lem:dis_embedding} yield
    \begin{align*}
         \| \partial_d \mathrm{II}^n_2\|_{L^2(\omega)}\le c\tau \int_0^{t_n} (t_n-s)^{(1-\gamma)\alpha-1}\|\partial_d\mathcal{F}'(s)\|_{L^2(\Omega)}\d s.
    \end{align*}
    This completes the proof of the lemma.
\end{proof}
\begin{remark}\label{rmk:wnh-wdag}
    Lemma \ref{lem:wnh-wdag} provides the error bound  $O(\tau t_n^{(1-\gamma)\alpha-1})$. The factor $t_n^{(1-\gamma)\alpha-1}$ indicates that the error becomes unbounded as $t_n \to 0$. If $\partial_d \mathcal{F}(0)=0$,  the proof indicates
    \begin{equation*}
        \|\partial_d w_h^n-\partial_d w^\dagger(t_n) \|_{L^2(\omega)}\le c(h^{\frac{1}{2}}+  \tau  ).
    \end{equation*} 

\end{remark}

\subsection{Numerical scheme for \textbf{(ISP)}}\label{subsec:inverse_err}
Now we develop a numerical scheme for recovering the space-time varying source $f^\dagger$. Assumption \ref{assum:noise} (i) imposes  regularity conditions on the source $\mathcal{F}$.
\begin{assumption}\label{assum:noise}
The problem data satisfies the following conditions.
    \begin{itemize}
        \item[{\rm(i)}] The domain $\omega$ is $C^3$,  and the source $\mathcal{F}=f^\dagger R$ satisfies $\mathcal{F},\partial_d\mathcal{F}\in C([0,T];H^q(\Omega))\cap W^{2,p}(0,T;L^2(\Omega))$, with $p>\frac1\alpha$ and $q\in (0,1]$. Moreover, $\mathcal{F}|_{t=0}=\partial_d\mathcal{F}|_{t=0}=\partial_t\mathcal{F}|_{t=0}=\partial_t\partial_d\mathcal{F}|_{t=0}=0$. 
        \item[{\rm(ii)}]  $z^\delta\in C([0,T];L^2(\omega))$ with $z^{\delta}(0,x')=0$ and $\|z^\delta-u^\dagger(t,x',\ell)\|_{C([0,T];L^2(\omega))}\le \delta $.  
    \end{itemize} 
\end{assumption}

Then, following the argument of \cite[Lemma 2.4]{GaitanKian:2013}, the solution $u^\dagger$ to \eqref{eqn:forward_eq} satisfies 
\begin{equation}\label{eqn:reg-udag}
    u^\dagger(t,x',\ell)\in C([0,T];H^{2+q}(\omega))\cap C^2([0,T];L^2(\omega)).
\end{equation}

The scheme \eqref{eqn:recon} requires differentiating the data $z^\delta$, which is however ill-defined due to the low regularity of $z^\delta$. Instead we define a discrete fractional derivative $\bar{\partial}_\tau^\alpha z^\delta$ using \eqref{eqn:BECQ}:
\begin{equation}\label{eqn:dtz}
    \bar{\partial}_\tau^\alpha (z^\delta)^n:=\tau^{-\alpha}\sum_{j=0}^{n} \omega_j^{(\alpha)}((z^\delta)^{n-j}-(z^\delta)^0).
\end{equation} 
Let   $\mathcal{T}_h^\partial$ be  the restriction of $\mathcal{T}_h$ on the boundary $\omega\subset\partial\Omega$, and let $V_h^\partial\subset H^1(\omega) $ and $X_h^\partial\subset H^1(\omega)$ be the finite element spaces  with continuous piecewise polynomials of degree 1 and 2, respectively, on $\mathcal{T}_h^\partial$. Then we define the discrete Laplacian $-\Delta_h' z^\delta\in V_h^\partial$ by 
\begin{equation}\label{eqn:dxxz}
    (-\Delta_h' z^\delta, \varphi_h)_{\omega}:=-(\partial_nP_{X_h^\partial} z^\delta,\varphi_h)_{\partial\omega}+(\nabla P_{X_h^\partial} z^\delta,\nabla\varphi_h)_{\omega},\quad \forall\varphi_h\in V_h^\partial,
\end{equation}
where $\partial_n$ denotes the unit outward normal derivative to the boundary $\partial\omega$ and $P_{X_h^\partial}:L^2(\omega)\to X_h^\partial$ is the $L^2(\omega) $ projection onto the space $X_h^\partial$. Note that the data $z^\delta $ is  projected onto the space $X_h^\partial$, but the Laplacian is approximated in the space $V_h^\partial$.
The following lemma provides a crucial estimate.
\begin{lemma}\label{lem:dis_noise}
Let Assumption \ref{assum:noise} hold and  $\bar{\partial}_\tau^ \alpha z^\delta$ and $-\Delta_h' z^\delta $ be defined by \eqref{eqn:dtz} and  \eqref{eqn:dxxz}, respectively.  Then there exists $c>0$ independent of $h$, $\tau$ and $t_n$ such that
    \begin{equation*}
        \|\partial_t^\alpha u^\dagger(t_n)- \bar{\partial}_\tau^\alpha z^\delta(t_n) \|_{ L^2(\omega) }\le c(\tau^{-\alpha}\delta+\tau)\quad \text{and}\quad \|  \Delta' u^\dagger(t_n)-\Delta_h'z^\delta(t_n) \|_{ L^2(\omega)}\le c(h^{-2}\delta+h^q).
    \end{equation*} 
\end{lemma}
\begin{proof}
By the triangle inequality, we have 
    \begin{equation*}
        \| \partial_t^\alpha u^\dagger(t_n)- \bar{\partial}_\tau^\alpha z^\delta(t_n) \|_{ L^2(\omega) }\le \| \partial_t^\alpha u^\dagger(t_n)- \bar{\partial}_\tau^\alpha u^\dagger(t_n) \|_{ L^2(\omega) }+\| \bar{\partial}_\tau^\alpha u^\dagger(t_n)- \bar{\partial}_\tau^\alpha z^\delta(t_n) \|_{ L^2(\omega) }.
    \end{equation*}
By the definition \eqref{eqn:dtz}, we have $\omega_0^{(\alpha)}=1$, $\omega_j^{(\alpha)}<0$ for $j\ge1$, and $\sum_{j=0}^\infty \omega_j^{(\alpha)}=0$. Hence, \begin{align}\label{eqn:dtau_u-z}
        \| \bar{\partial}_\tau^\alpha u^\dagger(t_n)- \bar{\partial}_\tau^\alpha z^\delta(t_n) \|_{ L^2(\omega) }
        \le& \tau^{-\alpha}\sum_{j=0}^n \|\omega_j^{(\alpha)} (u^\dagger(t_{n-j})-z^\delta(t_{n-j}) )\|_{ L^2(\omega) }\notag\\
        \le& \tau^{-\alpha}\|u^\dagger-z^\delta\|_{C([0,T];L^2(\omega))}\sum_{j=0}^n|\omega_j^{(\alpha)}|\le c\tau^{-\alpha}\delta.
    \end{align}
    Meanwhile,  the regularity estimate \eqref{eqn:reg-udag} and \cite[Theorem 3.1]{JinZhou:2023book} imply 
    \begin{equation*}
         \| \partial_t^\alpha u^\dagger(t_n)- \bar{\partial}_\tau^\alpha u^\dagger(t_n) \|_{ L^2(\omega) }\le c\tau.
    \end{equation*}
This estimate and \eqref{eqn:dtau_u-z} imply
    \begin{align*}
        \| \partial_t^\alpha u^\dagger(t_n)- \bar{\partial}_\tau^\alpha z^\delta(t_n) \|_{  L^2(\omega)  } \le c(\tau^{-\alpha}\delta+\tau).
    \end{align*}
    Next, using the $L^2(\omega)$ projection $P_{V_h^\partial}:L^2(\omega)\to V_h^\partial$, we have 
    \begin{equation*}
        \|\Delta' u^\dagger(t_n) -\Delta_h' z^\delta(t_n)  \|_{L^2(\omega)}\le  \|\Delta' u^\dagger(t_n) -P_{V_h^\partial} \Delta' u^\dagger  (t_n)\|_{L^2(\omega)}+ \|P_{V_h^\partial} \Delta' u^\dagger(t_n) -\Delta_h' z^\delta (t_n) \|_{L^2(\omega)}.
    \end{equation*} 
    The regularity estimate \eqref{eqn:reg-udag} and the error estimate \eqref{eqn:error_P_h} imply 
    \begin{equation*}
         \|\Delta' u^\dagger(t_n) -P_{V_h^\partial} \Delta' u^\dagger(t_n)  \|_{L^2(\omega)}\le ch^q.
    \end{equation*}
By the definition \eqref{eqn:dxxz}, the inverse inequality \eqref{eqn:inverse_ineq} and  trace inequality \eqref{eqn:dis_trace_ineq}, we have 
    \begin{align*}
        &\|P_{V_h^\partial} \Delta' u^\dagger(t_n) -\Delta_h' z^\delta (t_n) \|_{L^2(\omega)}= \sup_{\varphi\in L^2(\omega)}\frac{(P_{V_h^\partial} \Delta' u^\dagger(t_n) -\Delta_h' z^\delta (t_n), \varphi)_\omega}{\|\varphi\|_{L^2(\omega)}}\\
        =& \sup_{\varphi\in L^2(\omega)}\frac{( \Delta' u^\dagger(t_n) -\Delta_h' z^\delta(t_n), P_{V_h^\partial}\varphi_h)_\omega}{\|\varphi\|_{L^2(\omega)}}\\
        =&\sup_{\varphi\in L^2(\omega) }\frac{( \nabla P_{X_h^\partial} z^\delta(t_n)- \nabla u^\dagger(t_n) , \nabla P_{V_h^\partial}\varphi)_\omega+( \partial_n u^\dagger(t_n)-\partial_nP_{X_h^\partial} z^\delta(t_n),P_{V_h^\partial}\varphi)_{\partial\omega}}{\|\varphi\|_{L^2(\omega)}}\\
        \le & ch^{-1} \|\nabla P_{X_h^\partial} z^\delta(t_n)- \nabla u^\dagger (t_n)\|_{L^2(\omega)}+ ch^{-\frac{1}{2}} \|\partial_n P_{X_h^\partial} z^\delta(t_n)- \partial_n u^\dagger (t_n)\|_{L^2(\partial\omega)}.
    \end{align*}
    By the inverse inequality \eqref{eqn:inverse_ineq}, the discrete trace inequality \eqref{eqn:dis_trace_ineq} and the approximation property of $P_{X_h^\partial}$, we derive  
    \begin{align*}
        \|\nabla P_{X_h^\partial} z^\delta(t_n)- \nabla u^\dagger (t_n)\|_{L^2(\omega)}
        \le &\|\nabla P_{X_h^\partial} z^\delta(t_n)- \nabla  P_{X_h^\partial} u^\dagger(t_n) \|_{L^2(\omega)}\\
        &+ \|\nabla P_{X_h^\partial} u^\dagger(t_n)- \nabla u^\dagger(t_n) \|_{L^2(\omega)} 
        \le c(h^{-1}\delta+h^{1+q})\end{align*}
    and
         \begin{align*} \|\partial_n P_{X_h^\partial} z^\delta(t_n)- \partial_n u^\dagger(t_n) \|_{L^2(\partial\omega)}
         \le &\|\partial_n P_{X_h^\partial} z^\delta(t_n)- \partial_n  P_{X_h^\partial} u^\dagger(t_n) \|_{L^2(\partial\omega)}\\
        &+ \|\partial_n P_{X_h^\partial} u^\dagger(t_n)- \partial_n u^\dagger(t_n) \|_{L^2(\partial\omega)} \le c(h^{-\frac{3}{2}}\delta+h^{\frac{1}{2}+q}).
    \end{align*}
    Combining the preceding estimates gives the desired estimate on $\|\Delta' u^\dagger(t_n) -\Delta_h' z^\delta(t_n)  \|_{L^2(\omega)}$.
\end{proof}

\begin{remark}\label{rmk:dis_noise}
Lemma \ref{lem:dis_noise} relies on the additional regularity provided by Assumption \ref{assum:noise}. The requirement $u^\dagger(\cdot, x', \ell) \in C^2([0, T]; L^2(\omega))$ is restrictive.
    By the argument  of \cite[Lemma 4.3]{JinShinZhou:2023}, one can still derive the error estimate if
\begin{equation}\label{eqn:weak-1}
u^\dagger(t,x',\ell)\in C^\alpha([0,T];L^2(\omega))\cap  C^2((0,T];L^2(\omega)),
\end{equation}
and there exists $t_0\in(0,T)$ such that 
\begin{equation}\label{eqn:weak-2}
\|\partial_t^\alpha u^\dagger(t,\cdot,\ell)\|_{L^2(\omega)}\le c \quad \text{and}\quad\|\partial_t\partial_t^\alpha u^\dagger(t,\cdot,\ell)\|_{L^2(\omega)}\le ct^{-1}\quad \forall t\in(0,t_0).
\end{equation}
Then the following estimate holds
\begin{equation}\label{eqn:weak-err}
\|(\partial_t^\alpha u^\dagger(t_n)- \bar{\partial}_\tau^\alpha z^\delta(t_n) )_{n=1}^{N}\|_{ \ell^2(L^2(\omega)) }\le c(\tau^{-\alpha}\delta+\tau^{\frac{1}{2}}|\log \tau|).
    \end{equation} 
Moreover, conditions \eqref{eqn:weak-1} and \eqref{eqn:weak-2} hold if the source $\mathcal{F}=f^\dagger R$ satisfies
$$\mathcal{F},\partial_d\mathcal{F}\in C([0,T];H^q(\Omega))\cap W^{2,p}(0,T;L^2(\Omega)),\quad 
\text{with} ~~ p>1/\alpha~~ \text{and} ~~q\in(0,1].$$ Then the compatibility conditions
$\mathcal{F}|_{t=0}=\partial_d\mathcal{F}|_{t=0}=\partial_t\mathcal{F}|_{t=0}=\partial_t\partial_d\mathcal{F}|_{t=0}=0$
are not needed, which is much weaker than Assumption~\ref{assum:noise}(i), at the expense of only $\ell^2(L^2(\Omega))$ bound and a temporal convergence rate $O(\tau^{1/2})$.
\end{remark} 
\vskip5pt

Next, we propose a fully discrete iterative scheme for recovering the source $f^\dagger(t,x^\prime)$, based on the   formula  \eqref{eqn:recon}. For an initial guess $f^0=(f^{0,n})_{n=1}^N $, consider the following iteration scheme, with the update $f^{k+1}=(f^{k+1,n})_{n=1}^N $ from $f^{k}=(f^{k,n})_{n=1}^N$ given by
\begin{equation}\label{eqn:recon_dis}
        f^{k+1}(t,x')=\frac{\bar{\partial}_\tau^\alpha z^\delta-\Delta_h' z^\delta}{R}(t,x')-\frac{\partial_d w_h^n(f^k)}{R}(t,x'),
\end{equation}
where $w_h^n(f^k)$ solves problem \eqref{eqn:fully-dis-w} (with $f\equiv f^k$).  To analyze the convergence of the algorithm, we define the  weighted $\ell^2$ norm. For $\lambda>0$, let
\begin{equation}\label{eqn:dis_weighted_norm}
    \|v\|_{\ell^2_\lambda }=\Big(\tau \sum_{n=1}^{\infty}  |e^{-\lambda t_n } v^n|^2  \Big)^{\frac{1}{2}}.
\end{equation}
For a finite sequence $(v^n)_{n=1}^{N}$, we denote $\|(v^n)_{n=1}^{N}\|_{\ell_\lambda^2}:=\|(v^n)_{n=1}^{\infty}\|_{\ell_\lambda^2}$, by setting $v^n=0$ for $n>N$. Note for a finite sequence, the  $\ell^2_\lambda$ norm is equivalent to the standard $\ell^2$ norm.  Similarly, for a Banach space $X$, we denote by $\|v\|_{\ell^2_\lambda( X)} $  the weighted norm for the vector-valued space $\ell^2_\lambda(X)$.

The next lemma  gives the stability in the $\ell^2_\lambda$ norm for  $u_h^n $ and $w_h^n $.
\begin{lemma}\label{lem:dis_forward_stab}
    Let Assumption \ref{assum:reg} hold. Let $f_1,f_2\in\ell^2(L^2(\omega))$, and  $w_h^n(f_i)$ be the solution to problem \eqref{eqn:fully-dis-w}. Then for  $\gamma\in(\frac{3}{4},1)$,  there exists $c>0$  independent of $h$, $\tau$, $N$ and $\lambda$ such that
    \begin{equation*}
         \|\partial_d w_{1,h}^n-\partial_dw_{2,h}^n \|_{\ell^2_\lambda(L^2(\omega))}\le c   (\tau^{(1-\gamma)\alpha}+ \lambda^{(\gamma-1)\alpha} )   \|   f_1 -  f_2 \|_{\ell^2_\lambda(L^2(\omega))}.
    \end{equation*} 
\end{lemma}
\begin{proof}
    Let $w_{i,h}^n=w_h^n(f_i)$ and $e_h^n=w_{1,h}^n-w_{2,h}^n$. The solution representation \eqref{eqn:dis-solrep} implies   
    \begin{equation*}
        e_h^n=\tau \sum _{j=1}^{n} E_{h,\tau}^{n-j}\big( (f_1^j-f_2^j)\partial_d R^j  \big).
    \end{equation*}
Lemmas \ref{lem:dis-sol-op} and \ref{lem:dis_embedding} and the  upper bound of $ \partial_d R $ in Assumption \ref{assum:reg}  yield
    \begin{align*}
     \|\partial_d e_h^n \|_{L^2(\omega)}
       \le c  \| (-\Delta)^{\gamma}e_h^n \|_{L^2(\Omega)}  
      \le  c \tau \sum_{j=1}^n (t_{n+1}-t_j)^{(1-\gamma)\alpha-1}  \| f_1^j-f_2^j  \|_{L^2(\omega)}.
    \end{align*} 
    Multiplying $e^{-\lambda t_n}$ on both sides, taking the   $\ell^2_\lambda$ norm and applying Young’s inequality lead to 
    \begin{align*}
        \|\partial_d e_h^n \|_{\ell^2(L^2(\omega))}\le &c \Big(\tau \sum_{n=1}^N \big|e^{-\lambda t_n} \tau \sum_{j=1}^n (t_{n+1}-t_j)^{(1-\gamma)\alpha-1} \|   f_1^j -  f_2^j \|_{L^2(\omega)} \big|^2\Big)^{\frac{1}{2}}\\
        \le &c \Big(\tau \sum_{n=1}^N \big| \tau \sum_{j=1}^n  e^{-\lambda (t_n-t_j)} (t_{n+1}-t_j)^{(1-\gamma)\alpha-1}  e^{-\lambda t_j}  \|   f_1^j -  f_2^j \|_{L^2(\omega)}  \big|^2\Big)^{\frac{1}{2}}\\
        \le & c \Big( \tau \sum_{n=1}^N e^{-\lambda t_{n-1}} t_n^{(1-\gamma)\alpha-1}  \Big) \|   f_1^j -  f_2^j \|_{\ell^2_\lambda(L^2(\omega))}.
    \end{align*}
Using the estimate \eqref{eqn:int_ett}, we arrive at 
\begin{align*}
    \tau \sum_{n=1}^N e^{-\lambda t_{n-1}} t_n^{(1-\gamma)\alpha-1} = &\tau^{(1-\gamma)\alpha-1}+ \tau \sum_{n=2}^N e^{-\lambda t_{n-1}} t_n^{(1-\gamma)\alpha-1} \\
    \le &\tau^{(1-\gamma)\alpha-1}+c  \int_0^T e^{-\lambda s}s^{(1-\gamma)\alpha-1}\d s \le c  (\tau^{(1-\gamma)\alpha}+ \lambda^{(\gamma-1)\alpha} ).
\end{align*} 
Therefore, we conclude that
    \begin{equation*}
        \|\partial_d e_h^n\|_{\ell^2(L^2(\omega))}\le    c   (\tau^{(1-\gamma)\alpha}+ \lambda^{(\gamma-1)\alpha} )   \|   f_1^j -  f_2^j \|_{\ell^2_\lambda(L^2(\omega))}  .
    \end{equation*}
\end{proof}

The next theorem gives the convergence of the  scheme \eqref{eqn:recon_dis}.
\begin{theorem}\label{thm:recon_dis}
Let Assumption \ref{assum:reg} hold.  Then for sufficiently large $\lambda>0$, there exists $\tau_0>0$, such that for all $0<\tau \le \tau_0$, given  $f^0=(f^{0,n})_{n=1}^{N}\in \ell^2_\lambda(L^2(\omega))$,  the   scheme \eqref{eqn:recon_dis} converges to a unique limit $f^*=(f^{*,n})_{n=1}^{N}\in  \ell^2_\lambda(L^2(\omega))$, and moreover, there exist $\gamma\in(\frac{3}{4},1)$ and $c>0$ independent of $h$, $\tau$, $N$ and $\lambda$ such that \begin{equation}\label{eqn:recon_dis_conv}
        \|f^k-f^*\|_{\ell^2_\lambda(L^2(\omega))}\le (c\tau^{(1-\gamma)\alpha } + c\lambda^{(\gamma-1)\alpha })^k\| f^0-f^* \|_{\ell^2_\lambda(L^2(\omega))}.
    \end{equation} 
\end{theorem}
\begin{proof}
We define the map $K_{h,\tau}:\ell^2_\lambda(L^2(\omega))\to \ell^2_\lambda(L^2(\omega))$ by 
    \begin{equation*}
        K_{h,\tau}(f)=\frac{\bar{\partial}_\tau^\alpha z^\delta-\Delta_h' z^\delta}{R} +\frac{\partial_d w_h^n(f^k)}{R} .
    \end{equation*}
    For any $f^i=(f^{i,n})_{n=1}^N\in \ell^2_\lambda(L^2(\omega)) $, by   Lemma \ref{lem:dis_forward_stab} and Assumption \ref{assum:reg}, we have
    \begin{align*}
        \|K_{h,\tau}f^1-K_{h,\tau}f^2\|_{\ell^2_\lambda(L^2(\omega))}= &\|R^{-1}\partial_d(w_{h}^{1,n} -w_{h}^{2,n}   )\|_{\ell^2_\lambda(L^2(\omega))}\le  c (\tau^{(1-\gamma)\alpha}+ \lambda^{(\gamma-1)\alpha} ) \| f^{1} -  f^{2}  \|_{\ell^2_\lambda(L^2(\omega))}.  
    \end{align*}
    We choose a sufficiently large $\lambda$ such that $c\lambda^{(\gamma-1)\alpha}\le \frac{1}{3}$ and then choose a sufficiently small $\tau_0$ such that $c\tau_0^{(1-\gamma)\alpha}\le \frac{1}{3}$.
    Thus, $K_{h,\tau}$ is a contraction map on the space $\ell^2_\lambda(L^2(\omega))$ for all $0<\tau\le \tau_0$. The Banach fixed-point theorem implies that the iteration converges to a unique limit $f^*=(f^{*,n})_{n=1}^{N}\in \ell^2_\lambda(L^2(\omega))$ and the estimate \eqref{eqn:recon_dis_conv} holds.  
\end{proof}

Last, we derive an error bound on the recovered source $f^*$.
\begin{theorem}\label{thm:error}
    Let Assumption \ref{assum:reg}, \ref{assum:reg-num} and \ref{assum:noise} hold. Let $f^\dagger$ be the exact source and $f^*$ be the limit of the   scheme \eqref{eqn:recon_dis}. Then  there exists $\tau_0>0$ such that for all $0<\tau\le \tau_0$, there exists $c>0$ independent of $h$, $\tau$, $\delta$ and $N$ such that \begin{equation}\label{eqn:error}
        \| ((f^*)^n-(f^\dagger)^n)_{n=1}^{N} \|_{ \ell^2(L^2(\omega))}\le  c(\tau^{-\alpha}\delta+\tau+ h^{-2}\delta+h^{\min(q,\frac{1}{2})}   ).
    \end{equation}   
\end{theorem}
\begin{proof}
    Note that the exact source $f^\dagger$ satisfies
    \begin{equation*}
        f^\dagger=\frac{  \partial_t^\alpha u^\dagger-\Delta'u^\dagger}{R}- \frac{  \partial_d w^\dagger}{R},
    \end{equation*}
with   $w^\dagger=\partial_d u^\dagger$ satisfying \eqref{eqn:dudx}. Since $f^*$ is the fixed point of the scheme \eqref{eqn:recon_dis}, we obtain
    \begin{align*}
        (f^*)^n-(f^\dagger)^n
        &= \frac{\bar{\partial}_\tau^\alpha (z^\delta)^n-\partial_t^\alpha u^\dagger(t_n)}{R} +\frac{ \Delta'u^\dagger(t_n) - \Delta_h' (z^\delta)^n }{R} +\frac{\partial_d w^\dagger(t_n) - \partial_d w_h^n(f^*)}{R} \\  &:=\mathrm{I}^n+\mathrm{II}^n+\mathrm{III}^n.
    \end{align*}
    By Lemma \ref{lem:dis_noise} and the positive lower bound $|R(t,x',\ell)|\ge c_R>0$ from Assumption \ref{assum:reg}, we have
    \begin{equation*}
        \| \mathrm{I}^n \|_{ L^2(\omega)}
        \le  c(\tau^{-\alpha}\delta+\tau )\quad \text{and}\quad  \| \mathrm{II}^n \|_{ L^2(\omega)}\le  c(h^{-2}\delta+h^q ),
    \end{equation*}
    where $c>0$ is independent of $h$, $\tau$, $t_n$ and $\delta$. For the term $\mathrm{III}^n$, by the triangular inequality, Lemma \ref{lem:wnh-wdag}, Remark \ref{rmk:wnh-wdag} and Lemma \ref{lem:dis_forward_stab}, we have  
    \begin{align*}
    &\quad \| (\mathrm{III}^n)_{n=1}^{N} \|_{\ell^2_\lambda( L^2(\omega))}\\
        & \le c\, \|(\partial_d w_h^n(f^*) -\partial_d w_h^n(f^\dagger) )_{n=1}^{N}\|_{\ell^2_\lambda( L^2(\omega))}+ c\, \|(\partial_d w_h^n(f^\dagger) -\partial_d w^\dagger(t_n) )_{n=1}^{N}\|_{\ell^2_\lambda( L^2(\omega))}\\
       &  \le c\, (\tau^{(1-\gamma)\alpha}+\lambda^{(\gamma-1)\alpha})\| ( (f^*)^n -  (f^\dagger)^n )_{n=1}^N\|_{\ell^2_\lambda(L^2(\omega))}+c\,(h^{\frac{1}{2}}+ \tau ),
    \end{align*} 
    with $\gamma\in (\frac{3}{4},1)$ and  $c>0$   independent of $h$, $\tau$, $t_n$, $ \delta$, $\lambda$ and $\gamma$. {\color{blue}Therefore, we choose $\tau_0$ such that $c\tau_0^{(1-\gamma)\alpha}\le \frac{1}{3}$ and choose $\lambda$ such that $c\lambda^{(\gamma-1)\alpha}<\frac{1}{3}$.}  By combining the estimates of  $\mathrm{I}^n$ and $\mathrm{II}^n$, taking $\ell_\lambda^2$ norm, we get
    \begin{equation*}
        \| ((f^*)^n-(f^\dagger)^n)_{n=1}^{N} \|_{ \ell^2_\lambda(L^2(\omega))}\le  c(\tau^{-\alpha}\delta+\tau+ h^{-2}\delta+h^{\min(q,\frac{1}{2})}   ).
    \end{equation*}
     The desired estimate follows from the equivalence of the $\|\cdot\|_{\ell^2_\lambda(L^2(\omega))} $  and $\|\cdot\|_{\ell^2(L^2(\omega))} $ norms.
\end{proof}
\begin{remark}\label{rmk:error}
One may derive an error bound  under the  weaker regularity  in Remark \ref{rmk:dis_noise}. Indeed, Lemma \ref{lem:dis_noise} provides
\begin{align*}
\|\partial_t^\alpha u^\dagger(t_n)- \bar{\partial}_\tau^\alpha z^\delta(t_n) \|_{ L^2(\omega) }&\le c(\tau^{-\alpha}\delta+\tau^{\frac{1}{2}}|\log\tau|),\\
 \|  \Delta' u^\dagger(t)-\Delta_h'z^\delta(t) \|_{ L^2(\omega)}&\le c(h^{-2}\delta+h^q.
\end{align*}
By Lemma \ref{lem:wnh-wdag} and taking the weighted $\ell_\lambda^2$ norm,
    \begin{align*}
        \|(\partial_d w_h^n(f^\dagger) -\partial_d w^\dagger(t_n) )_{n=1}^{N}\|_{\ell^2_\lambda( L^2(\omega))}
        & \le ch^{\frac{1}{2}}+c\Big(\tau\sum_{n=1}^{\infty}|e^{-\lambda t_n}\tau  t_n^{(1-\gamma)\alpha-1}|^2\Big)^{\frac{1}{2}}  \\ & \le c(h^{\frac{1}{2}}+\tau^{(1-\gamma)\alpha+\frac{1}{2}}).
    \end{align*}
    Hence, we arrive at  
    \begin{equation}\label{eqn:error_weak}
        \| ((f^*)^n-(f^\dagger)^n)_{n=1}^{N} \|_{ \ell^2_\lambda(L^2(\omega))}\le  c(\tau^{-\alpha}\delta+\tau^{\frac{1}{2}}|\log\tau|+ h^{-2}\delta+h^{\min(q,\frac{1}{2})}  ).
    \end{equation}
    The error estimates \eqref{eqn:error} and \eqref{eqn:error_weak} provide a guideline for choosing the discretization parameters $h$ and $\tau$.   In practice, it is essential to choose $h$ and $\tau$ such that the error of  numerical differentiation is properly balanced with the discretization error, i.e. $\tau^{-\alpha}\delta\sim \tau $, $h^{-2}\delta\sim h^{\min(q,\frac{1}{2})} $ in \eqref{eqn:error} and $\tau^{-\alpha}\delta\sim \tau^{\frac{1}{2}} $, $h^{-2}\delta\sim h^{\min(q,\frac{1}{2})} $ in \eqref{eqn:error_weak}. 
\end{remark}
\begin{remark}\label{rmk:error_high_reg}
The error estimate in Theorem \ref{thm:error} can be improved if the problem data has higher regularity: the $O(h^{\frac12})$ term can be refined by bounding the $L^\infty(\Omega)$ norm. If  $\mathcal{F},\partial_d \mathcal{F}\in C([0,T];L^\infty(\Omega))\cap W^{1,p}(0,T;L^\infty(\Omega))$, with $p>\frac{1}{\alpha}$, following \cite[Theorem 6.8]{Thomee:2006}, we can derive 
$$ \|\partial_d w_h^n-\partial_d w^\dagger(t_n)\|_{L^\infty(\omega)}\le c(h+\tau t_n^{(1-\gamma)\alpha-1}),$$ and, consequently, 
    \begin{equation*}
        \| ((f^*)^n-(f^\dagger)^n)_{n=1}^{N} \|_{ \ell^2_\lambda(L^2(\omega))}\le  c(\tau^{-\alpha}\delta+\tau^{\frac{1}{2}}|\log\tau|+ h^{-2}\delta+h^{\min(q,1)}  ).
    \end{equation*}
Moreover, given smooth data, higher-order finite element spaces and convolution quadrature based on higher-order backward differentiation formulas can be employed to achieve higher order convergence rates. However, the factors $\tau^{-\alpha}\delta$ and $h^{-2}\delta$ are inevitable, reflecting the ill-posedness of inherent {\rm\textbf{(ISP)}}. Since the measured data $z^\delta$ is often irregular, the analysis focuses on the case of minimal regularity. The sharpness of these   estimates will be numerically illustrated in Section \ref{sec:num}.
\end{remark}

\section{Numerical results and discussions}\label{sec:num}
Now we present numerical experiments to illustrate the reconstruction algorithm. Fix a square domain $\Omega = (0, 1)^2$ and $T=1$. The noisy data $z^\delta$ is measured on the boundary segment $(0, 1) \times \{1\}$ over the time interval $(0, T)$:
\begin{equation*}
    z^\delta(x_1,t)=u^\dagger(x_1,1,t)+\delta\| u^\dagger(\cdot,1,t)\|_{L^\infty(\omega)}\xi(x_1,t),
\end{equation*}
where $\delta$ denotes the relative noise level and $\xi$ follows the standard Gaussian distribution. The exact state $u^\dagger=u(f^\dagger)$  is computed on  a highly refined mesh  with $h=1/200$ and $\tau =1/1000$. The accuracy of a reconstruction $f^*$
relative to the exact one $f^\dagger$ is measured by the relative $L^2$ error  $e_f={\|f^*-f^\dagger\|_{\ell^2(L^2(\omega))}}/{\|f^\dagger\|_{L^2(0,T;L^2(\omega))}} $. 

The first example is about a smooth source.
\begin{example}\label{ex:smooth1}
     $R(x_1,x_2,t)=x_2$ and $f^\dagger(x_1,t)=\sin(\pi x_1) (1-\cos(4\pi t))t$.
\end{example}

First we verify the following error  bound in Theorem \ref{thm:error} and Remark \ref{rmk:error_high_reg}:
\begin{equation*}
    \| ((f^*)^n-(f^\dagger)^n)_{n=1}^{N} \|_{ \ell^2(L^2(\omega))}\le  c(\tau^{-\alpha}\delta+\tau+ h^{-2}\delta+ {h} ),
\end{equation*}
with $\alpha\in \{0.25,0.50,0.75,1\}$.
For $\delta=0$, Theorem \ref{thm:error} and Remark \ref{rmk:error_high_reg} predict a convergence rate $O(h)$ in space and $O(\tau)$ in time. Fig. \ref{Fig:conv_smooth1} (a)-(b) indicate an $O(h)$ empirical convergence rate in space and $O(\tau)$ in time, which aligns well with the theoretical prediction. 
To examine the convergence rate with respect to $\delta$, we first choose the parameters $h=\frac{1}{6}$ and $\tau=\frac{1}{20}$ for $\delta=\text{1e-2}$ and then refine the parameters as $h^{-2}\delta\sim h$ and $\tau^{-\alpha}\delta\sim\tau$  as the noise level $\delta $ decreases. Fig. \ref{Fig:conv_smooth1} (c) indicates an empirical rate $O(\delta^{0.33})$, which  aligns well with the theoretical one $O(\delta^{\frac{1}{3}})$.

Fig. \ref{Fig:recon_smooth1} presents the recovered source $f^*$ at three   noise levels $\delta\in\{\text{1e-2, 1e-3, 1e-4}\}$ and $\alpha=0.75$. 
Fig. \ref{Fig:err_evolution} (a) shows the variation of the errors $\|f^k-f^\dagger\|_{\ell^2(L^2(\omega))}/ \| f^\dagger\|_{\ell^2(L^2(\omega))}$ and $ 
\|f^k-f^\dagger\|_{\ell_\lambda^2(L^2(\omega))}/ \| f^\dagger\|_{\ell_\lambda^2(L^2(\omega))}$ (with  $\lambda=10$),
with respect to the iteration number $k$ at $\delta=\text{1e-4}$, initialized to $f^0=0$. The iteration is efficient: the reconstruction errors decay rapidly and stabilize within 10 iterations. The evolution profiles for the standard and weighted errors are nearly identical, indicating that the weighted norm $\|\cdot\|_{\ell^2_\lambda(L^2(\omega))}$ serves primarily as a theoretical tool.
\begin{figure}[htbp]
	\centering
	\begin{tabular}{ccc}
		\includegraphics[width=0.30\textwidth]{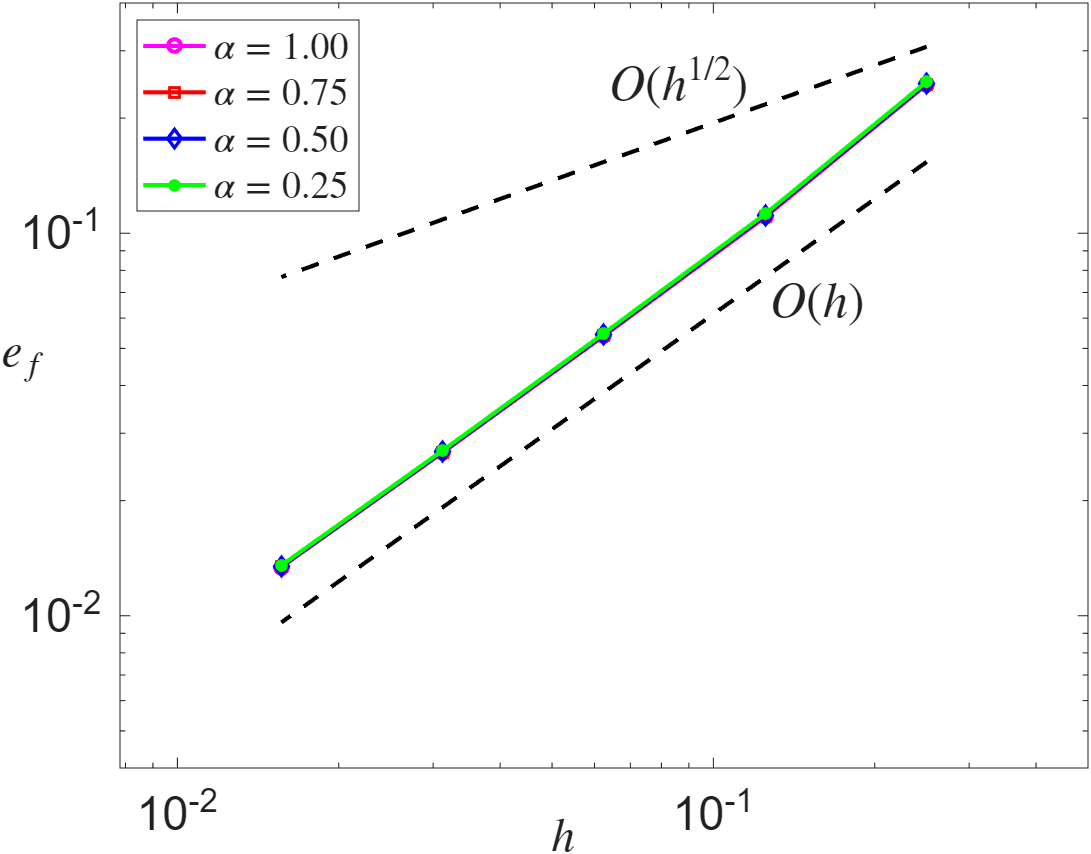}&
		\includegraphics[width=0.30\textwidth]{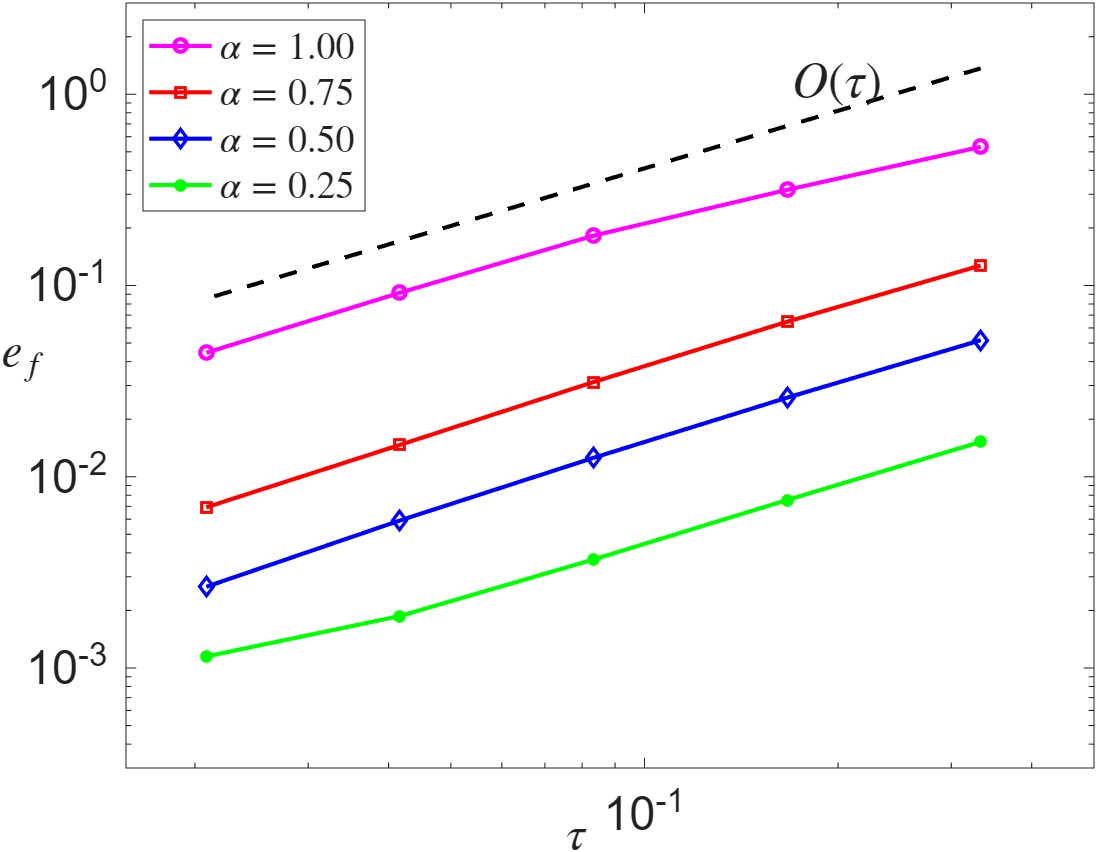}&
		\includegraphics[width=0.30\textwidth]{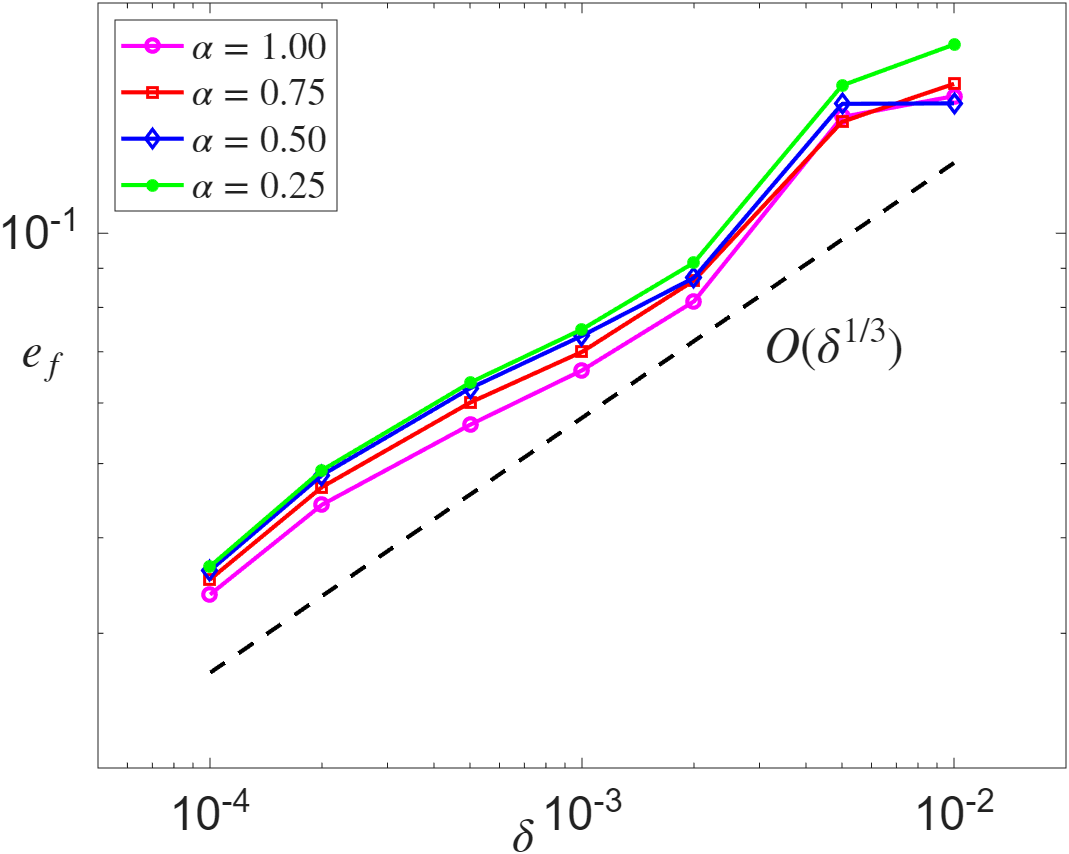}\\
		(a) $e_f$ versus $h$ & (b) $e_f$ versus $\tau$  & (c) $e_f$ versus $\delta$  
	\end{tabular}
	\caption{The convergence of $e_f$ with respect to $h$, $\tau$ and $\delta$ for Example \ref{ex:smooth1}.}
	\label{Fig:conv_smooth1}
\end{figure}
\begin{figure}[htbp]
	\centering
	\begin{tabular}{cccc}
		\includegraphics[width=0.22\textwidth]{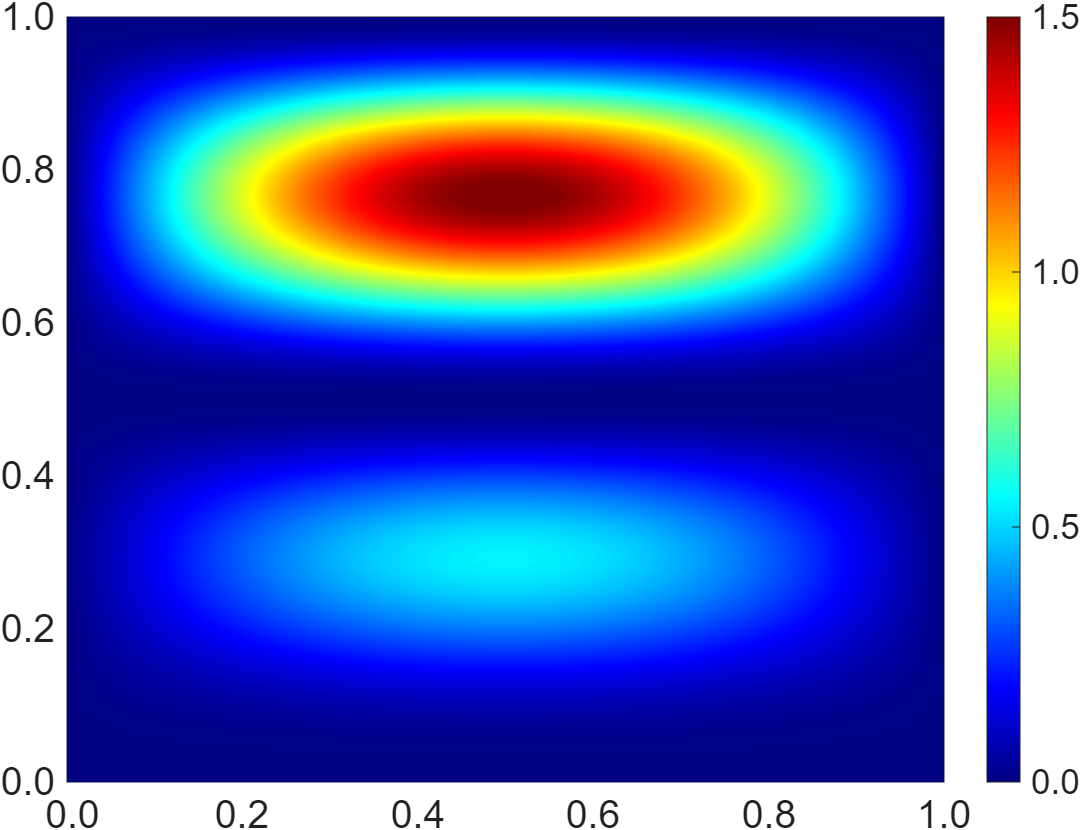}&
		\includegraphics[width=0.22\textwidth]{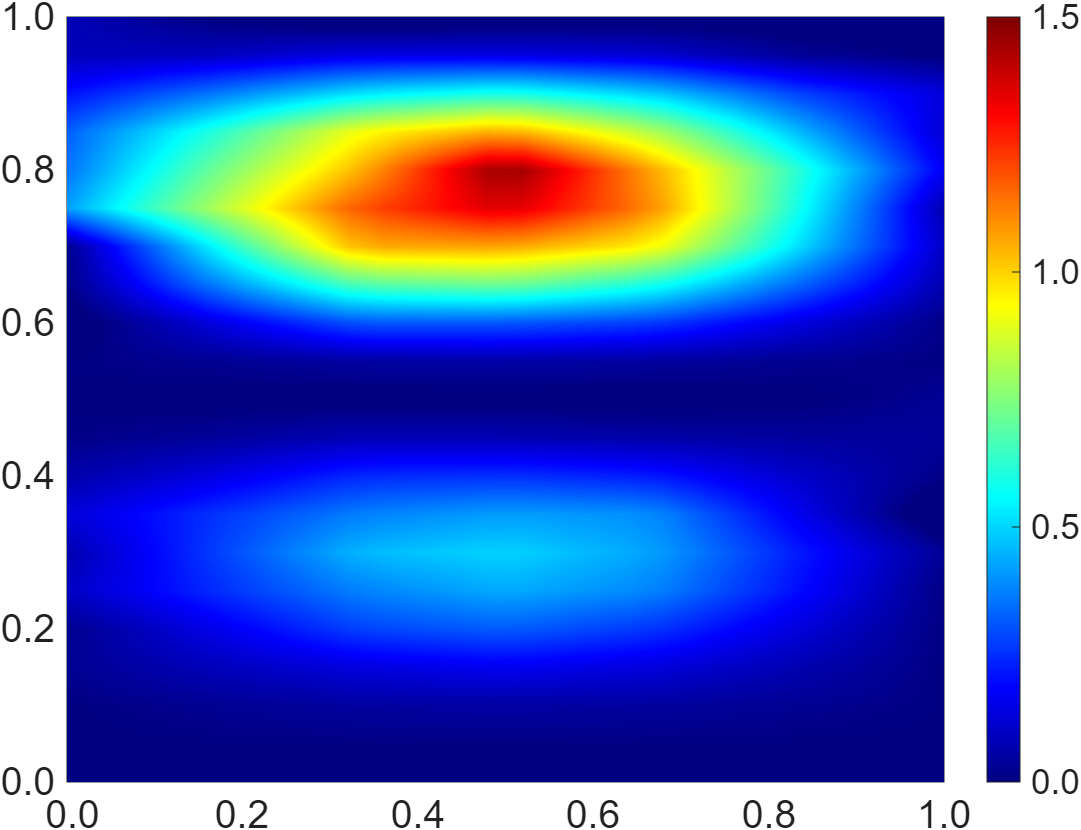}&
		\includegraphics[width=0.22\textwidth]{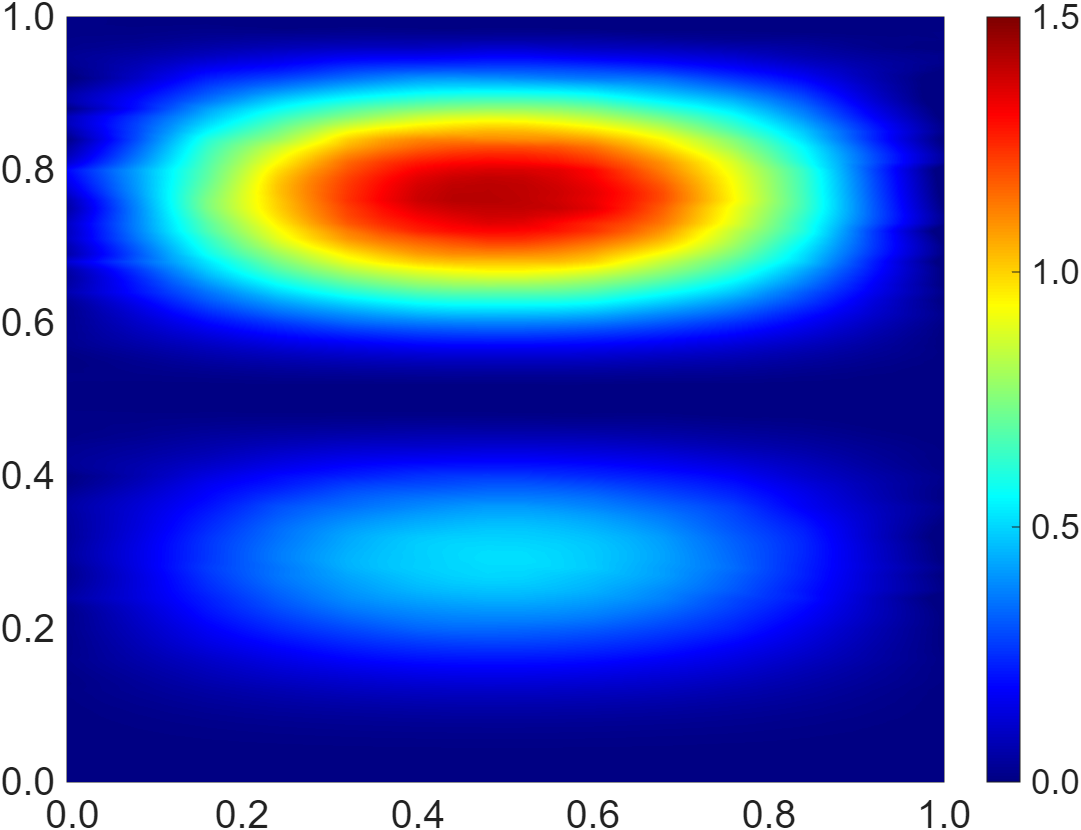}&
		\includegraphics[width=0.22\textwidth]{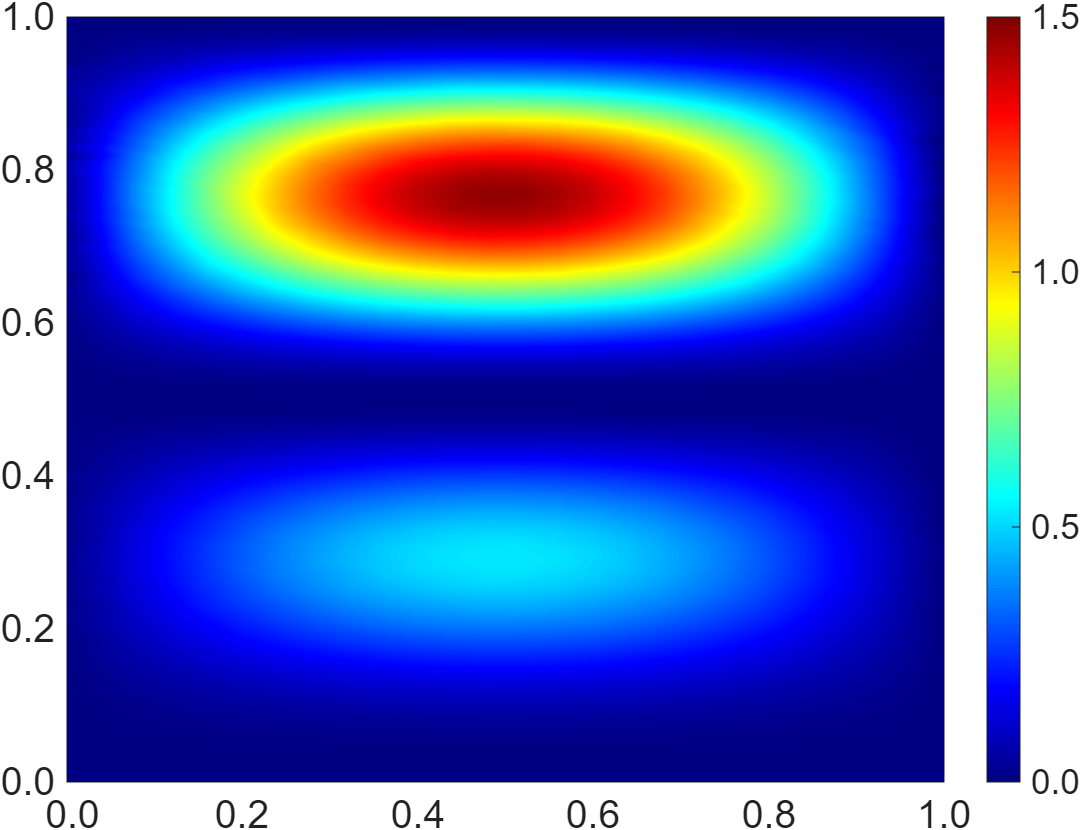}\\
		(a) exact & (b) $\delta=\text{1e-2}$  & (c) $\delta=\text{1e-3}$ & (d) $\delta=\text{1e-4}$  
	\end{tabular}
	\caption{The reconstruction $f^*$  for Example \ref{ex:smooth1}  with $\alpha=0.75$ at different noise levels.}
	\label{Fig:recon_smooth1}
\end{figure}
\begin{figure}[htbp]
	\centering
	\begin{tabular}{cccc} 
		\includegraphics[width=0.22\textwidth]{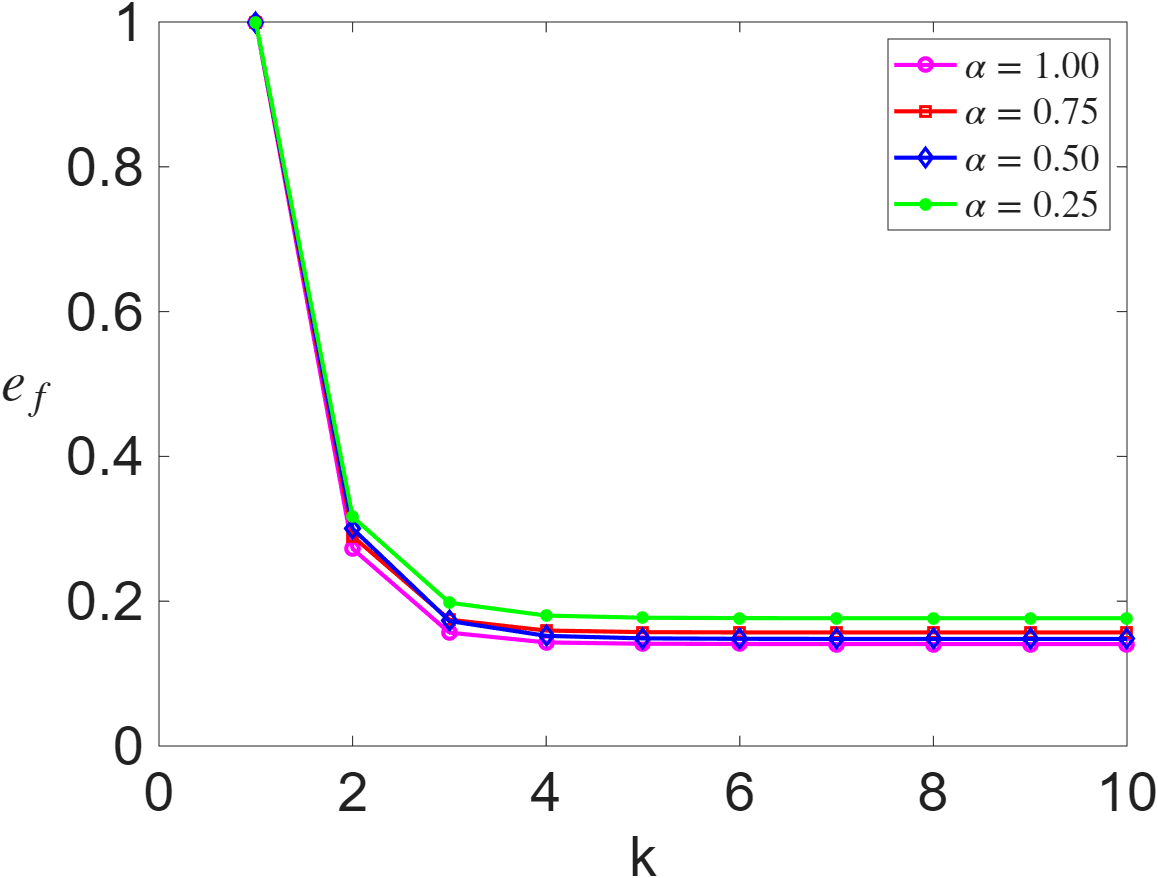}&
        \includegraphics[width=0.22\textwidth]{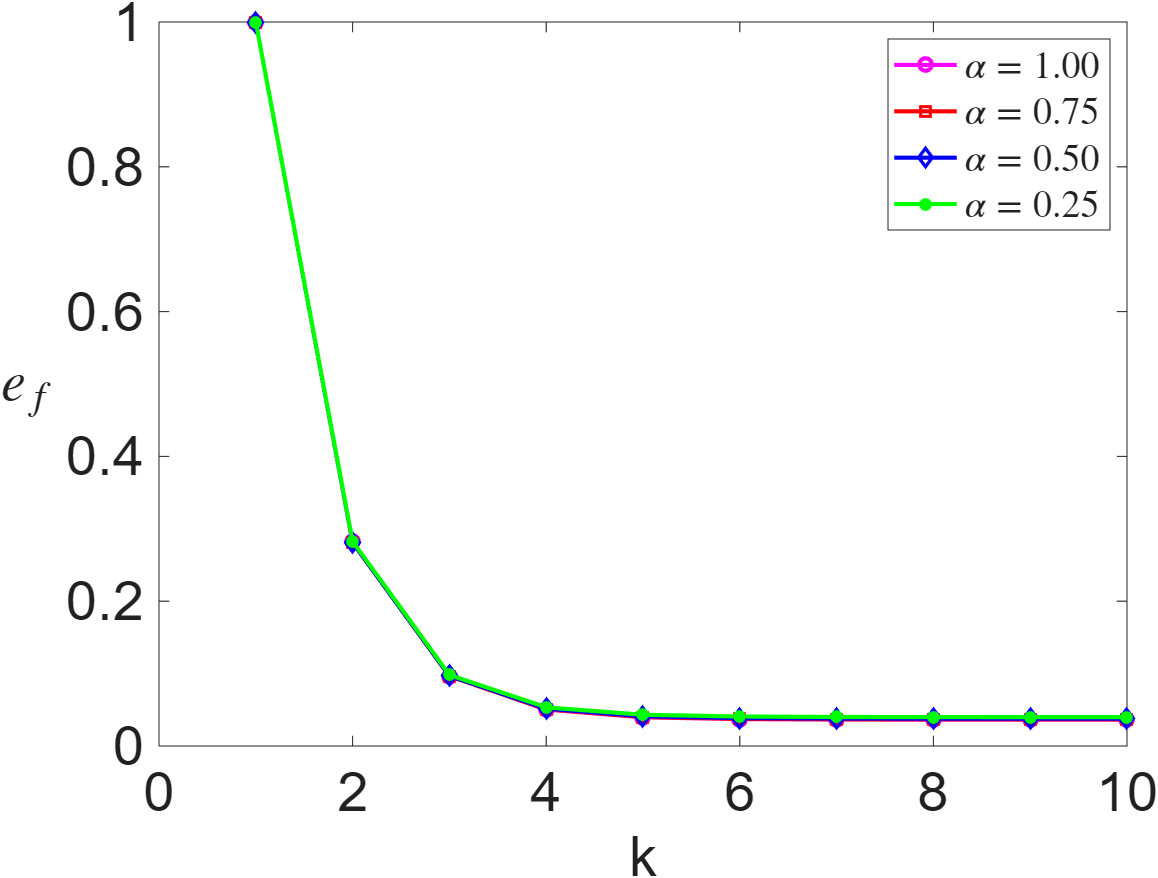}&
        \includegraphics[width=0.22\textwidth]{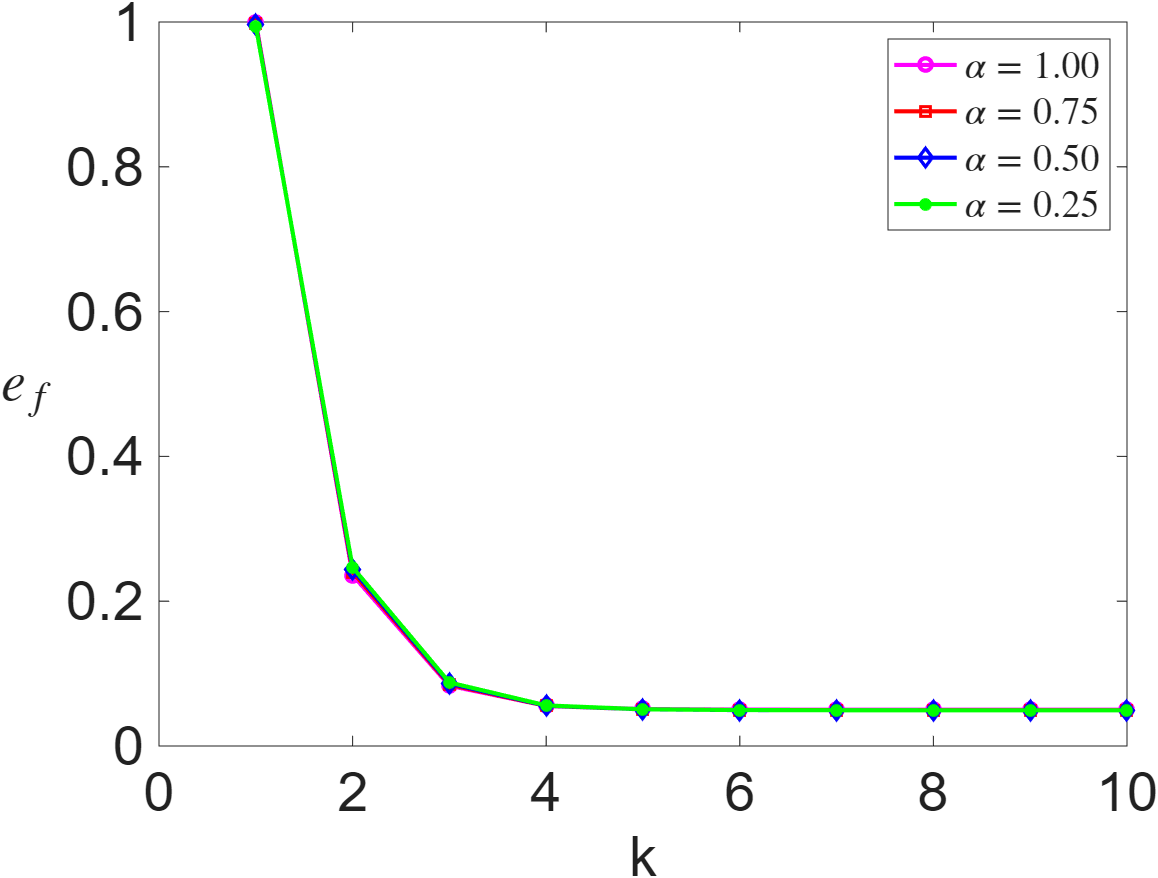}&
        \includegraphics[width=0.22\textwidth]{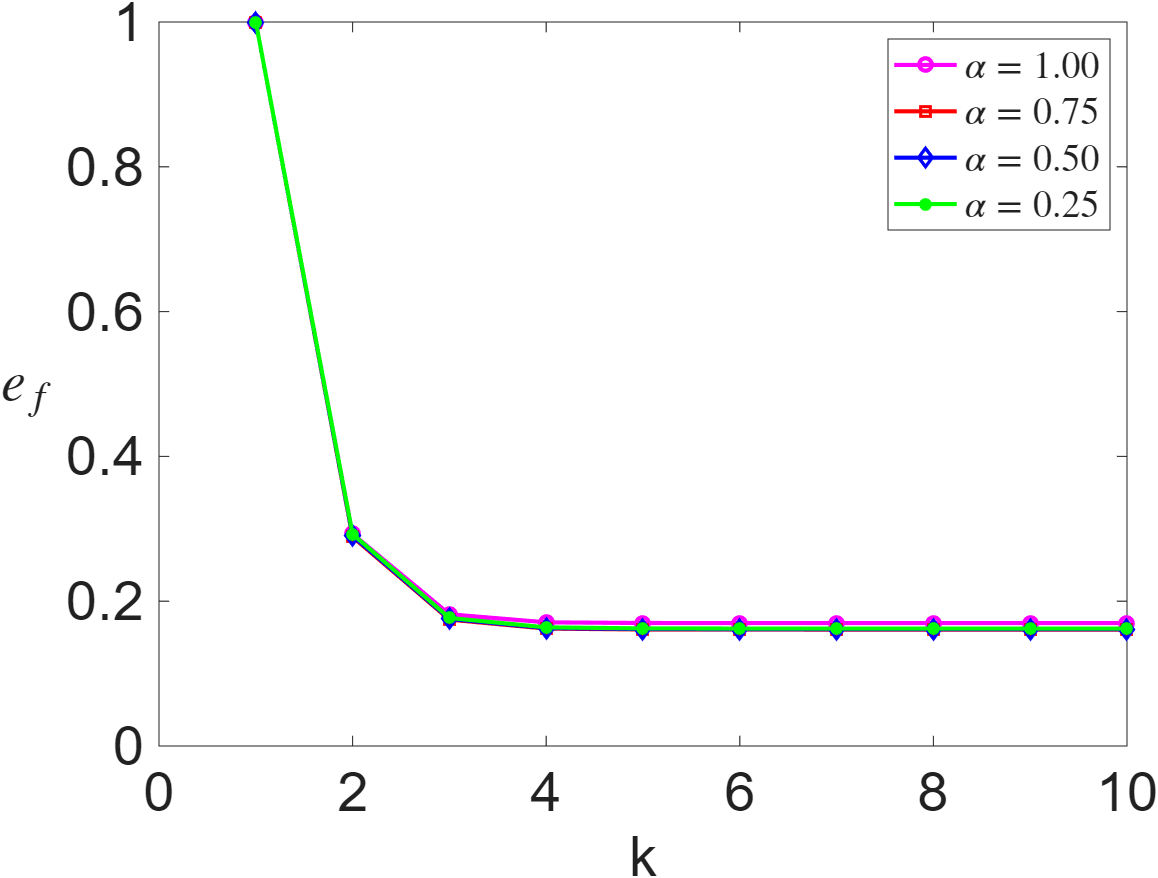}\\ 
		\includegraphics[width=0.22\textwidth]{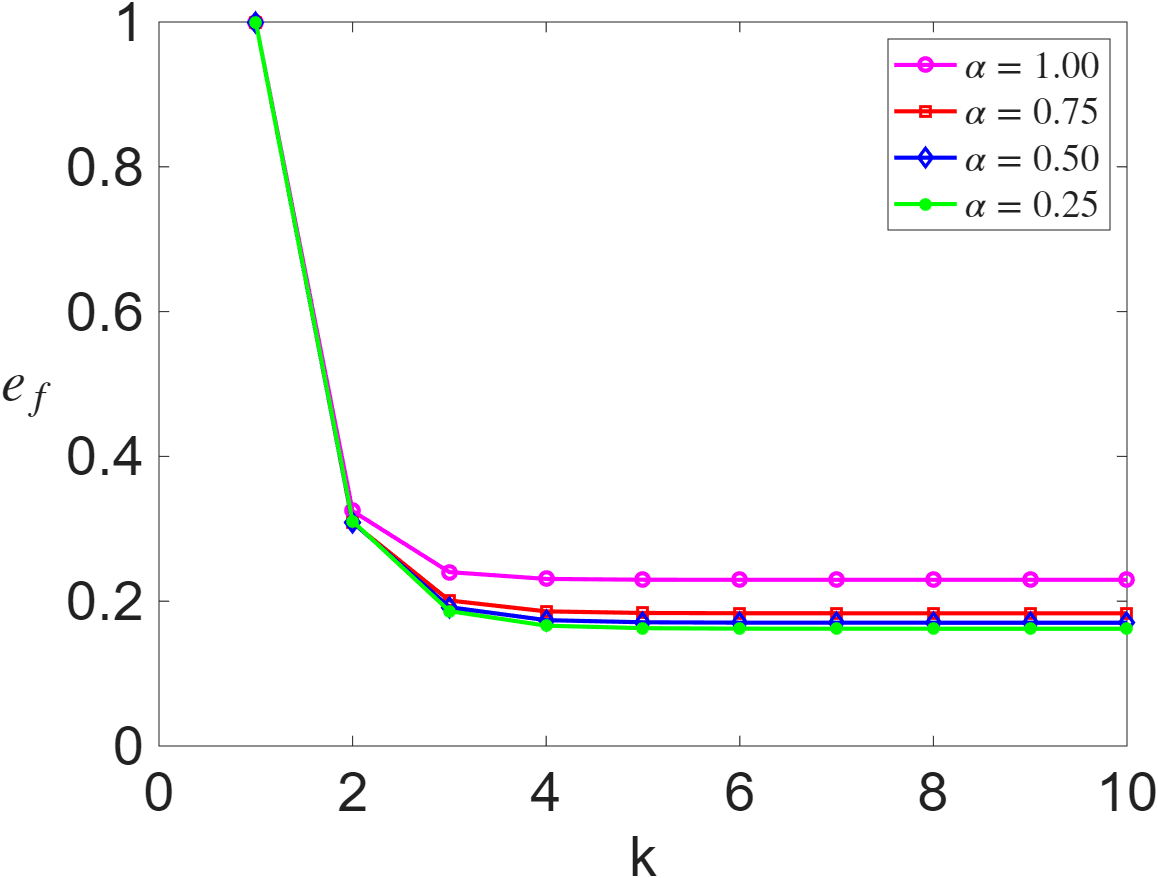}&
        \includegraphics[width=0.22\textwidth]{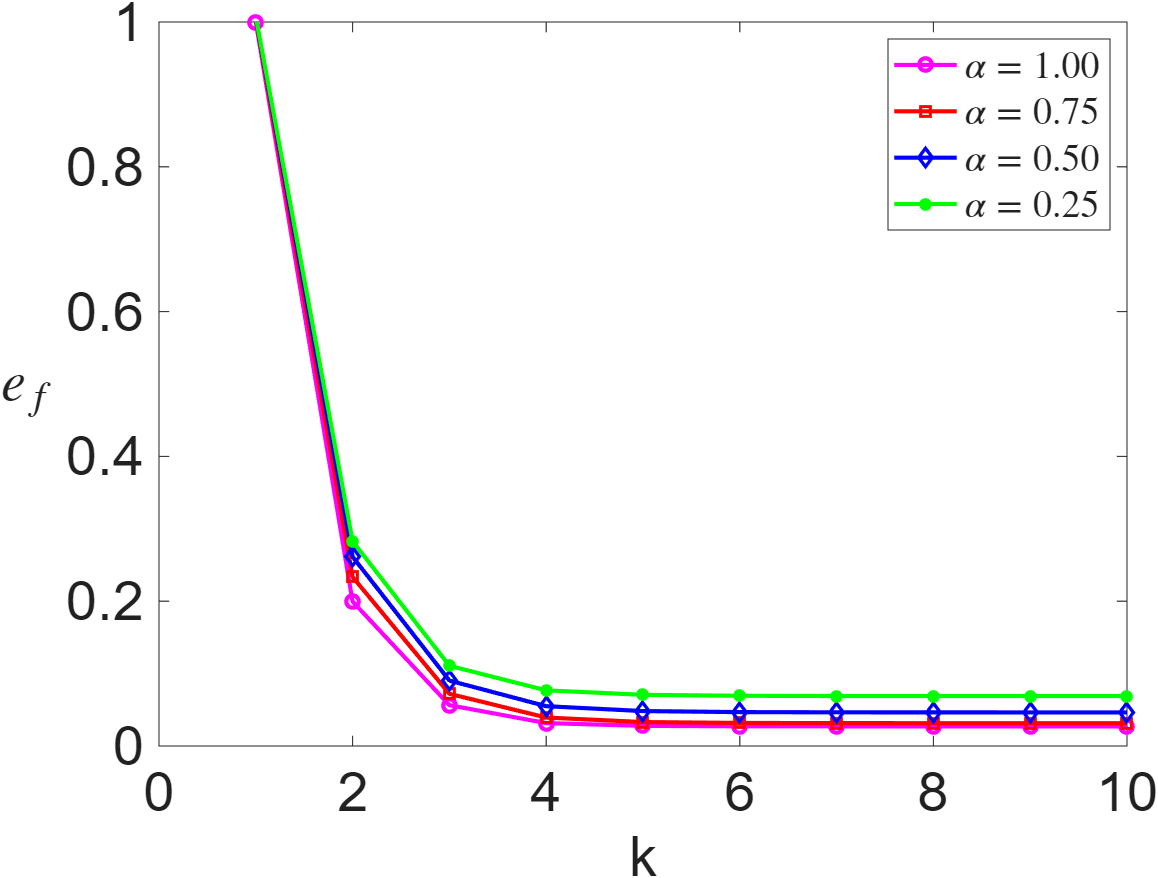}&
        \includegraphics[width=0.22\textwidth]{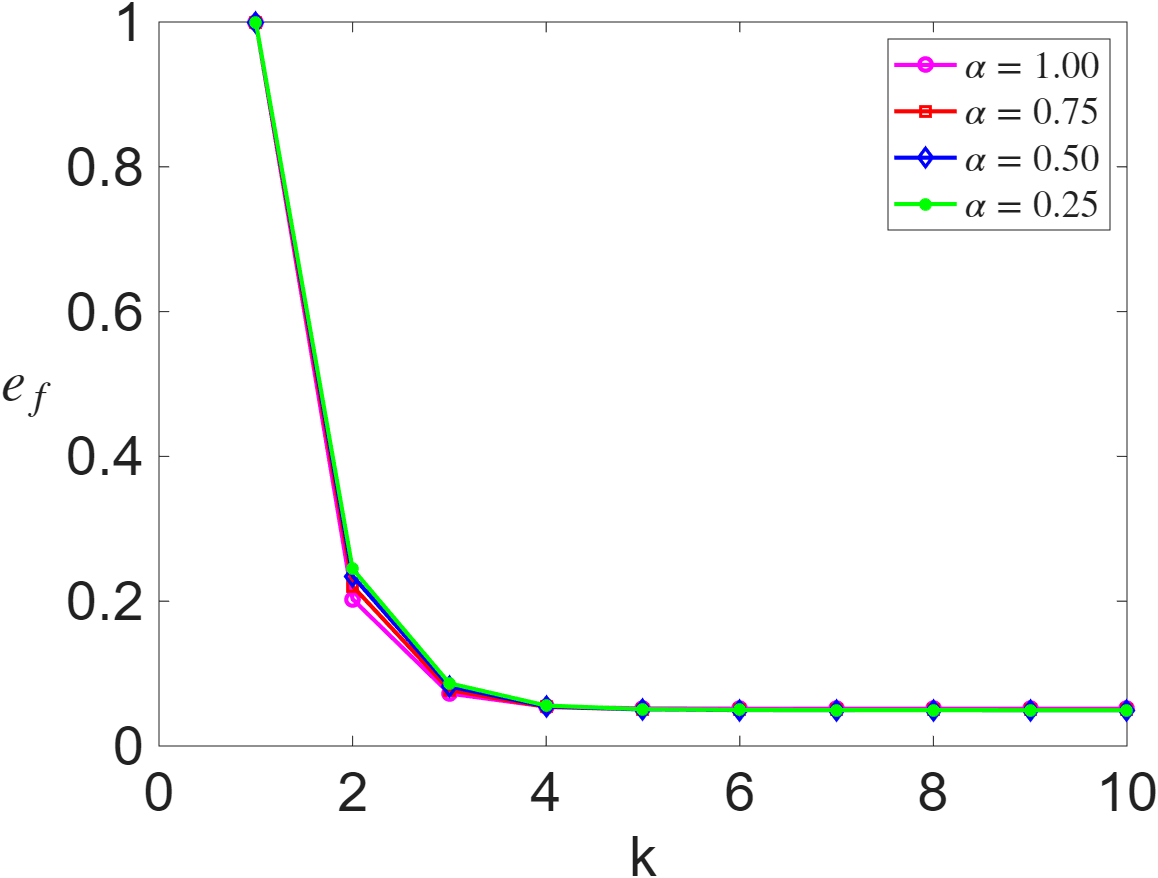}&
        \includegraphics[width=0.22\textwidth]{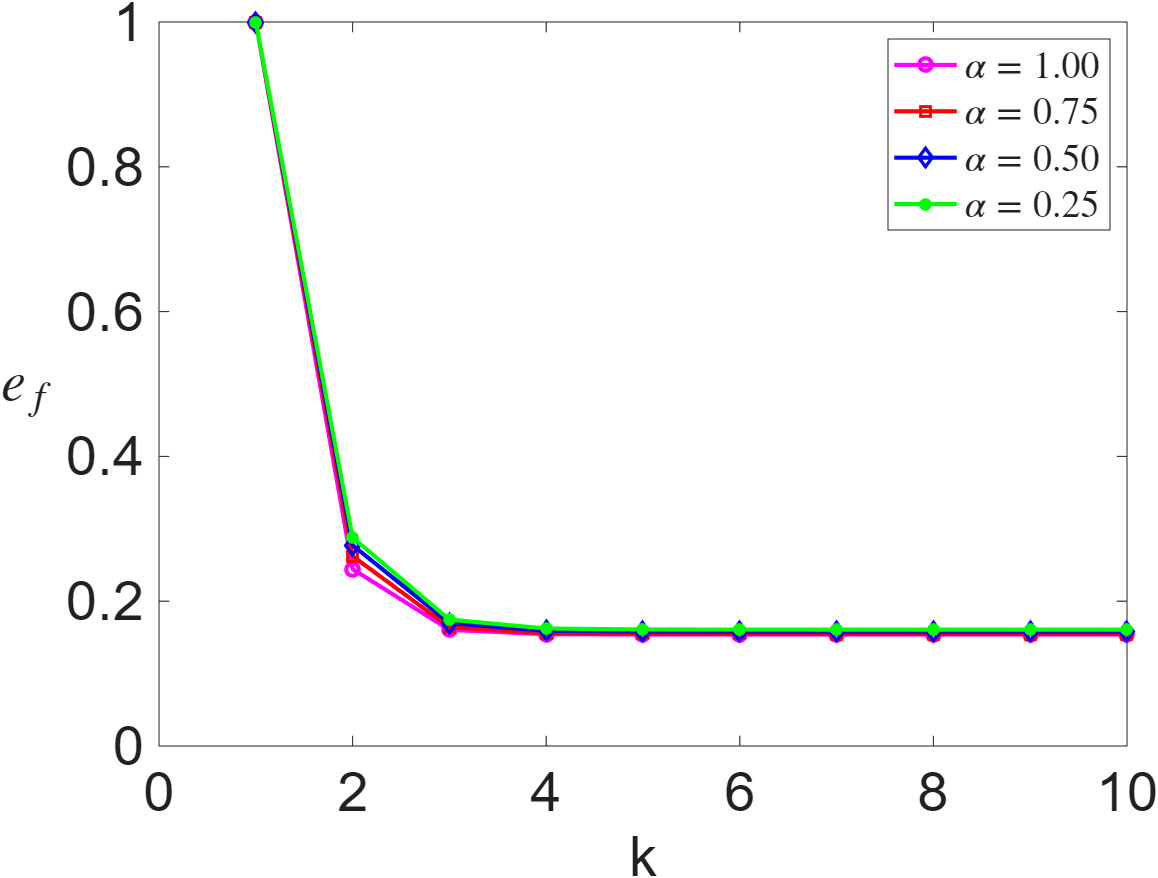}\\
		(a)  \ref{ex:smooth1} & (b)   \ref{ex:smooth2} & (c)  \ref{ex:nonsmooth1} & (d)   \ref{ex:nonsmooth2}
	\end{tabular}
	\caption{Evolution of errors with respect to iteration number $k$ in the case $\delta=\text{1e-4}$. Top: evolution of standard $\ell^2(L^2(\omega))$ error. Bottom:  evolution of weighted $\ell_\lambda^2(L^2(\omega))$ error, with $\lambda=10$.}
	\label{Fig:err_evolution}
\end{figure}

The next example involves a smooth source   not satisfying  Assumption \ref{assum:noise}.
\begin{example}\label{ex:smooth2}
     $R(x_1,x_2,t)=x_2$ and $f^\dagger(x_1,t)=\sin(\pi x_1) (2+\sin(4\pi t))$.
\end{example}

Both $f(x_1,0)$ and $\partial_t f(x_1,0)$ are non-vanishing, and Assumption \ref{assum:noise} does not hold. The state $u^\dagger=u(f^\dagger)$ has a weak singularity at $t=0$. We  verify the error bound  in Remarks \ref{rmk:error} and \ref{rmk:error_high_reg}:
\begin{equation*}
    \| ((f^*)^n-(f^\dagger)^n)_{n=1}^{N} \|_{ \ell^2(L^2(\omega))}\le  c(\tau^{-\alpha}\delta+\tau^{\frac{1}{2}}|\log\tau|+ h^{-2}\delta+ h ),
\end{equation*}
Fig. \ref{Fig:conv_smooth2} presents the convergence rates   with respect to $h$, $\tau$ and $\delta$.  First, we consistently observe an $O(h)$ convergence rate.   However, the temporal convergence varies with the order $\alpha$. For $\alpha=1$, the convergence rate is $O(\tau)$, since the solution to the heat equation is smooth at $t = 0$. For $\alpha=0.75$, the convergence rate is $O(\tau^{0.74})$. The singularity becomes stronger as $\alpha$ decreases to zero, which leads to a further decline in the convergence rate. The convergence rate is around $O(\tau^{0.5})$ for $\alpha=0.5$ and $\alpha=0.25$, which agrees with the theoretical  prediction $O(\tau^{\frac{1}{2}}|\log\tau|)$ from Remark \ref{rmk:error}. 
To study the convergence rate with respect to $\delta$, we choose $h$ and $\tau$ as $h^{-2}\delta\sim h$ and $\tau^{-\alpha}\delta\sim\tau^{\frac{1}{2}}$, with the initial parameters $h=1/6$ and $\tau=1/20$ for $\delta=\text{1e-2}$. Fig. \ref{Fig:conv_smooth2} (c) indicates a  convergence rate $O(\delta^{0.32})$, which is consistent with the predicted rate $O(\delta^{\frac{1}{3}})$.
The reconstructions for $\alpha=0.75$ and $\delta\in \{\text{1e-2, 1e-3, 1e-4}\}$ are presented in Fig. \ref{Fig:recon_smooth2}. The reconstructions are reliable with minor deterioration in the case $\delta=\text{1e-2}$. 
Fig. \ref{Fig:err_evolution} (b) shows the error evolutions with respect to the iteration number $k$ in the case $\delta=\text{1e-4}$, starting from the initial guess  $f^0=0$. The error decay behaves similarly to Example \ref{ex:smooth1}.

\begin{figure}[htbp]
	\centering
	\begin{tabular}{ccc}
		\includegraphics[width=0.30\textwidth]{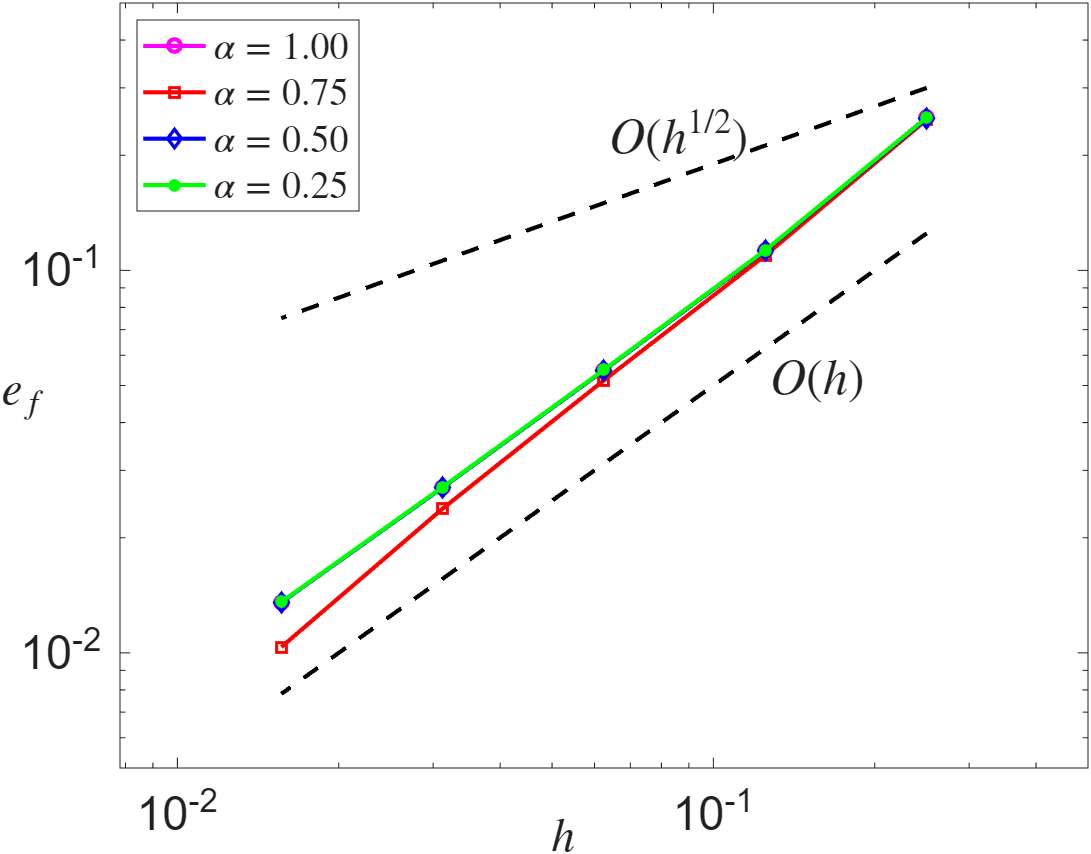}&
		\includegraphics[width=0.30\textwidth]{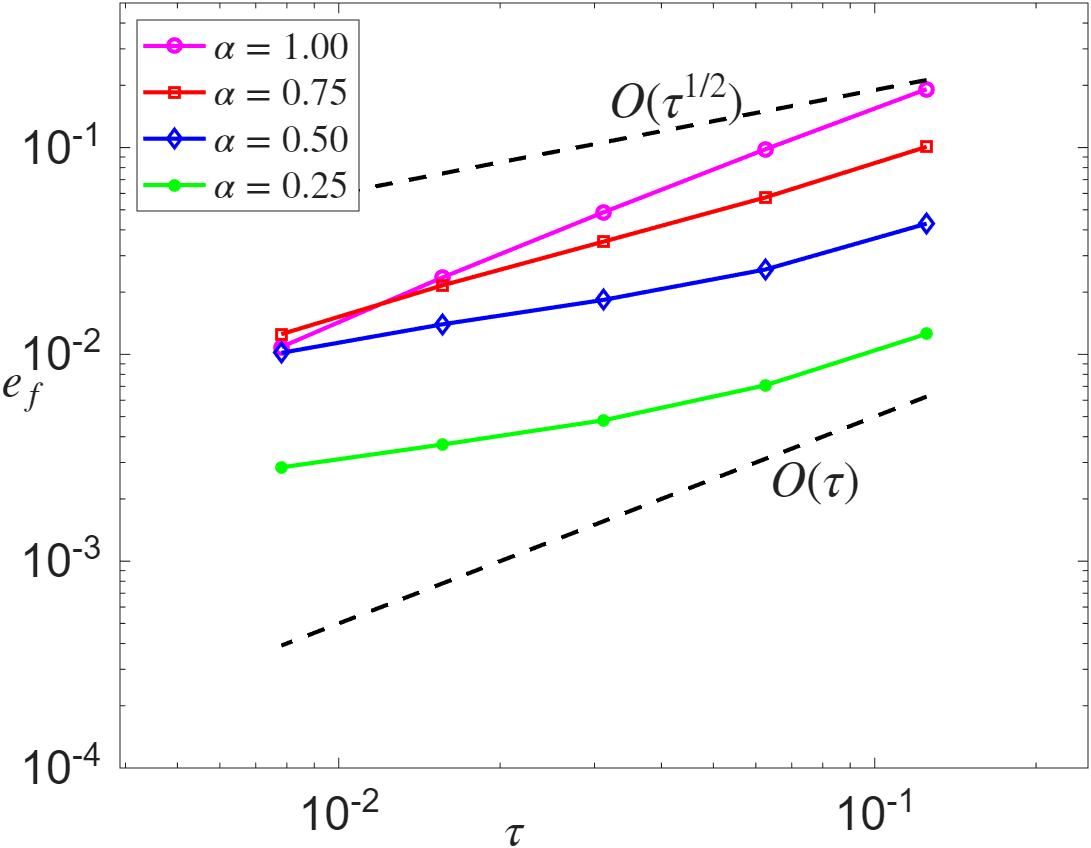}&
		\includegraphics[width=0.30\textwidth]{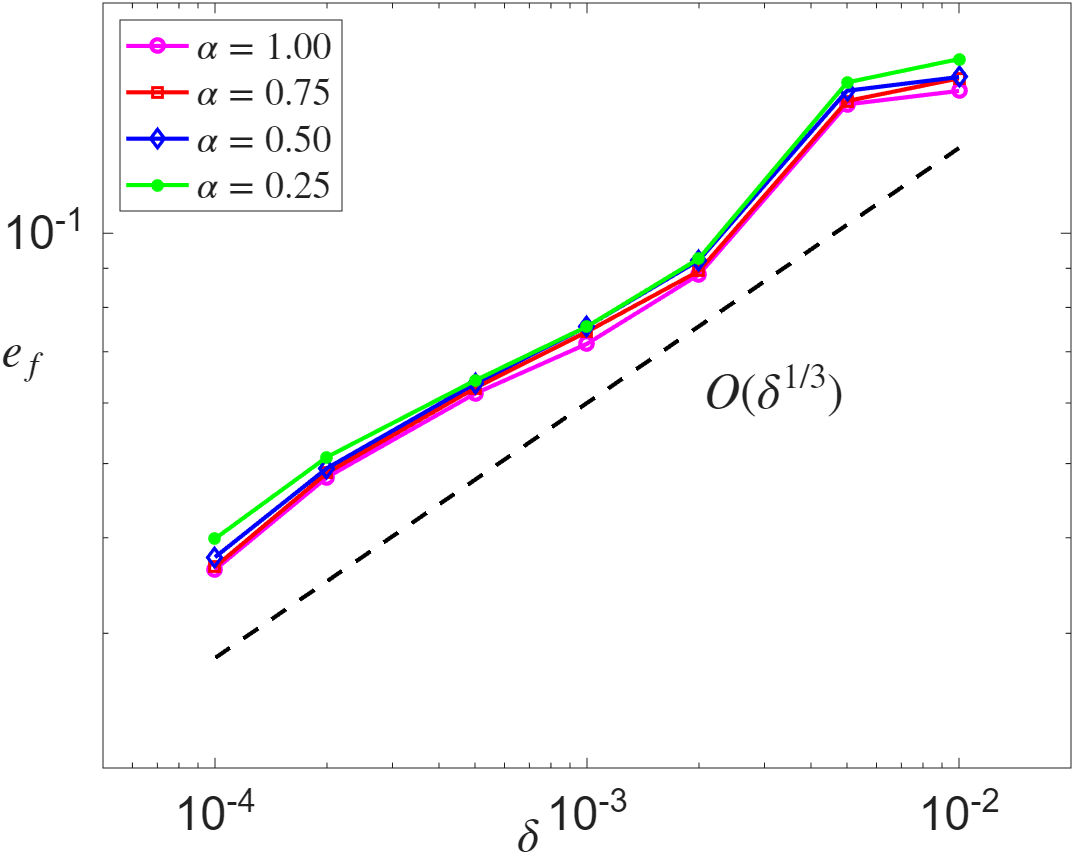}\\
		(a) $e_f$ versus $h$ & (b) $e_f$ versus $\tau$  & (c) $e_f$ versus $\delta$  
	\end{tabular}
	\caption{The convergence of $e_f$ with respect to $h$, $\tau$ and $\delta$ for Example  \ref{ex:smooth2}.}
	\label{Fig:conv_smooth2}
\end{figure}
\begin{figure}[htbp]
	\centering
	\begin{tabular}{cccc}
		\includegraphics[width=0.22\textwidth]{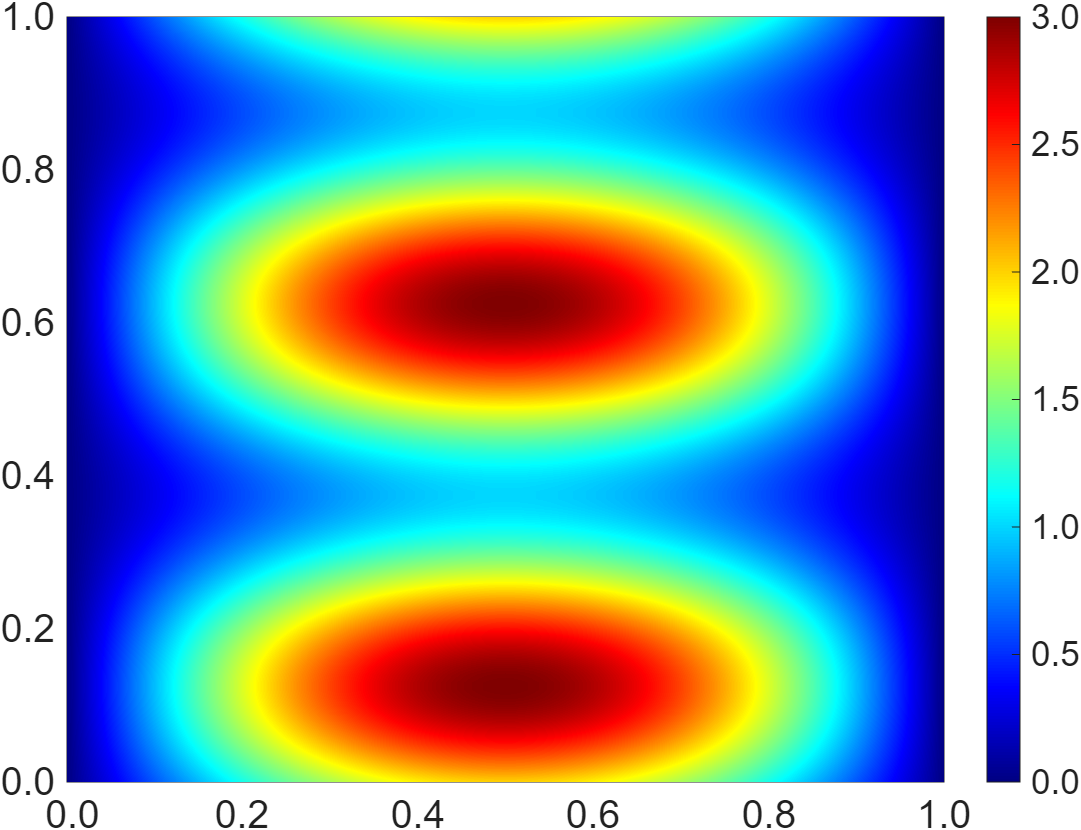}&
		\includegraphics[width=0.22\textwidth]{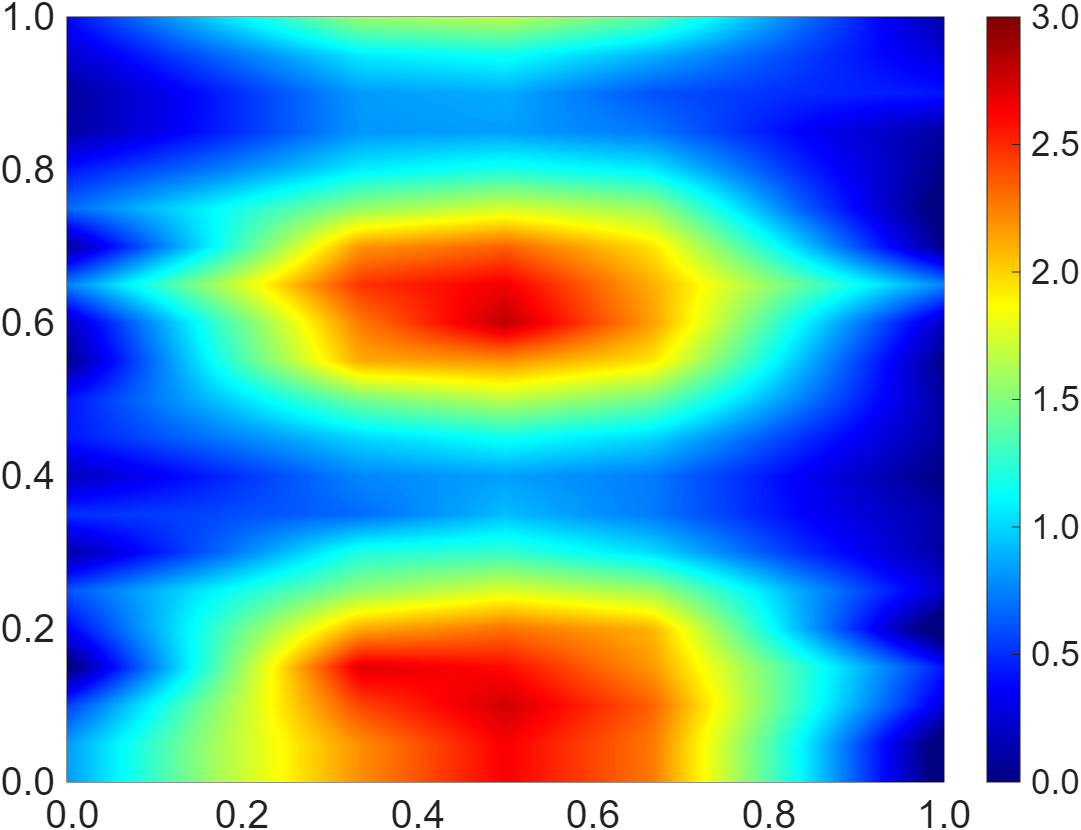}&
		\includegraphics[width=0.22\textwidth]{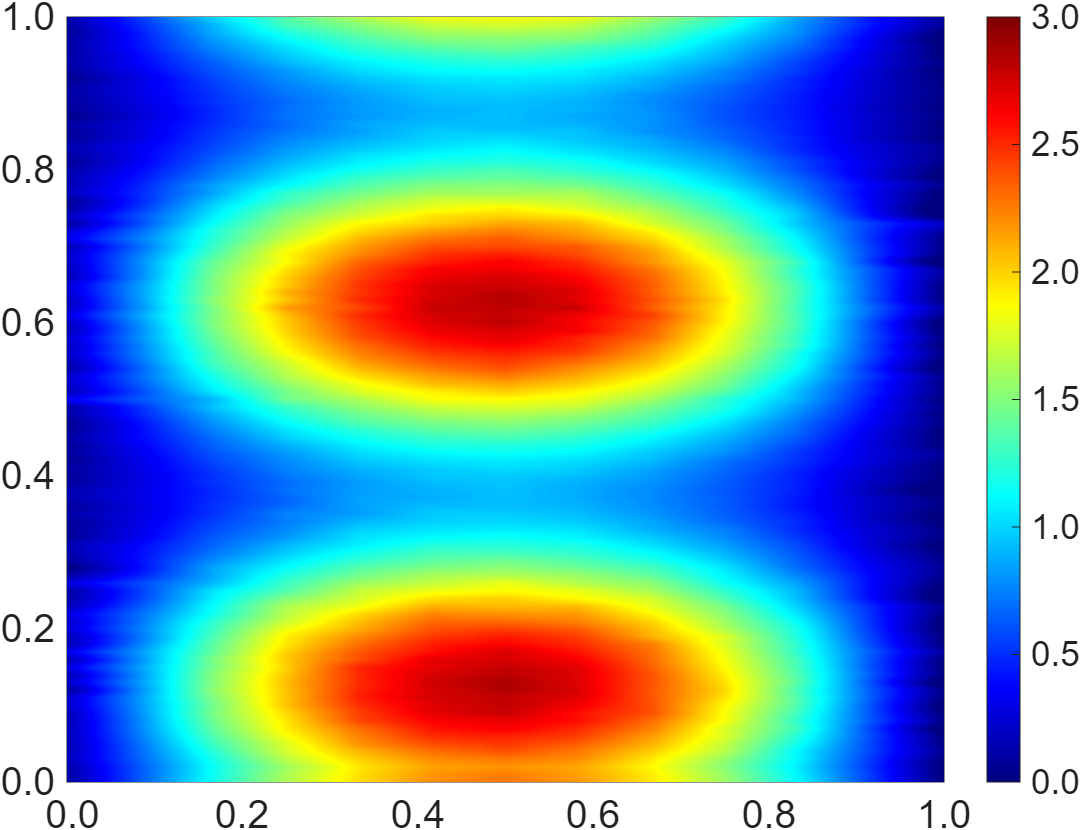}&
		\includegraphics[width=0.22\textwidth]{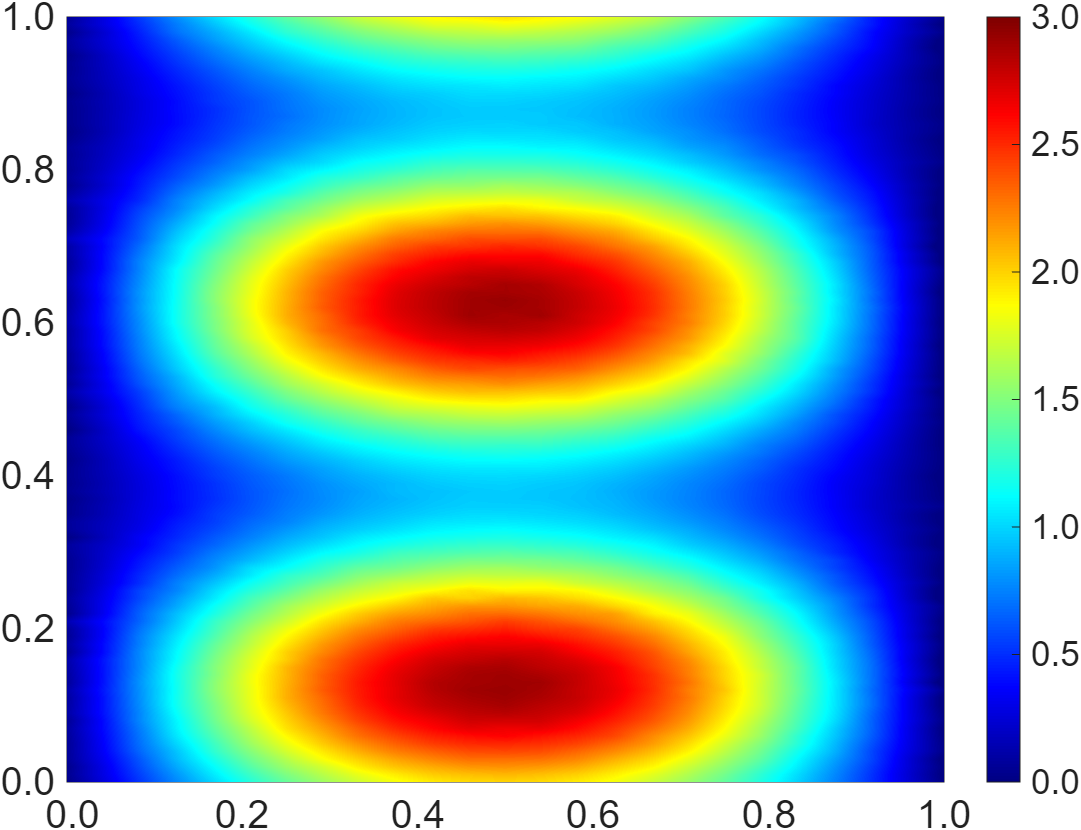}\\
		(a) exact & (b) $\delta=\text{1e-2}$  & (c) $\delta=\text{1e-3}$ & (d) $\delta=\text{1e-4}$  
	\end{tabular}
	\caption{The reconstruction $f^*$  for Example \ref{ex:smooth2}  with $\alpha=0.75$ at different noise levels.}
	\label{Fig:recon_smooth2}
\end{figure}

Next, we investigate nonsmooth sources to verify the sharpness of the error estimate in Theorem \ref{thm:error}.  $\chi_S$ denotes the characteristic function
of a set $S$.
\begin{example}\label{ex:nonsmooth1}
    $R(x_1,x_2,t)=x_2$ and $f^\dagger(x_1,t)= (1+\chi_{[0.5,1]}(x_1))(e^t-1)t$.
\end{example}

The source $f^\dag$ is smooth in time and piecewise constant in space and belongs to $H^{\frac{1}{2}-\epsilon}(\omega)$, for any $\epsilon>0$. Theorem \ref{thm:error} and Remark \ref{rmk:error_high_reg}  imply the error bound
\begin{equation*}
    \| ((f^*)^n-(f^\dagger)^n)_{n=1}^{N} \|_{ \ell^2(L^2(\omega))}\le  c(\tau^{-\alpha}\delta+\tau+ h^{-2}\delta+  h^{\frac{1}{2}-\epsilon}).
\end{equation*}
When $\delta=0$, Fig.~\ref{Fig:conv_nonsmooth1} shows a convergence rate $O(\tau)$ in time, and $O(h^{0.64})$ in space. The latter is close to the theoretical prediction $O(h^{\frac12})$, and indicates the  sharpness of  the error estimate.
To examine the convergence with respect to $\delta$, we take the parameters $h=1/6$ and $\tau=1/20$ for $\delta=10^{-2}$ and then choose $h$ and $\tau$ according to Theorem \ref{thm:error}, namely $h^{-2}\delta \sim h^{\frac12}$ and $\tau^{-\alpha}\delta \sim \tau$.  The empirical rate is about  $O(\delta^{0.24})$, which is very close to the theoretical rate $O(\delta^{\frac15})$.

Fig. \ref{Fig:recon_nonsmooth1} illustrates the numerical reconstructions with $\alpha=0.75$ and $\delta\in \{\text{1e-2, 1e-3, 1e-4}\}$. The recovered source $f^*$ fails to capture the discontinuity of the exact source $f^\dagger$. Moreover, refining the mesh does not necessarily improve the reconstruction quality, since the algorithm involves numerically differentiating the noisy data $z^\delta$. The error bound $h^{-2}\delta+h^{\min(q,\frac{1}{2})}$ indicates that a finer mesh may produce similar or even worse reconstructions.
The error evolutions in the case $\delta=\text{1e-4}$ with respect to the iteration number $k$ are shown in Fig. \ref{Fig:err_evolution} (c), with $f^0=0$.  

\begin{figure}[htbp]
	\centering
	\begin{tabular}{ccc}
		\includegraphics[width=0.30\textwidth]{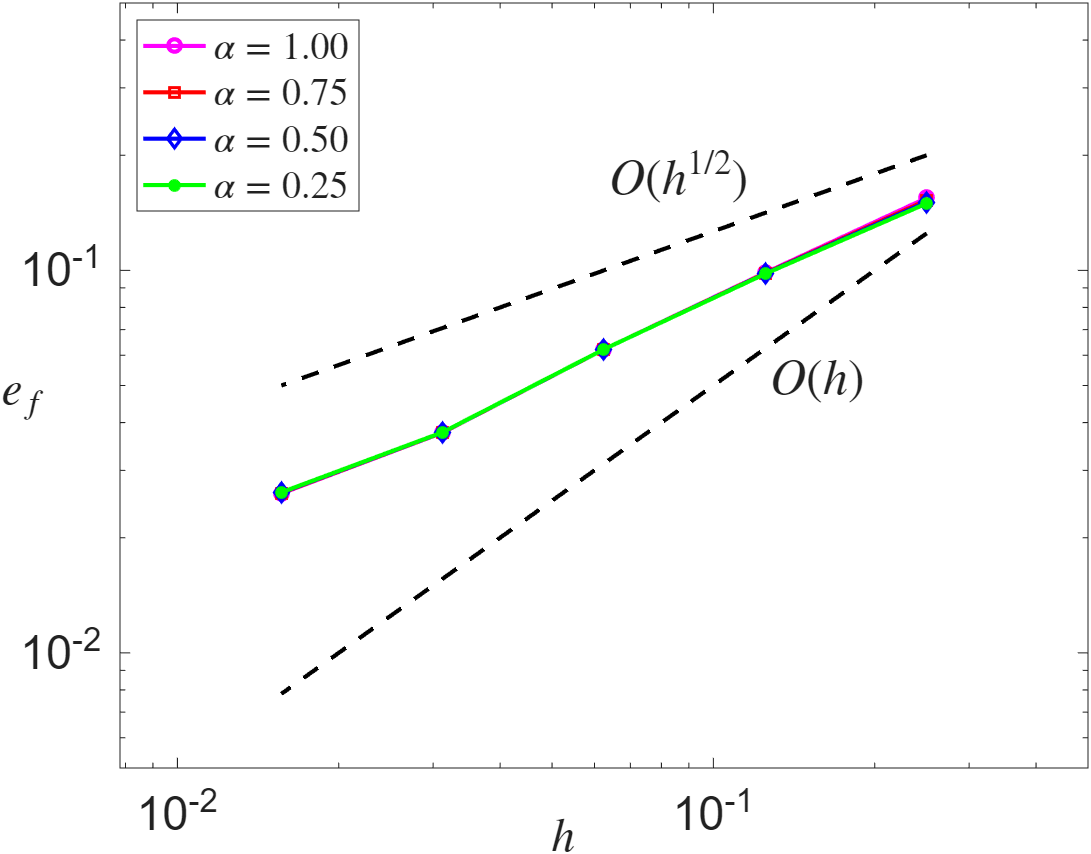}&
		\includegraphics[width=0.30\textwidth]{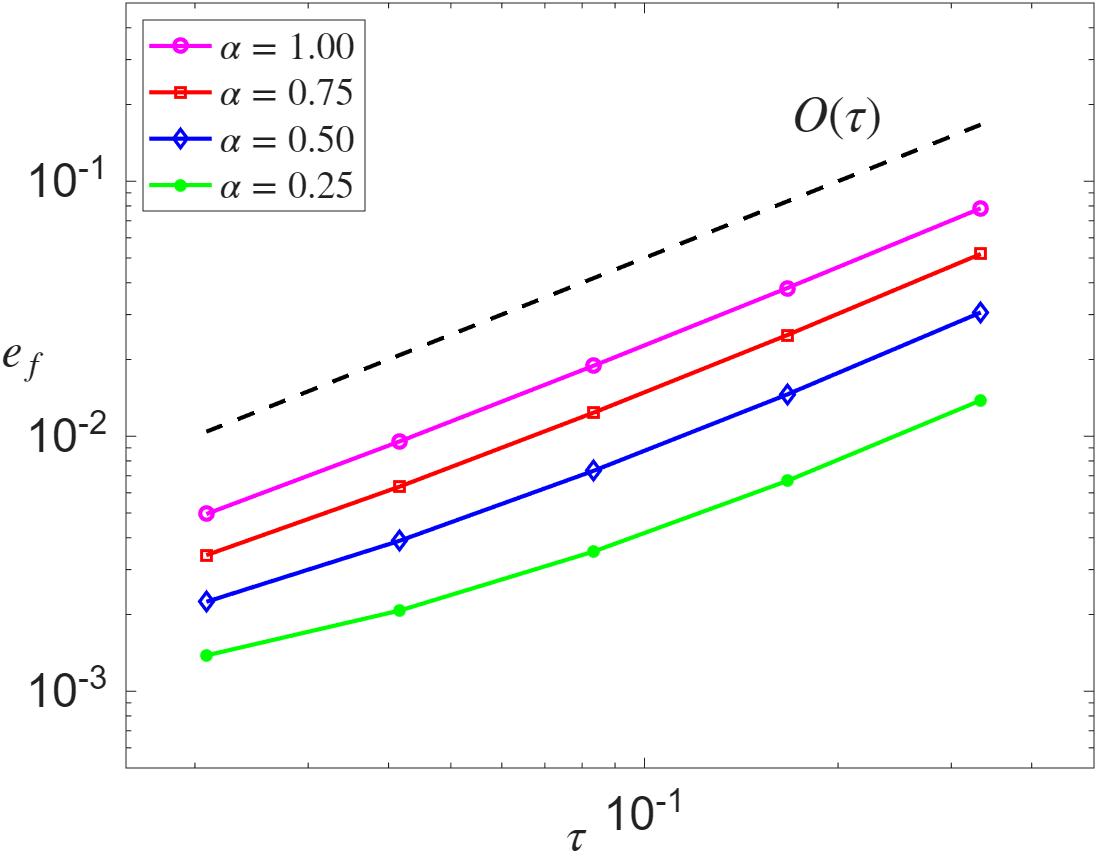}&
		\includegraphics[width=0.30\textwidth]{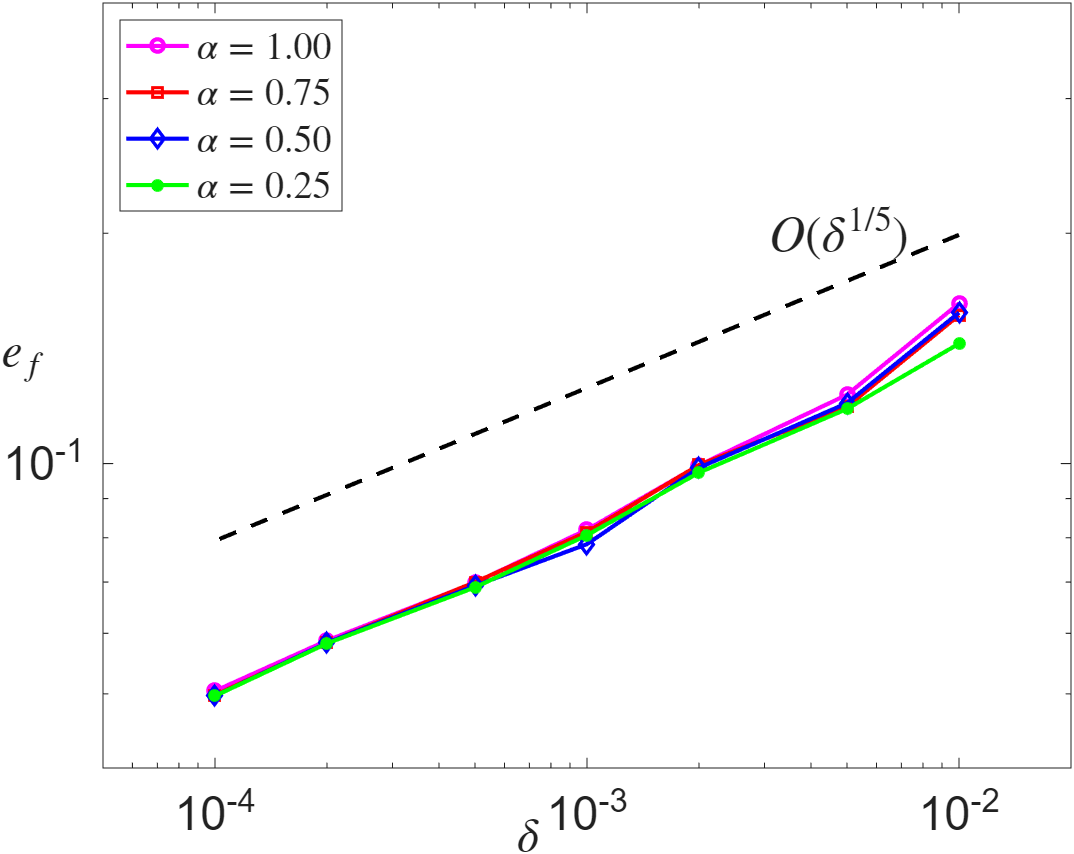}\\
		(a)  $e_f$ versus $h$ & (b)  $e_f$ versus $\tau$  & (c)  $e_f$ versus $\delta$  
	\end{tabular}
	\caption{The convergence of $e_f$ with respect to $h$, $\tau$ and $\delta$ for Example \ref{ex:nonsmooth1}.}
	\label{Fig:conv_nonsmooth1}
\end{figure}
\begin{figure}[htbp]
	\centering
	\begin{tabular}{cccc}
		\includegraphics[width=0.22\textwidth]{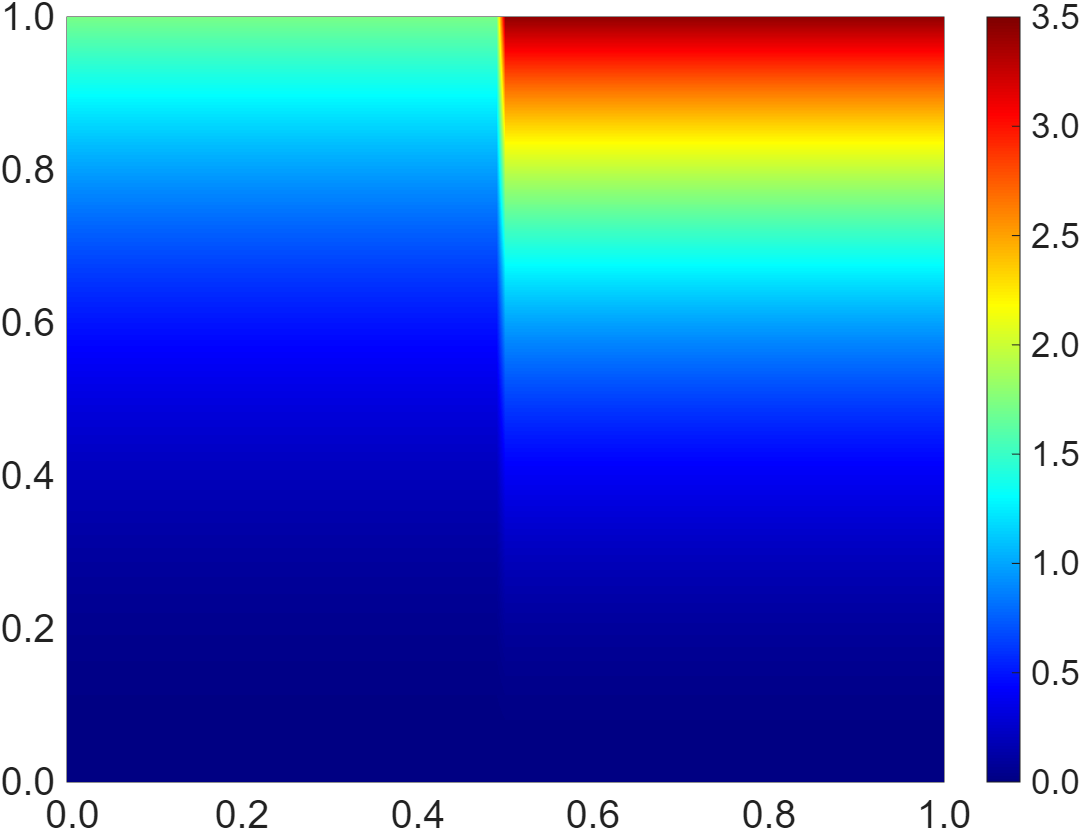}&
		\includegraphics[width=0.22\textwidth]{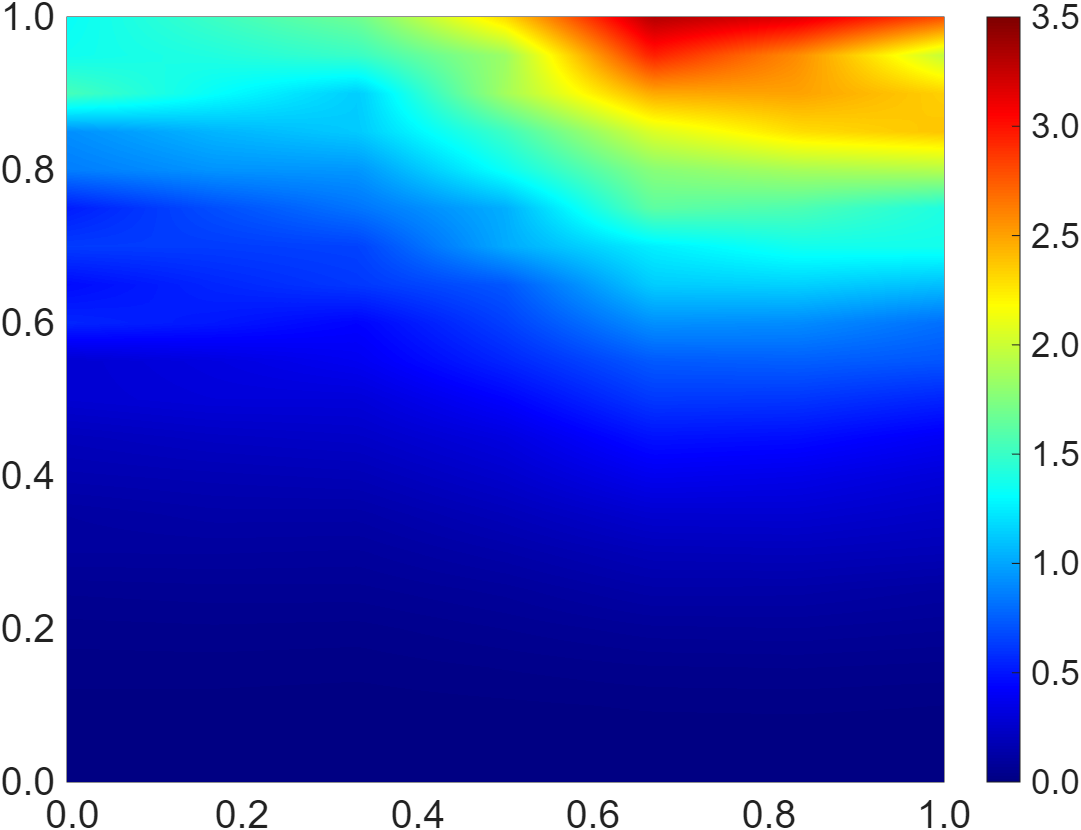}&
		\includegraphics[width=0.22\textwidth]{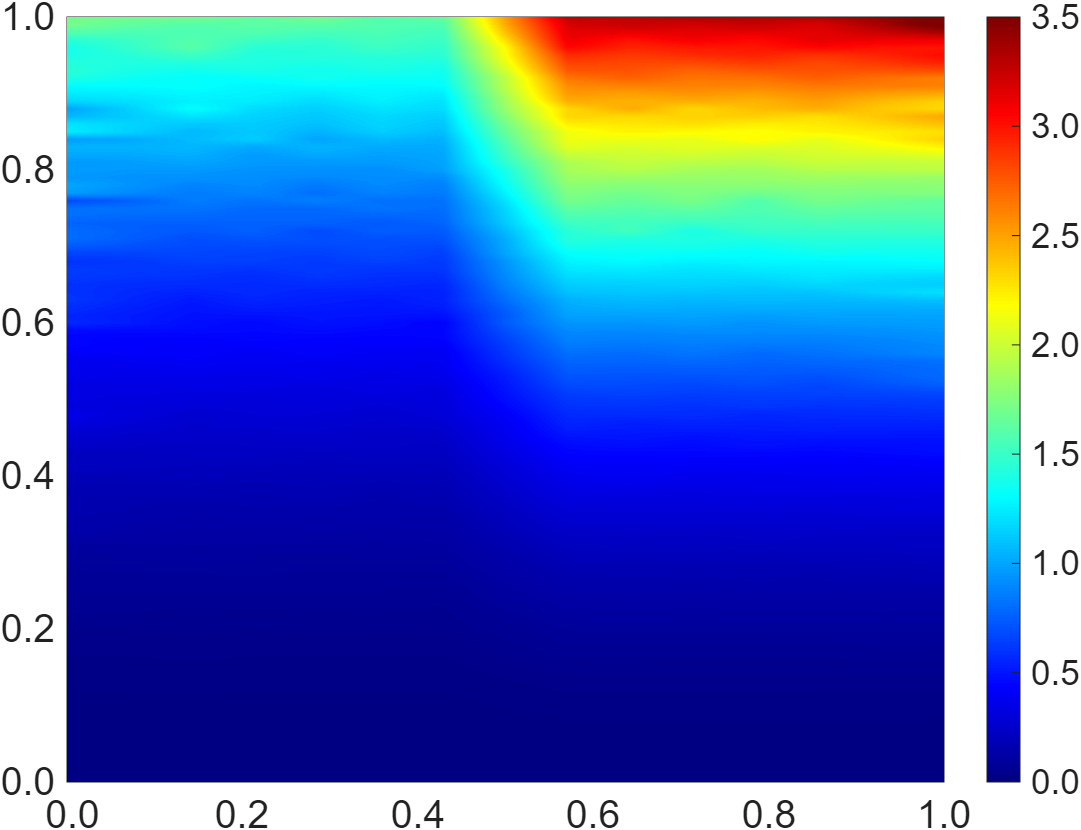}&
		\includegraphics[width=0.22\textwidth]{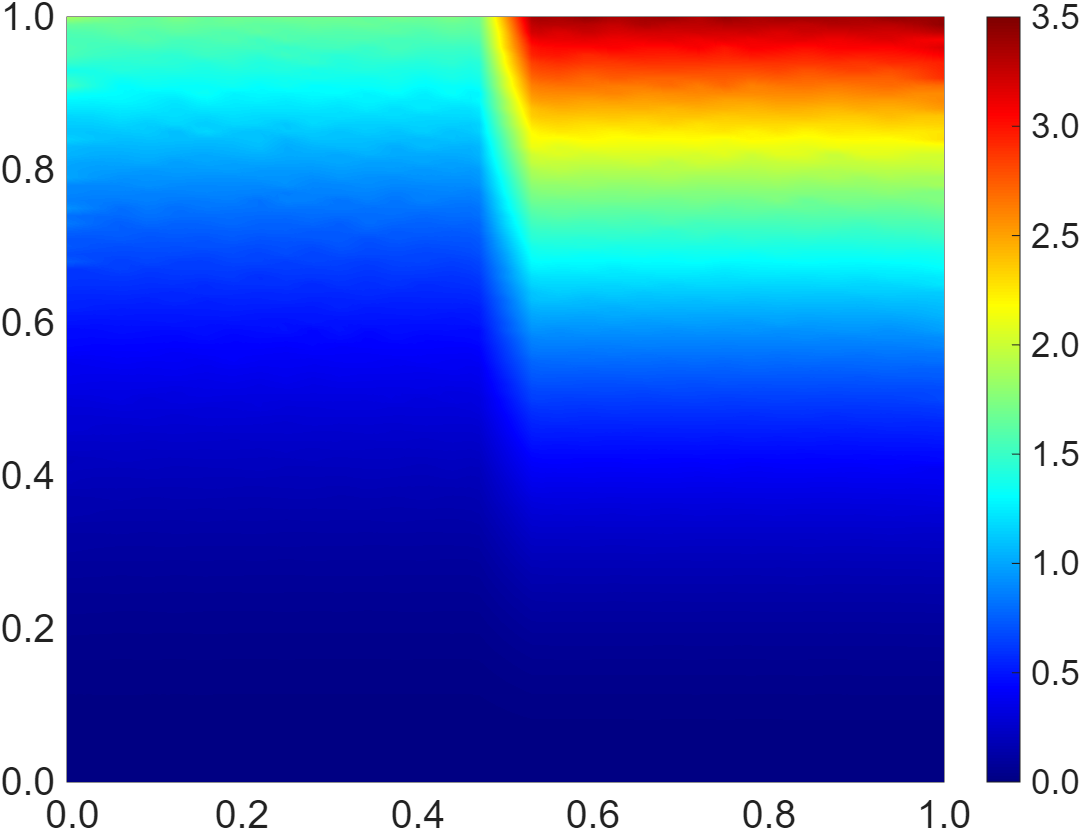}\\
		(a) exact & (b) $\delta=\text{1e-2}$  & (c) $\delta=\text{1e-3}$ & (d) $\delta=\text{1e-4}$  
	\end{tabular}
	\caption{The reconstruction $f^*$  for Example \ref{ex:nonsmooth1}  with $\alpha=0.75$ at different noise levels.}
	\label{Fig:recon_nonsmooth1}
\end{figure}

Last, we consider a more singular source.
\begin{example}\label{ex:nonsmooth2}
      $R(x_1,x_2,t)=x_2$ and $f^\dagger(x_1,t)= (|x_1-0.5|+\text{1e-4})^{-0.4}(1-\cos(4\pi t))t$.
\end{example}

Note that the function $ (|x_1-0.5| )^{-0.4}$ only belongs to $H^{0.1-\epsilon}(\omega)$, with $\epsilon>0$, and the constant $\text{1e-4}$ is to regularize the spatial singularity at $x_1 = 0.5$.  Fig. \ref{Fig:conv_nonsmooth2} shows an $O(\tau)$ and $O(h^{0.17})$ rate in time and space, respectively. The spatial convergence aligns closely with the theoretical prediction $O(h^{0.1})$, indicating the sharpness of the error estimate. To test the convergence with respect to $\delta$, we initialize $h=1/6$ and $\tau=1/20$ for $\delta=\text{1e-2}$, and choose $h$ and $\tau$ following Theorem \ref{thm:error}, i.e.,  $h^{-2}\delta\sim h^{0.1}$ and $ \tau^{-\alpha}\delta\sim \tau$. Fig. \ref{Fig:conv_nonsmooth2} (c) shows an empirical rate  $O(\delta^{0.08})$, slightly higher than the theoretical rate $O(\delta^{1/21})$.  

Fig. \ref{Fig:recon_nonsmooth2} presents the  reconstructions for $\alpha=0.75$  and $\delta\in \{\text{1e-2, 1e-3, 1e-4}\}$. The exact source $f^\dagger$ consists of a peak centered at $x=0.5$ with a magnitude of $60$. The reconstructed source $f^*$ successfully captures the qualitative shape of $f^\dagger$, but it fails to recover the peak intensity, with a maximum magnitude  around $20$. The resolution is expected to improve with adaptive mesh refinement techniques. 
The error evolutions in the case $\delta=\text{1e-4}$ with the initial guess $f^0=0$ are presented in Fig. \ref{Fig:err_evolution} (d).  

\begin{figure}[htbp]
	\centering
	\begin{tabular}{ccc}
		\includegraphics[width=0.30\textwidth]{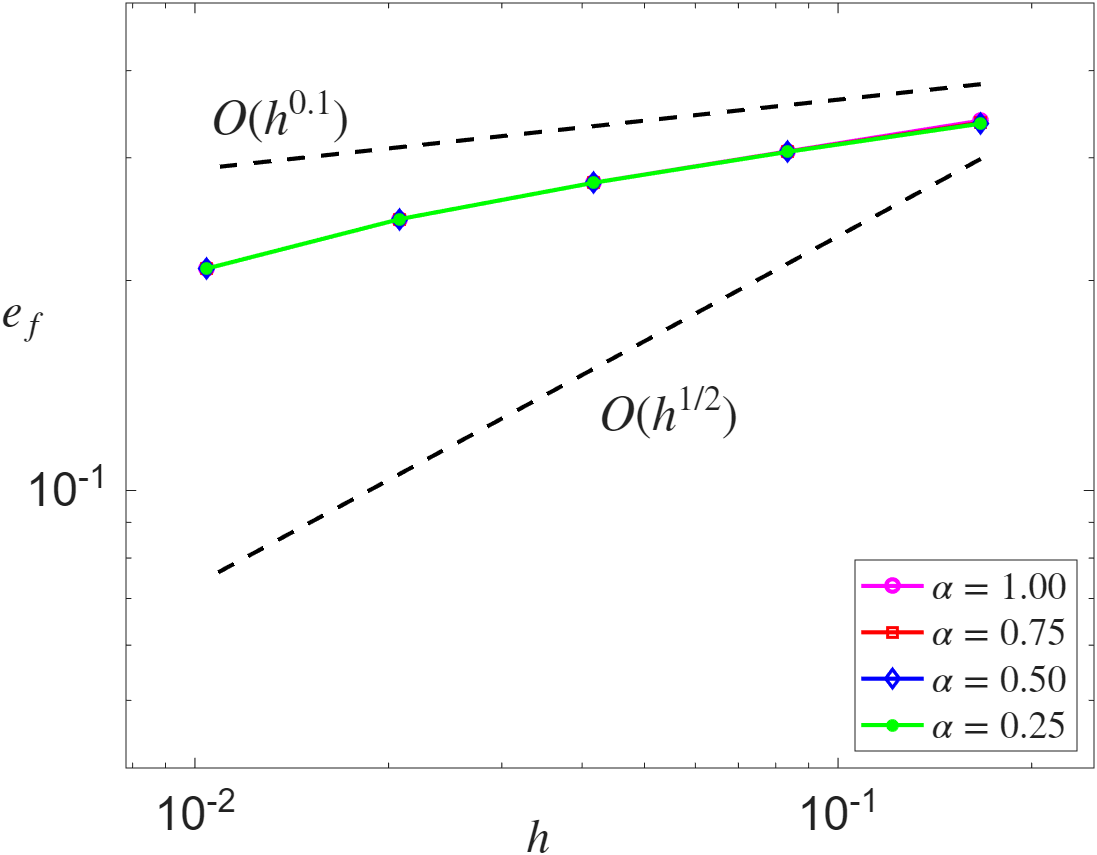}&
		\includegraphics[width=0.30\textwidth]{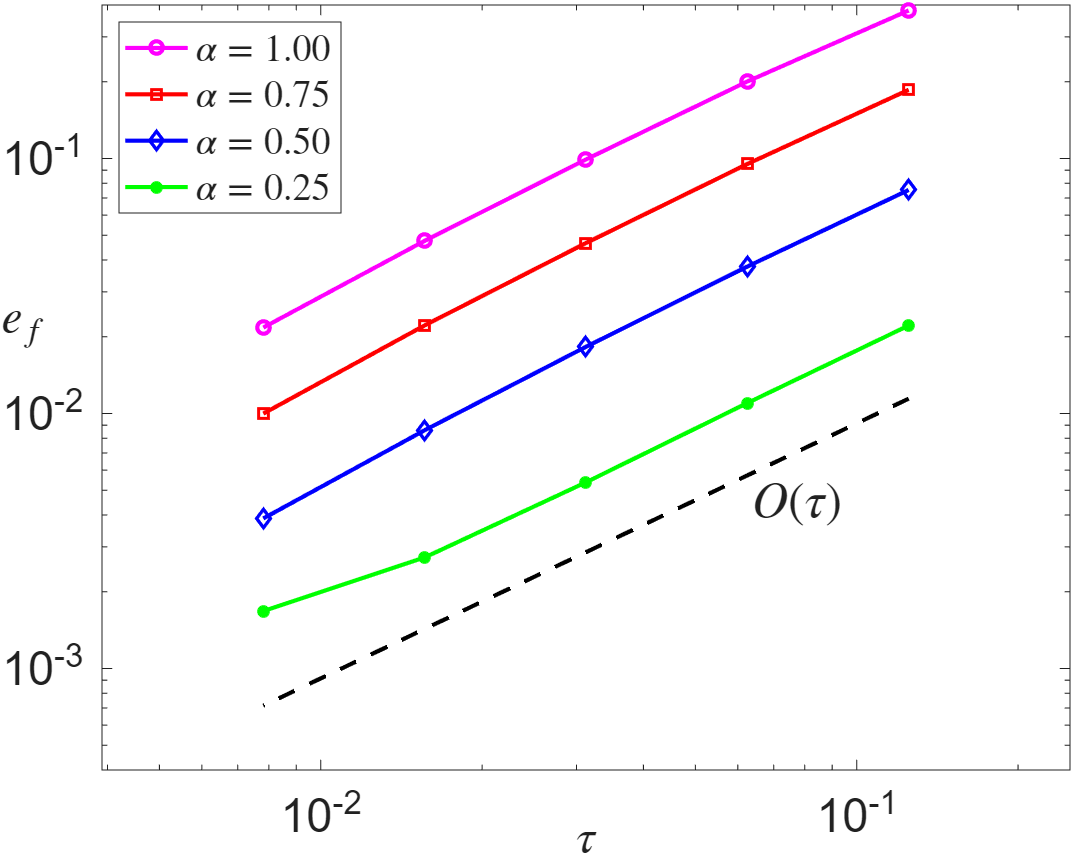}&
		\includegraphics[width=0.30\textwidth]{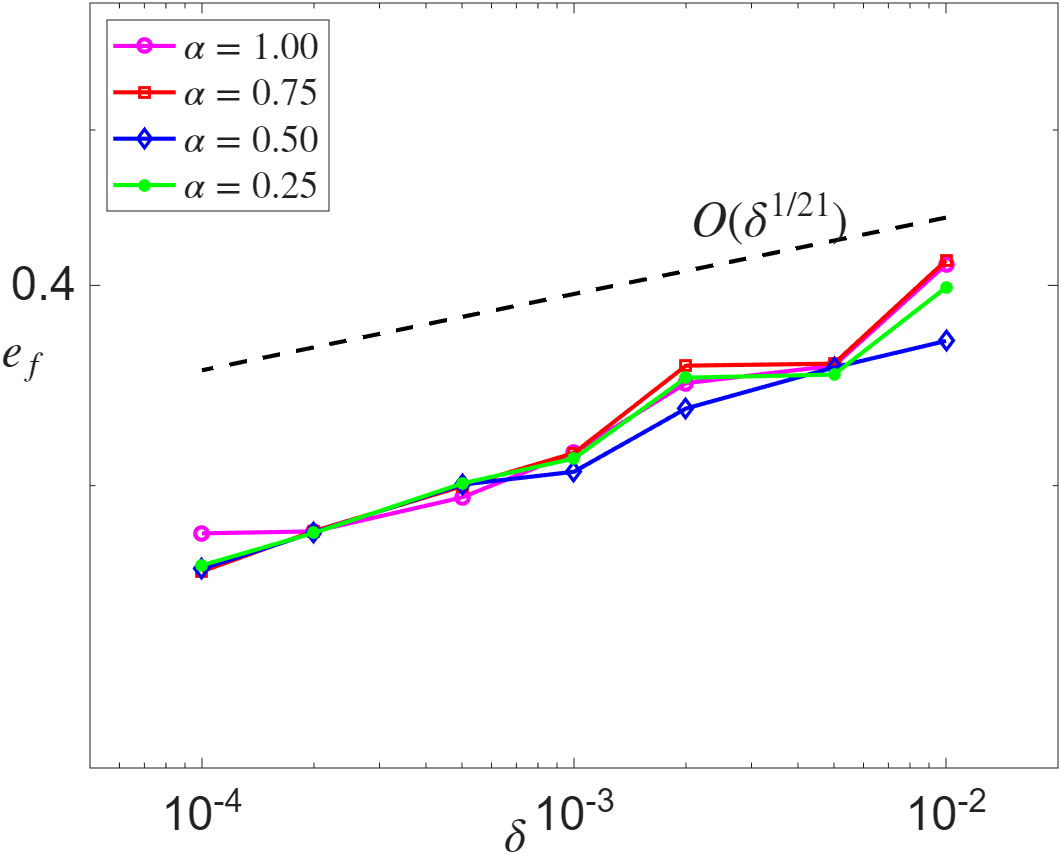}\\
		(a)  $e_f$ versus  $h$ & (b)  $e_f$ versus  $\tau$  & (c)  $e_f$ versus  $\delta$  
	\end{tabular}
	\caption{The convergence of $e_f$ with respect to $h$, $\tau$ and $\delta$ for Example  \ref{ex:nonsmooth2}.}
	\label{Fig:conv_nonsmooth2}
\end{figure}
\begin{figure}[htbp]
	\centering
	\begin{tabular}{cccc}
		\includegraphics[width=0.22\textwidth]{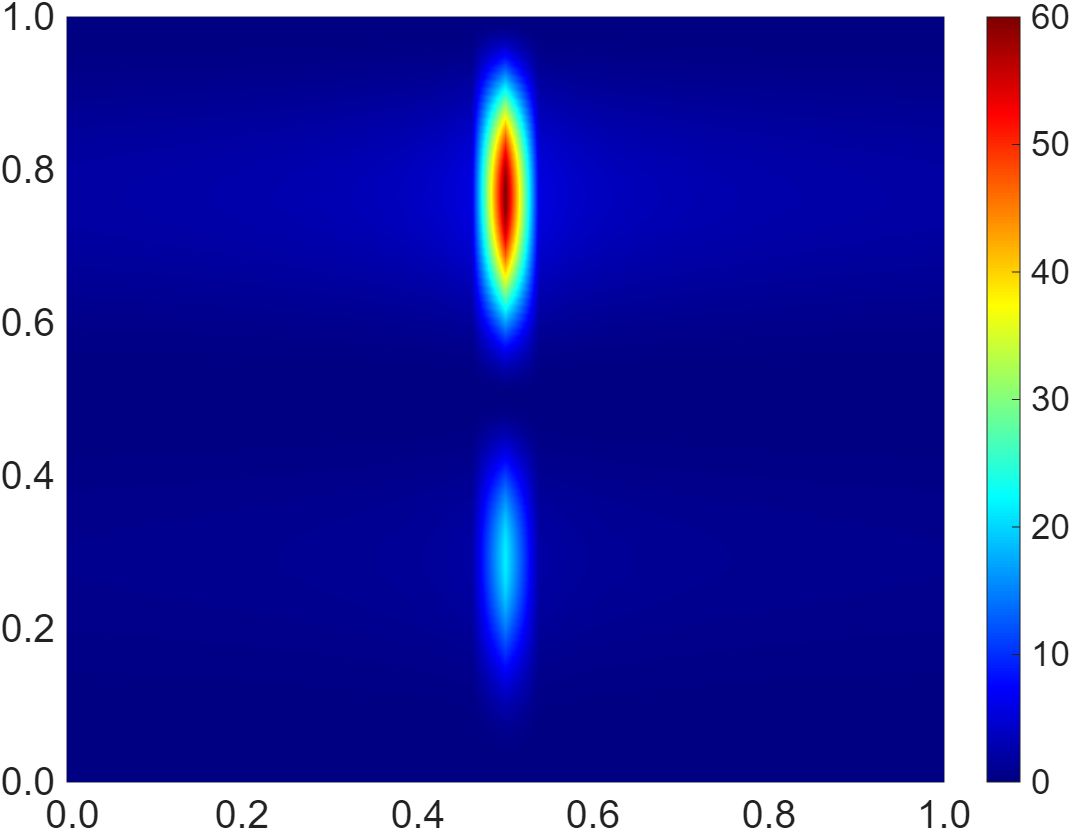}&
		\includegraphics[width=0.22\textwidth]{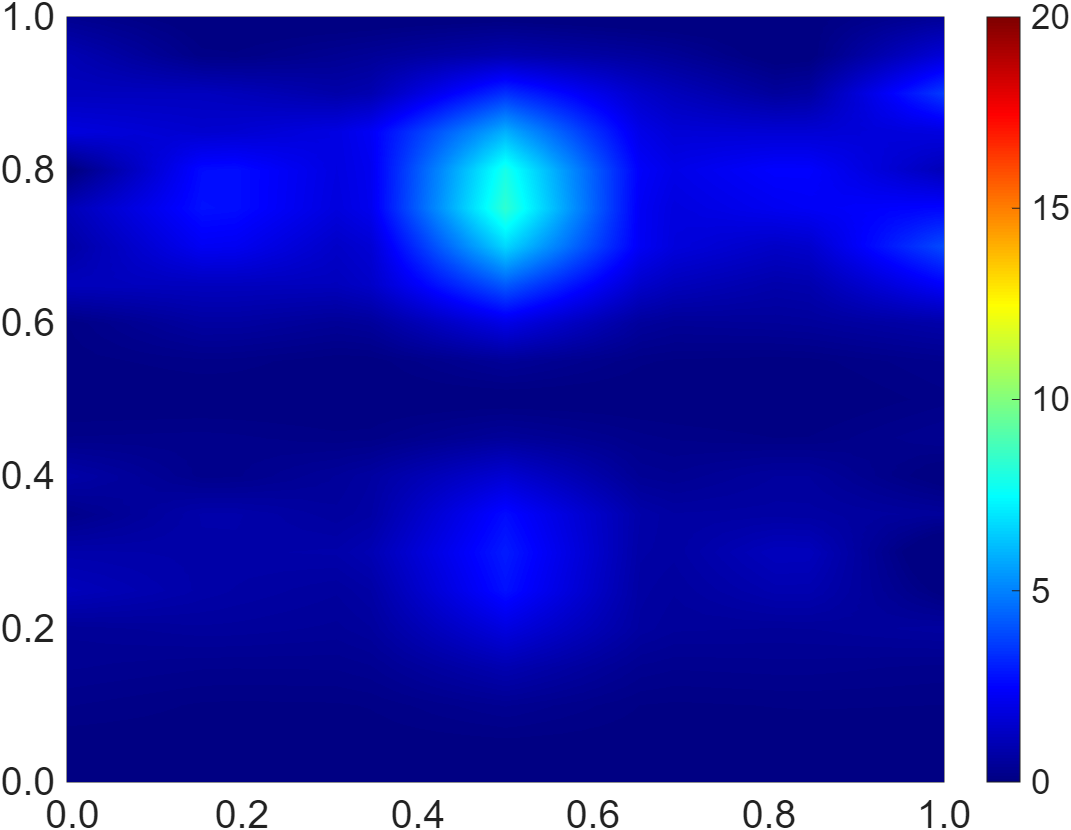}&
		\includegraphics[width=0.22\textwidth]{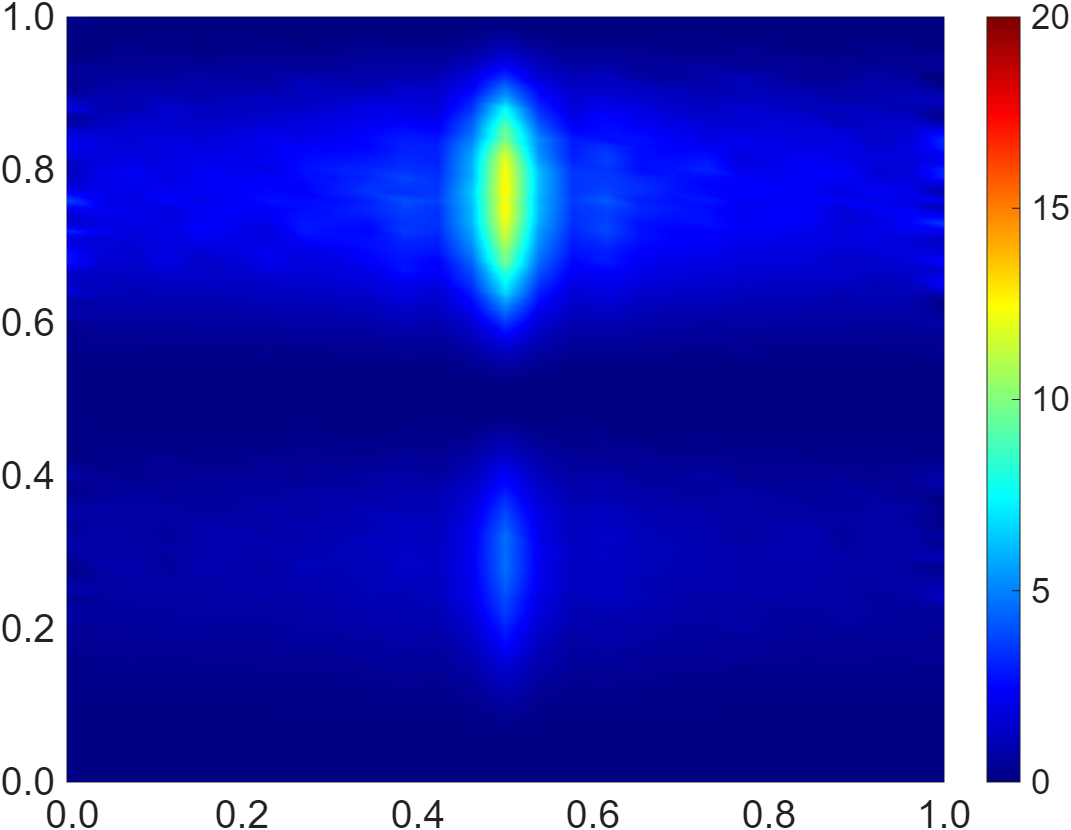}&
		\includegraphics[width=0.22\textwidth]{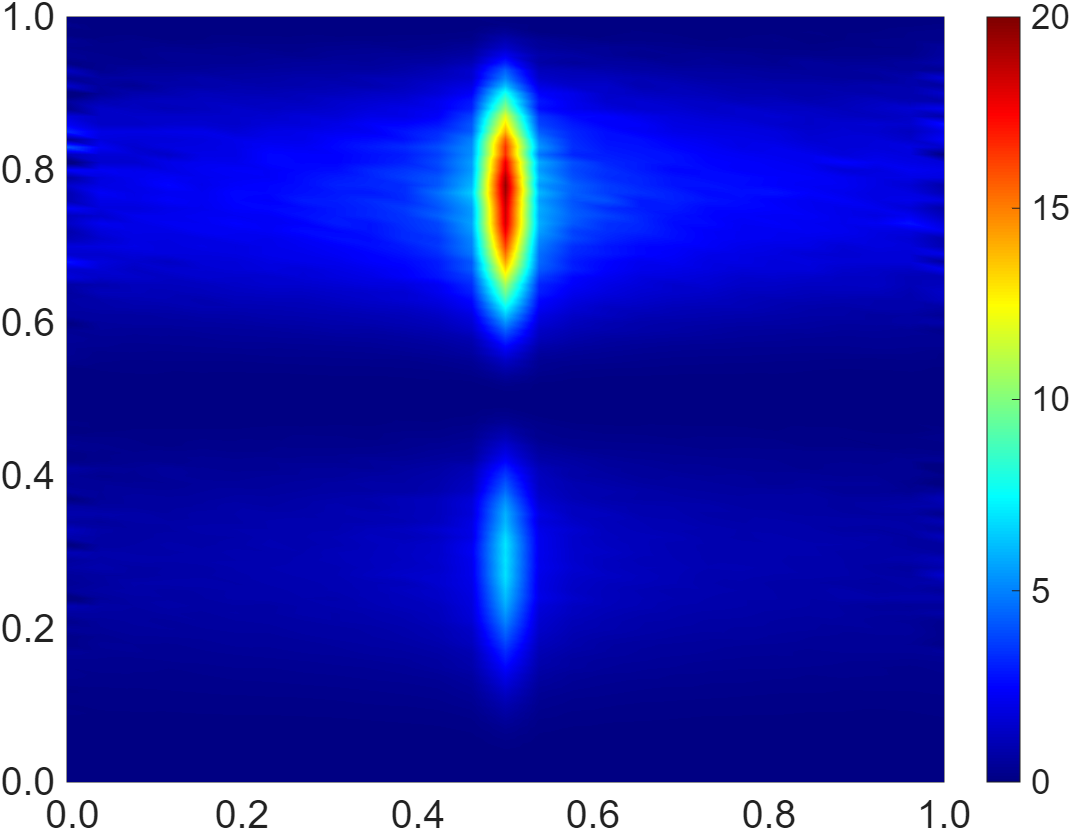}\\
		(a) exact & (b) $\delta=\text{1e-2}$  & (c) $\delta=\text{1e-3}$ & (d) $\delta=\text{1e-4}$  
	\end{tabular}
	\caption{The reconstruction $f^*$  for Example \ref{ex:nonsmooth2}  with $\alpha=0.75$ at different noise levels.}
	\label{Fig:recon_nonsmooth2}
\end{figure}

\section{Conclusion}\label{sec:conclu}
In this paper, we have investigated the reconstruction algorithm and error estimate for the inverse source problem of recovering the space-time varying component of the source on a cylindrical domain from lateral boundary measurement. The additional regularity of the solution along the cylinder's axis enables the construction of a contraction mapping. We have thoroughly analyzed the mapping in a fully discrete setting, using the finite element method in space and convolution quadrature generated by the backward Euler method in time. We proved  the convergence of the discrete reconstruction formula and derived explicit error bounds in terms of the discretization parameters $h$ and $\tau$, which provide  a guideline for choosing these parameters. We presented several numerical experiments that validate the theoretical analysis. 

There are several directions for future research. First, it is of interest to derive \emph{a posteriori} error estimates for the inverse source problem and to develop adaptive numerical algorithms and to provide relevant analysis. Second, the present framework is promising for extension to more challenging nonlinear inverse problems, e.g., the recovery of a space-time dependent potential from boundary measurements, for which establishing well-posed reconstruction schemes and sharp error estimates represents an important topic for further investigation.

\bibliographystyle{abbrv}
\bibliography{ref}

\end{document}